\documentclass{article}
\usepackage[utf8]{inputenc}
\usepackage{amsthm}
\usepackage{amsmath}
\usepackage{amssymb,textgreek}
\usepackage{enumerate}
\usepackage{hyperref}
\usepackage{color,xcolor}
\usepackage{natbib}
\usepackage{graphicx}
\setcitestyle{authoryear,open={(},close={)}}

\usepackage{tikz}
\usetikzlibrary{automata}
\usetikzlibrary{arrows,shapes,calc}
\usetikzlibrary{positioning}
\usepackage{geometry}
\usepackage{authblk} %%% to add authors affiliations

\usepackage[ruled,vlined]{algorithm2e}
\usepackage{stmaryrd}

\def\bZ{{\mathbb Z}}
\def\bN{{\mathbb N}}

\def\bP{{\mathbb P}}
\def\bE{{\mathbb E}}

\def\cF{\mathcal F}
\def\bF{\mathbb F}

\def\cC{\mathcal C}

 \def\llb{\llbracket}
\def\rrb{\rrbracket}

\def\bfX{\boldsymbol{X}}

\newcommand\1{\leavevmode\hbox{\rm \small1\kern-0.35em\normalsize1}}

\def\mCov{\text{Cov}}

\def\bR{\mathbb{R}}
\def\N{\mathbb N}

\def\Supp{\text{supp}({\bf P})}

\newtheorem{theorem}{Theorem}
\newtheorem{corollary}{Corollary}
\newtheorem{proposition}{Proposition}
\newtheorem{lemma}{Lemma}
\newtheorem{remark}{Remark}

\newtheorem{assumption}{Assumption}

\title{Sparse Markov Models for High-dimensional Inference}

\author[1]{G. Ost \thanks{guilhermeost@im.ufrj.br}}
\author[2]{D. Y. Takahashi \thanks{takahashiyd@gmail.com}}

\affil[1]{Universidade Federal do Rio de Janeiro, Rio de Janeiro, Brazil}
\affil[2]{Universidade Federal do Rio Grande do Norte, Natal, Brazil}

\date{\today}

\begin{document}

\maketitle

\tableofcontents

\abstract
Finite order Markov models are theoretically well-studied models for dependent discrete data.  Despite their generality, application in empirical work when the order is large is rare.  Practitioners avoid using higher order Markov models because (1) the number of parameters grow exponentially with the order and (2) the interpretation is often difficult. Mixture of transition distribution models (MTD)  were introduced to overcome both limitations. MTD represent higher order Markov models as a convex mixture of single step Markov chains, reducing the number of parameters and increasing the interpretability. Nevertheless, in practice, estimation of MTD models with large orders are still limited because of curse of dimensionality and high algorithm complexity. Here, we prove that if only few lags are relevant we can consistently and efficiently recover the lags and estimate the transition probabilities of high-dimensional MTD models.
The key innovation is a recursive procedure for the selection of the relevant lags of the model.  Our results are  based on (1) a new structural result of the MTD and (2) an improved martingale concentration inequality. We illustrate our method using simulations and a weather data.

\medskip 


\section{Introduction}
 From daily number of COVID-19 cases to the activity of neurons in the brain, discrete time series are ubiquitous in our life. 
 A natural way to model these time series is by describing how the present events depend on the past events, \emph{i.e.}, characterizing the transition probabilities. 
Therefore, finite-order Markov chains  - models specified by transition probabilities that depend only on a limited portion of the past - are an obvious choice to model time series with discrete values. The length of the portion of the relevant past defines the order of the Markov chain. At first glance, estimating transition probabilities of a Markov chain from the data is straightforward. Given a sample $X_{1:n}:=(X_1, X_2, \ldots, X_n)$ of a stationary $d$-th order Markov chain on a discrete alphabet $A$, the empirical transition probabilities are computed, for all past $x_{-d:-1}:=(x_{-d},\ldots,x_{-1})\in A^{\{-d,\ldots, -1\}}$ and symbol $a\in A$, as 
$$\hat{p}_n(a|x_{-d:-1}):= \frac{N_n(x_{-d:-1},a)}{\sum_{b\in A}N_n(x_{-d:-1},b)},$$
where $N_n(x_{-d:-1},a)$ denotes the number of occurrences of the past $x_{-d:-1}$ followed by the symbol $a$ in the sample $X_{1:n}$.

Nevertheless, some difficulties become apparent. First, for a Markov chain of order $d$, we have to estimate $|A|^{d}(|A|-1)$ transition probabilities ({\it parameters}), making the uniform control of estimation errors much harder when the order $d$ increases. 
One solution to avoid the exponential increase in the number of parameters is to consider more parsimonious classes of models. One such popular class of models is the {\it variable length Markov chains} (VLMC), in which $$\bP(X_t = a|X_{t-d:t-1}=x_{-d:-1}) = \bP(X_t = a|X_{t-\ell:t-1}=x_{-\ell:-1}),$$
where $\ell$ is a function of the past $x_{-d:-1}$ \citep{rissanen1983universal, buhlmann1999variable, galves2012context}. The relevant portion $x_{-\ell:-1}$ of the past $x_{-d:-1}$ is called a {\it context}. The key feature of VLMC is that all transition probabilities with the same context have the same values. Therefore, denoting $\tau$ the set of all contexts, the number of transition probabilities that needs to be estimated reduces to $|\tau|(|A|-1)$.  
Another class of models that is even more parsimonious is the {\it Minimal Markov Models} - also known as {\it Sparse Markov Chains} (SMC) \citep{garcia2011minimal,jaaskinen2014sparse}. In SMC, we say that the pasts $x_{-d:-1}$ and $y_{-d:-1}$ are related if for all symbols $a \in A$,
$$\bP(X_t = a|X_{t-d:t-1}=x_{-d:-1}) = \bP(X_t = a|X_{t-d:t-1}=y_{-d:-1}).$$ This relation generates the equivalent classes $\mathcal{C}_1, \ldots,\mathcal{C}_K$ that partition $A^{\{-d,\ldots, -1\}}$. Now, the number of transition probabilities that needs to be estimated is $K(|A|-1)$. %{\bf Whenever the number of classes $K$ is smaller than $|A|^d$, the corresponding SMC model is also more sparse than a Markov chain model of order d}. 
Both VLMC and SMC have the advantage of better balancing the bias and variance tradeoff. Nevertheless, in any of the two models we still need to estimate the transition probability using $\hat{p}_n(a|x_{-d:-1})$, either because we need to estimate the largest context (for VLMC) or because we need to first calculate the transition probabilities to establish the partitions (for SMC). This creates a second difficulty.
For the estimator $\hat{p}_n(a|x_{-d:-1})$ to have any meaning, we have to observe the sequence $x_{-d:-1}$ in the sample $X_{1:n}$ at least once. By ergodicity, the number of times that we will observe the sequence $x_{-d:-1}$ is roughly $n\bP(X_{1:d}= x_{-d:-1})$. It is straightforward to show that, if the transition probabilities are bounded below from zero, there exists a constant $c > 0$ such that $\bP(X_{1:d} = x_{-d:-1}) < e^{-cd}$. Therefore, in general, it is hopeless to have a reasonable estimator $\hat{p}_n(a|x_{-d:-1})$ if $d > (1+\varepsilon)\log n/c$, for some positive value $\varepsilon$. This imposes a fundamental limit to the size of the past that can be included in the description of the time series.   

Notwithstanding, Markov chains with small orders are not always consistent with the known workings of natural phenomena where the transition probabilities might depend on remote pasts. For example, in predicting whether today will be a warm or cold day,  we might need know remote past events like the corresponding weather approximately a year ago   \citep{kiraly2006correlation, yuan2013long}. Physiological phenomena in humans with cycles of different lengths might result from dependence on events that happens at vastly different temporal scales \citep{gilden19951, chen1997long,buzsaki2004neuronal}. Importantly, not all portions of the past are necessarily relevant. These observations motivate us to explore sparser representations of the dependence on the past events. The {\it mixture of transition distribution} model (MTD) is a subclass of finite order Markov chains that can be used to obtain such sparse representation. Similar to VLMC and SMC, MTD was initially introduced to overcome the problem of exponential increase in the number of transition probabilities for Markov chains \citep{Raftery_85, Berchtold_Raftery:02}. MTD represents higher order Markov models as a convex mixture of single-step Markov chains, where each single-step Markov chain depends on a single time point in the past. If a MTD model is a mixture of only few single step Markov chains, we naturally obtain a class of sparse Markov chains that depends only on a small portion of the past events. Nevertheless, available methods to consistently estimate the transition probabilities of MTD still need to consider all the past events up to the MTD order \citep{Berchtold_Raftery:02}, which might include irrelevant portions of the past. In practice, this fact still restricts the MTD order to $d=\mathcal{O} (\log n)$. 

In this work, we show a simple method that consistently recovers the relevant part of the past even when the order $d$ of the MTD model is proportional to the sample size $n$ (\textit{i.e}, $d=\mathcal{O}(n)$) if the size of the relevant past is $\mathcal{O}(\log n)$. As a consequence we prove that we can consistently estimate the transition probabilities for high dimensional MTD under sparsity constraint.  Our estimator is computationally efficient, consisting of a forward stepwise procedure that finds the candidates for relevant past portions and a cutting procedure to eliminate the irrelevant portions. The theoretical guarantees of our estimator are based on a novel structural result for MTD and an improved martingale concentration inequality. Both results might have an interest in its own. Moreover, when the alphabet is binary, we show that the estimator can be further improved. We prove that in several cases, our estimator is minimax rate optimal. 

Finally, using simulated data, we show that our method's performance is in general superior to a best subset selection method, where the lags with $k$ largest weights are selected after estimating the model with a classical MTD estimation method \citep{berchtold2001estimation}, and similar to the performance of Conditional Tensor Factorization (CTF) based on higher order Markov chain estimation when the order is moderate \citep{sarkar2016bayesian}. We also applied our method on a weather data to model a binary sequence indicating days with and without rains. Our method successfully capture long-range dependencies (\emph{e.g.} annual cycle) that were not detected neither by VLMC algorithm with BIC order selection \citep{csiszar2006context} nor by the CTF based higher order Markov chain estimation. 
Let us mention that recently new Bayesian approaches for higher order VLMC and MTD selection were introduced in \citep{kontoyiannis2020bayesian, heiner2021estimation}, where a posteriori most likely model estimation is considered.
These works provide interesting alternative approaches for modeling higher order Markovian dependence in a Bayesian setting.

We organized the paper as follows. In Section \ref{sec:notations} we introduce the main notations, definitions, and assumptions that we will use throughout the paper. In Section \ref{sec:SLS} we introduce the algorithms to select the relevant part of the past. In Section \ref{Sec:ETP} we provide an estimate of the rate of convergence of the estimator for the transition probabilities.
In Section \ref{sec:minimax_rate}, we show our estimator achieves the optimal minimax rate.
In Section \ref{sec:sim}  we illustrate the performance of the proposed estimators through simulations and an application on a weather data.

\section{Notation, model definition and preliminary remarks} \label{sec:notations}

\subsection{General notation}

We denote $\bZ=\{\ldots -1,0,1,\ldots\}$ and $\bZ_+=\{1,2\ldots\}$  the set of integers and positive integers respectively.
%and $\bN=\{0,1,\ldots\}$ the set of non-negative integers.
For $s,t\in\bZ$ with $s\leq t$, we write $\llb s,t \rrb $ to denote the discrete interval $\bZ\cap [s,\ldots,t]$. 
%For any non-empty subset $S\subseteq \bN$, we denote 
%$\sup(S)$ the supremum of the set S.
Throughout the article $A$ denotes a finite subset of $\bR$, called {\it alphabet}. The elements of $A$ will be denoted by the first letters of the alphabet $a$, $b$ and $c$. Hereafter, we denote $\|A\|_{\infty}=\max_{a\in A}|a|$ and  $Diam(A)=\max_{a,b\in A}|a-b|.$
For each $S\subset \bZ$, the set $A^{S}$ denotes the set of all $A$-valued strings $x_S=(x_j)_{j\in S}$ indexed by the set $S$.
To alleviate the notation, if $S=\llb s,t \rrb$ for $s,t\in\bZ$ with $s\leq t$, we write $x_{s:t}$ instead of $x_{\llb s, t\rrb}$. 
For any non-empty subsets $U\subset S\subseteq \bZ$ and any string $x_{S}\in A^{S}$, we denote $x_{(S\setminus U)}\in A^{(S\setminus U)}$ the string obtained from $x_{S}$ by removing the string $x_{U}\in A^{U}$. For all $t\in\bZ$ and $S\subset\bZ$, we will write in some cases $t+S$ do denote the set $\{t+s:s\in S\}$.

The set of all finite $A$-valued strings is denoted by
$$
\mathcal{A}=\bigcup_{S\subset\bZ: S\ \text{finite}}A^{S}.
$$
For all $x\in \mathcal{A}$, we denote $S_x\subset \bZ$ the set 
indexing the string $x$, i.e., such that $x\in A^{S_x}$.  

Given two probability measures $\mu$ and $\nu$ on $A$, we denote $d_{TV}(\mu,\nu)$ the total variation distance between $\mu$ and $\nu$, defined as
$$
d_{TV}(\mu,\nu)=\frac{1}{2}\sum_{a\in A}|\mu(a)-\nu(a)|.
$$
For $q\in\bZ_+$, the $\|\cdot\|_{q}$-norm of vector $v\in \bR^L$ is defined as
$$
\|v\|_{q}=\left(\sum_{\ell=1}^{L}|v_{\ell}|^q\right)^{1/q}.
$$
The dimension $L\in\bZ_+$ will be implicit in most cases.

For two probability distributions $P$ and $Q$ on $A^{\llb 1, k\rrb}$ where $P$ is absolutely continuous with respect to $Q$, we denote $KL(P||Q)$ the {\it Kullback-Leibler} divergence between $P$ and $Q$, given by
$$
KL(P||Q)=\sum_{x_{1:k}\in A^{\llb 1, k\rrb}}P(x_{1:k})\log\left(\frac{P(x_{1:k})}{Q(x_{1:k})}\right).
$$

\subsection{Markov models}
Let $\bfX=(X_t)_{t\in\bZ}$ be a discrete time stochastic chain, defined in a suitable probability space $(\Omega,\cF,\bP)$, taking values in an alphabet $A$.
For a $d\in\bZ_+$, we say that ${\bf X}$ is a \emph{Markov chain of order $d$} if for all $k\in \bZ_+$ with $k>d$, $t\in\bZ$ and $x_{t-k:t}\in A^{\llbracket t-k, t\rrbracket}$ with $\bP(X_{t-k:t-1}=x_{t-k:t-1})>0$, we have
    \begin{equation}
    \label{def:Markov_property}
    \bP\left(X_t=x_t|X_{t-k:t-1}=x_{t-k:t-1}\right)=\bP\left(X_t=x_t|X_{t-d:t-1}=
    x_{t-d:t-1}\right).
    \end{equation}
%The smallest number $d$ for which \eqref{def:Markov_property} holds is called the {\it order} of the Markov chain.

We say that a Markov chain is {\it stationary} if $X_{s:t}$ and $X_{s+h:t+h}$ have the same distribution
%is equal to the distribution  $\bP(X_{s:t}=x_{s:t})=\bP(X_{1:t-s+1}=x_{s:t})$ 
for all $t,s,h\in\bZ$.
% and $x_{s:t}\in A^{\llbracket s,t \rrbracket}.$ 
Throughout the article, the distribution of a stationary Markov chain will be denoted by ${\bf P}$. For a finite $S\subset \bZ$ and $x_S\in A^S$, we write ${\bf P}(x_S)$ to denote $\bP(X_S=x_S)$.
The {\it support} of a stationary Markov model is the set $\text{supp}({\bf P})=\{x\in \mathcal{A}: {\bf P}(x_{S_x})>0\}.$ 

For stationary Markov chains, the conditional probabilities in \eqref{def:Markov_property} do not depend on the time index $t$. Therefore, for a stationary Markov chain of order d, for any $a\in A$, $x_S\in \Supp$ with $S\subseteq \llbracket -d,-1\rrbracket$ and $t\in \bZ$,
we denote 
\[
p(a|x_S)=\bP\left(X_{t}=a|X_{t+S}=x_S\right).
\]

Notice that $p(\cdot|x_S)$ is a probability measure on $A$, for each fixed past $x_S\in \Supp$. The set $\left\{p(\cdot|x_{-d:-1}):x_{-d:-1}\in  \Supp \right\}$ is called
the family of {\it transition probabilities} of the chain. 
In this article, we consider only stationary Markov chains. 

For a Markov chain of order $d$, the {\it oscillation} $\delta_j$ for $j\in \llb -d, -1\rrb $ is defined as
\begin{equation*}
%\label{def_2:lag_weights}
\delta_j=\max \{d_{TV}\left(p(\cdot|x_{-d:-1}),p(\cdot|y_{-d:-1})\right): (x_{-d:-1},y_{-d:-1}) \in A^{\llb -d, -1\rrb}, x_{-k}=y_{-k},\ \forall \ k\neq j\}. 
\end{equation*}
The oscillation is useful to measure the influence of a $j$-th past value in the values of the transition probabilities.

\subsection{Mixture transition distribution (MTD) models}
A MTD model of order $d\in \bZ_+$ is a Markov chain of order $d$ for which the associated family  of transition probabilities $\left\{p(\cdot|x_{-d:-1}):x_{-d:-1}\in \Supp\right\}$
admits the following representation:
\begin{equation}
\label{def:transition_probabilities}
p(a|x_{-d:-1})=\lambda_0p_0(a)+\sum_{j=-d}^{-1}\lambda_jp_j(a|x_{j}), \ a\in A, 
\end{equation}
with $\lambda_0, \lambda_{-1}, \ldots,\lambda_{-d}\in [0,1]$ satisfying $\sum_{j=-d}^{0}\lambda_j=1$ and $p_0(\cdot)$ and $p_j(\cdot|b)$,$j\in \llb -d, -1\rrb $ and $b\in A$, being probability measures on $A$. 

Following \citep{Berchtold_Raftery:02}, we call  the index $j\in \llb -d, 0\rrb $ of the weight $\lambda_j$ in \eqref{def:transition_probabilities} the {\it $j$-th lag} of the model. The representation in  \eqref{def:transition_probabilities} has the following probabilistic interpretation.  To sample a symbol from $p(\cdot|x_{-d:-1})$, we first choose a lag in $\llb -d, 0\rrb$  randomly, being $\lambda_j$ the probability of choosing the lag $j$. Once the lag has been chosen, say lag $j$, we then sample a symbol from the probability measure $p_j(\cdot|x_{j})$ which depends on the past $x_{-d:-1}$ only through the symbol $x_{j}$. Notice that a symbol is sampled independently from the past $x_{-d:-1}$, whenever the lag $0$ is chosen.  

For later use, let us define the conditional average at lag $j$ as
\begin{equation}
\label{def:Conditional_average_pj_given_b}
m_j(b)=\sum_{a\in A}ap_j(a|b),    
\end{equation}
for each $j\in \llb -d, -1\rrb$ and $b\in A$.

For a MTD model of order $d$, we have that  the oscillation $\delta_j$ of the lag $j\in \llb -d, -1\rrb $ can be written as, 
\begin{equation}
\label{def:index_weights}
\delta_j=\lambda_j \max_{b,c\in A}d_{TV}(p_j(\cdot|b),p_j(\cdot|c)).
\end{equation}
Notice that in this case $\delta_j=0$ if and only if either $\lambda_j=0$ or $d_{TV}(p_j(\cdot|b),p_j(\cdot|c))=0$ for all $b,c\in A$.

%Moreover, one can check from \eqref{def:transition_probabilities} that 
%\begin{align}
%\label{def:TV_transition_prob_compatible_pasts}
%d_{TV}\left(p(\cdot|x_{-d:-1}),p(\cdot|y_{-d:-1})\right)=\lambda_jd_{TV}(p_j(\cdot|x_j),p_j(\cdot|y_j)),
%\end{align}
%for any lag $j\in \llb -d, -1\rrb$ and pair of pasts $(x_{-d:-1},y_{-d:-1})\in \cC_j$, where $\cC_j$ denotes the set %of all $(\llb -d, -1\rrb\setminus\{j\})$-{\it compatible} pasts, defined as
%$$
%\cC_j=\left\{(x_{-d:-1},y_{-d:-1})\in \Supp\times\Supp: x_{k}=y_{k}, \ k\neq j\right\}.
%$$
%As a consequence, it follows that the oscillation $\delta_j$ of the $j$-th lag can be rewritten as
%\begin{equation*}
%\label{def_2:lag_weights}
%\delta_j=\max_{(x_{-d:-1},y_{-d:-1})\in \cC_j}d_{TV}\left(p(\cdot|x_{-d:-1}),p(\cdot|y_{-d:-1})\right).
%\end{equation*}
%This shows that the definition of the oscillation $\delta_j$ of the $j$-th lag does not depend on the particular %representation of the family of transition probabilities. 
%Moreover, if the oscillation $\delta_j=0$, then the value of transition probability $p(a|x_{-d:-1})$, for any $a\in %A$, does not depend on the value of $x_{j}$ for all $x_{-d:-1}\in\Supp$. 

In the sequel, we say that the lag $j$ is {\it relevant} if $\delta_j>0$, and {\it irrelevant} otherwise. 
We will denote $\Lambda$ the set of all relevant lags, i.e.,  
\begin{equation}
 \Lambda=\{j\in \llb -d, -1\rrb: \delta_j>0\}.   
\end{equation}
The set $\Lambda$ captures the dependence structure of the MTD model.
The size $|\Lambda|$ of the set $\Lambda$ represents the degree of sparsity of the MTD model. The smaller the value of $|\Lambda|$, the sparser the MTD model.

The following quantities will appear in many of our results:
\begin{equation}
\label{def:delta_min_and_}
\delta_{min}=\min_{j\in\Lambda}\delta_j \ \text{and} \ \tilde{\delta}_{min}=\min_{j\in\Lambda}\lambda_j\|m_j\|_{Lip},
\end{equation}
where $\|m_j\|_{Lip}=\max\{|m_j(b)-m_j(c)|/|b-c|:b,c\in A,b\neq c\}$ denotes the Lipschitz norm of the function $m_j$ defined in \eqref{def:Conditional_average_pj_given_b}.
One can check easily that these quantities coincide when the alphabet A is {\it binary} (i.e. $A=\{0,1\}$). For general alphabets, the following inequality holds:
$$
\delta_{min}\geq \|A\|^{-1}_{\infty}\tilde{\delta}_{min}\min_{b,c\in A:b\neq c}|b-c|.
$$

%To conclude this section, let us stress that $p(a|x_{-d:-1})$ depends on the past $x_{-d:-1}$ only through $x_{\Lambda}$ in the sense that if a subset $S\subseteq \llb-d, -1\rrb$ is such that $\Lambda\subseteq S$, then $p(a|x_{-d:-1})=p(a|y_{-d:-1})$ for all $x_{-d:-1},y_{-d:-1}\in\Supp$ satisfying $x_{S}=y_{S}$. In light of these observations, we will often write with slight abuse of notation $p(a|x_{S})$ with $\Lambda\subseteq S\subseteq \llb -d, -1\rrb $ to the denote the transition probability $p(a|x_{-d:-1})$. 
%More generally, we shall also write with abuse of notation $p(a|x_{S})$ with $\Lambda\subseteq S$ even

%\gocom{Acho que precisamos dizer, pra ser bem preciso, que usamos a mesma notacao tambem quando S contem $\Lambda$ mas é maior do que $[-d,-1]$}

\subsection{Statistical lag selection}
Suppose that we are given a sample $X_{1:n}$ of a MTD model of known order $d < n$ and whose set of relevant lags $\Lambda$ is unknown. 
The goal of {\it statistical lag selection} is to estimate the set $\Lambda$ from the sample $X_{1:n}$.
Our particular interest is in the high-dimensional setting in which the parameters $d=d_n$ and $|\Lambda|=|\Lambda_n|$  scale as a function of the sample size $n$.
Let us write  $\hat{\Lambda}_n$ to indicate an estimator of the set of relevant lags $\Lambda$ computed from the sample $X_{1:n}$. We say that the estimator $\hat{\Lambda}_n$ is {\it consistent} if
$$
\bP(\hat{\Lambda}_n\neq \Lambda)\to 0 \ \text{as} \ n\to\infty.
$$
With respect to statistical lag selection, our goal is to exhibit sufficient conditions for each proposed estimator guaranteeing its consistency.

\subsection{Empirical transition probabilities}

%Once we have estimated the set of relevant lags $\Lambda$, we consider the problem of building confidence intervals for the transition probabilities
%of the underlying MTD model generanting the sample. In this problem, we are again given a sample $X_{1:n}$ of a MTD model of known order $d<n$ with an unknown set of relevant lags $\Lambda$. 
%The value of $n$ is called the sample size.  
%Let $n>d$
%Here, we use the first $m$ symbols $X_{1:m}$ of the sample to compute an estimator $\hat{\Lambda}_{m}$ of $\Lambda$. Given the estimated set $\hat{\Lambda}_{m}$,
%we then want to build confidence intervals for the transition probabilities of underlying MTD model with the $n-m$ remaining symbols $X_{m+1:n}$ of the sample.
%We will refer to these confidence intervals as {\it post lag selection} (PLS) confidence intervals.

Let $n,m$ and $d$ be positive integers such that $n-m> d$. We denote for each $a\in A$ and $x_S\in \mathcal{A}$ with $S\subseteq \llb -d, -1\rrb$ non-empty,
\[
N_{m,n}(x_S,a)=\sum_{t=m+d+1}^{n}1\{X_{t+j}=x_{j},j\in S,X_{t}=a\}.
\]
%in the case  $m-d>\ell_{S_x}$, and $N_{n,m}(x,a)=0$ otherwise. 
The random variable $N_{m,n}(x_S,a)$ indicates the number of occurrences of
the string $x_S$ ``followed'' by the symbol $a$,
in the last $n-m$ symbols  $X_{m+1:n}$ of the sample $X_{1:n}$. 
We also define $\bar{N}_{m,n}(x_S)=\sum_{a\in A}N_{m,n}(x_S,a)$. 
With this notation, the empirical transition probabilities computed from
the last $n-m$ symbols
$X_{m+1:n}$ of the sample $X_{1:n}$ are defined as,
\begin{equation}
\label{def:emp_trans_proba}
\hat{p}_{m,n}(a|x_S)=\begin{cases}
\frac{N_{m,n}(x_S,a)}{\bar{N}_{m,n}(x_S)}, \ \text{if} \ \bar{N}_{m,n}(x_S)>0\\
\frac{1}{|A|}, \ \text{otherwise}
\end{cases}.
\end{equation}
When the countings are made over the whole sample $X_{1:n}$, we denote $N_{n}(x_S,a)$ and $\bar{N}_{n}(x_S)$ the corresponding counting random variables, and $\hat{p}_{n}(a|x_S)$ the corresponding empirical transition probabilities. 

In the next sections, the estimators for the set of relevant lags we propose in this paper rely on these empirical transition probabilities.
If $\hat{\Lambda}_m$ denotes an estimator for the set of relevant lags $\Lambda$ computed from $X_{1:m}$, we expect that under some assumptions (guaranteeing in particular the consistency of $\hat{\Lambda}_m$) the empirical transition probability $\hat{p}_{m,n}(a|x_{\hat{\Lambda}_m})$ converges (in probability) to $p(a|x_{\Lambda})$ as $\min\{n,m\}\to\infty$, for any $x_{-d:-1}\in\Supp$. To understand the convergence for the transition probabilities of high order Markov chains is crucial in our analysis.
%the error probability $\bP(\hat{\Lambda}_n\neq \Lambda)$ of these estimators.

\subsection{Assumptions}
\label{sec:assump}
We collect here the main assumptions used in the article. 

\begin{assumption}
\label{Ass:MTD_stationary}
The MTD model has {\it full support}, that is, $\Supp=\mathcal{A}$.
\end{assumption}

 In other words, Assumption \ref{Ass:MTD_stationary} means that $\bP(X_S=x_S)={\bf P}(x_{S})>0$ for any string $x_S\in A^S$ with $S\subset\bZ$ finite. 
This means that the marginal distributions of the distribution generating the data are strictly positive. Such a condition is usually assumed in the problem of estimating the graph structure underlying graphical models (see for instance Chapter 11 of \citep{Wainwright2019}). Notice that this assumption implies, in particular, that 

\begin{equation}
\label{def:pstar}
p_{min}=\min\{p(a|x_{\Lambda}):a\in A, \ x_{\Lambda}\in A^{\Lambda}\}>0,    
\end{equation}

where $p(\cdot|x_{\Lambda})$ are the transition probabilities of MTD generating the data.

%In what follows, the weights and the transition probabilities associated to the lags $j$ of the true MTD model are denoted by $\lambda^*_j$ and $p^*_j(\cdot|\cdot)$  respectively.
%Since the set $A$ is finite, Assumption \ref{Ass:MTD_stationary} holds under mild assumptions on the MTD model, such as irreducibility. 

\begin{assumption}
\label{Ass:Concentration_inequalities}
The quantity $\Delta:=1-\sum_{j\in\Lambda}\delta_j>0$,
where $\delta_j$ is given by \eqref{def:index_weights}.
\end{assumption}

We have that $\lambda_0>0$ is a sufficient condition  to Assumption \ref{Ass:Concentration_inequalities} to hold. To check this, notice that
$$
\sum_{j\in\Lambda}\delta_j=\sum_{j\in\Lambda}\lambda_j\max_{b,c\in A}d_{TV}(p_j(\cdot|b),p_j(\cdot|c))\leq \sum_{j\in\Lambda}\lambda_j=1-\lambda_0,
$$
where we have used that $d_{TV}(p_j(\cdot|b),p_j(\cdot|c))\leq 1$
for all $b,c\in A$ and $j\in\Lambda.$ Hence, it follows that $\Delta>0$ whenever $\lambda_0>0$.

Assumptions \ref{Ass:MTD_stationary} and \ref{Ass:Concentration_inequalities} are used to obtain concentration inequalities for the counting random variables $N_{m,n}(x_S,a)$ and $\bar{N}_{m,n}(x_S)$ appearing in the definition of the empirical transition probabilities $\hat{p}_{m,n}(a|x).$

%Since the set $A$ is finite, we then know that under Assumption \ref{ass:non_null} there is only one invariant distribution $\pi$ for the MDT model. The invariant distribution is such that $\pi_S(x_{S})>0$ for any  $x_{S}\in A^{S}$ and non-empty $S\subseteq \llb 1, d\rrb$, where $\pi_S$ denotes the marginal of $\pi$ on $A^{S}$.

The next assumption is as follows.
\begin{assumption}
\label{Ass:non_determinist_conditional_averages}
For each $j\in\Lambda$, there exists $b^{\star},c^{\star}\in A$ such that $m_j(b^{\star})\neq m_j(c^{\star})$, where $m_j(\cdot)$ is defined in \eqref{def:Conditional_average_pj_given_b}.
\end{assumption}
Notice that if $A=\{0,1\}$, then 
$m_j(1)-m_j(0)=p_j(1|1)-p_j(1|0)$,
so that Assumption \ref{Ass:non_determinist_conditional_averages} holds whenever $d_{TV}(p_j(\cdot|1),p_j(\cdot|0))=|p_j(1|1)-p_j(1|0)|>0$ for each $j\in\Lambda$. In this case this is always true by the definition of the set $\Lambda$.
As we will see in Section \ref{sec:SLS}, the condition is crucial to prove a  structural result about MTD models, presented in Proposition \ref{Prop:structure_result}.

In what follows, $\bP_{x_{S}}(X_j\in \cdot|X_{k}=b)$ denotes the conditional distribution of $X_j$ given $X_{S}=x_{S}$ and $X_{k}=b$. 
We use the convention that, for $S= \emptyset$, these conditional probabilities  correspond to the unconditional ones.
Moreover, for any function $f:A\to\bR$, we write $\bE_{x_S}(f(X_j)|X_{k}=b)$ to denote the expectation of $f(X_j)$ with respect to $\bP_{x_{S}}(X_j\in \cdot|X_{k}=b)$.  

The next two assumptions are the following. 

\begin{assumption}[Inward weak dependence condition]
\label{ass:lower_bound_kappa}
There exists $\Gamma_1\in (0,1]$ such that the following condition holds: for all $S\subseteq \llb -d, -1\rrb$ such that $\Lambda\not\subseteq S$, $k\in \Lambda\setminus S$ and $b,c\in A$ with $b\neq c$ satisfying $|m_k(b)-m_k(c)|>0,$
\begin{equation}
\label{def:lower_bound_kappa}
\max_{x_S\in A^S}\sum_{j\in \Lambda\setminus S\cup\{k\}}\frac{\lambda_j\left|\bE_{x_S}(m_j(X_{j})|X_{k}=b)-\bE_{x_S}(m_j(X_{j})|X_{k}=c)\right|}{\lambda_k |m_k(b)-m_k(c)|}\leq (1-\Gamma_1).   
\end{equation}
\end{assumption}
\begin{assumption}[Outward weak dependence condition]
\label{ass:incoherence_condition_binary}
The alphabet is binary, i.e. $A=\{0,1\}$. Moreover, there exists $\Gamma_2\in (0,1]$ such that the following condition holds: for all $S\subseteq \llb -d, -1\rrb$ such that $S\subset\Lambda$ and $k\notin \Lambda$, 
\begin{equation}
\label{def:incoherence_condition_binary}
\sum_{j\in \Lambda\setminus S}\max_{x_S\in\{0,1\}^S}\left|\bP_{x_S}(X_{k}=1|X_{j}=1)-\bP_{x_S}(X_{k}=1|X_{j}=0)\right|\leq \Gamma_2.   
\end{equation}
\end{assumption}
Both Assumptions 4 and 5 are conditions of weak dependence. In words, Assumption 4 says that no relevant lag $j$ can be completely determined by any subset $S$ containing only relevant lags or any other relevant lag $k$ when combined with some irrelevant lags. Similarly, Assumption 5 says that irrelevant lags cannot be completely determined by some subset of relevant lags.  These two assumptions will be only necessary to obtain a computationally very efficient algorithm.

%The Assumption \ref{ass:incoherence_condition_binary} is similar to what is known in the literature as {\it incoherence condition}. It says that....
%This assumption is fundamental to prove another important structural result about MTD models, presented in Theorem \ref{thm:FSC_consistency_binary}.

\section{Statistical lag selection}
\label{sec:SLS}

In this section, we address the problem of statistical lag selection for the MTD models. We will first introduce a statistical procedure called PCP estimator that is general and works well if there is a known small set $S$ such that $\Lambda \subseteq S$. When such set $S$ is not available, we will have to consider an alternative procedure called FSC estimator, which will be introduced later.

%We suppose that the order $d^{\star}$ (also unknown) of the underlying MTD model is such that $d^{\star}\leq d$, where the value of upper bound $d<n$ is known. 
%Notice that this is an assumption on the unknown set of relevant lags $\Lambda$ since $\ell_{\Lambda}=d^{\star}$. 

\subsection{Estimator based on pairwise comparisons}
\label{Sec:PCP_estimator}

Throughout this section we suppose 
that there is a known set $S\subseteq \llb -d, -1 \rrb $ such that $\Lambda\subseteq S$.
Note that this is always satisfied  in the worse case scenario in which the set $S$ is the whole set $\llb -d, -1\rrb$. In some cases, however, we may have a prior knowledge on the set $\Lambda$ and we can use this information to restrict our analysis to the lags in a known set $S$ of size (possibly much) smaller than $d$.

The estimator discussed in this section is based on pairwise comparisons of empirical transition probabilities corresponding to compatible pasts. For this reason,  we  call it \texttt{PCP} estimator. The estimator is based on the following observation.
For any $j\in S$, we say that the pasts $x_S,y_S\in A^{S}$ are $(S\setminus \{j\})$-{\it compatible}, if $y_{S\setminus\{j\}}=x_{S\setminus\{j\}}$. We have that if $j\in\Lambda$, then there exist a pair of $(S\setminus\{j\})$-compatible pasts $x_S,y_S\in A^{S}$ such that total variation distance between $p(\cdot|x_{S})$ and $p(\cdot|y_{S})$ is strictly positive. On the other hand, if $j\in S\setminus\Lambda$, then the total variation distance between $p(\cdot|x_{S})$ and $p(\cdot|y_{S})$ is 0 for all $(S\setminus\{j\})$-compatible pasts $x_S,y_S\in A^{S}$.

 These remarks suggests to estimate $\Lambda$ by the subset of all lags $j\in S$ for which 
the total variation distance between $\hat{p}_n(\cdot|x_S)$ and $\hat{p}_n(\cdot|y_S)$ is larger than a suitable positive threshold, for some
pair of $(S\setminus\{j\})$-compatible pasts $x_S$ and $y_S$.
  An uniform threshold over all possible realizations usually gives suboptimal results by either underestimating or overestimating for some configurations. The threshold we use here is adapted to each realization of the MTD, relying on improved martingale concentration inequalities that are of independent interest (see Appendix \ref{Sec:Mart_Conc_Ineq}). 

Fix $\varepsilon>0,$ $\alpha>0$ and 
$\mu\in (0,3)$ such that $\mu>\psi(\mu):=e^{\mu}-\mu-1$. For each $x_S,y_S\in A^{S}$, consider the random threshold $t_{n}(x_S,y_S)$ defined as, 
\begin{equation}
\label{Def:threshold_for_PCP}
t_{n}(x_S,y_S)=s_n(x_S)+s_n(y_S),  
\end{equation}
where $s_n(x_S)$ is given by
\begin{equation}
\label{def:individual_threshold}
s_{n}(x_S)=\sqrt{\frac{\alpha(1+\varepsilon)}{2\bar{N}_{n}(x_S)}}\sum_{a\in A}\sqrt{\frac{\mu}{\mu-\psi(\mu)}\left(\hat{p}_{n}(a|x_S)+\frac{\alpha}{\bar{N}_{n}(x_S)}\right)} + \frac{\alpha |A|}{6\bar{N}_{n}(x_S)}.
\end{equation}

With this notation, the \texttt{PCP} estimator $\hat{\Lambda}_{1,n}$ is defined as follows.
A lag $j\in S$ belongs to $\hat{\Lambda}_{1,n}$ if and only if there exists a $(S\setminus\{j\})$-compatible pair of pasts $x_S,y_S\in A^{S}$ such that
\begin{equation}
\label{Def:PCP_estimator}
 d_{TV}(\hat{p}_n(\cdot|x_S),\hat{p}_n(\cdot|y_S))\geq  t_{n}(x_S,y_S).
\end{equation}

In the sequel, the set $S\subseteq \llb -d, -1 \rrb $ such that $\Lambda\subseteq S$ and the constants $\varepsilon>0,$ $\alpha>0$ and 
$\mu\in (0,3)$ such that $\mu>\psi(\mu)$ are called parameters of the \texttt{PCP} estimator $\hat{\Lambda}_{1,n}$.

Hereafter, for each $j\in S$ and any $b,c\in A,$ let 
$$\cC_{j}(b,c)=\left\{(x,y)\in A^{S}\times A^{S}: x_{S\setminus\{j\}}=y_{S\setminus\{j\}}, \ x_{j}=b \ \text{and} \ y_{j}=c  \right\},$$ 
and define
\begin{equation}
\label{Def:thresholds_for_PCP_esti}
t_{n,j}(b,c)=\min_{(x_S,y_S)\in \cC_j(b,c)} t_{n}\left(x_S,y_S\right), \  t_{n,j}=\max_{b,c\in A:b\neq c}t_{n,j}(b,c) \ \text{and} \  
\gamma_{n,j}=2t_{n,j}.     
\end{equation}
Finally, consider the following quantity
\begin{equation}
\label{Def:min_min_max_prop_for_PCP_esti}
{\bf P}_S=\min_{j\in\Lambda}\min_{b,c\in A:b\neq c}\max_{(x_S,y_S)\in \cC_j(b,c)}\left({\bf P}(x_S)\wedge {\bf P}(y_S)\right). 
\end{equation}

With these definitions, we have the following result.

\begin{theorem}
\label{thm:consistency_PCP_estimator}
Let $X_{1:n}$  be a sample of MTD model with set of relevant lags $\Lambda$, where $n>d$. If $\hat{\Lambda}_{1,n}$ is the \texttt{PCP} estimator defined in \eqref{Def:PCP_estimator}
with parameters $\mu\in(0,3)$ such that $\mu>\psi(\mu)$, $\alpha>0$, $\varepsilon>0$ and $\Lambda \subseteq S\subseteq \llb -d, -1 \rrb $, we have that
\begin{enumerate}
    \item For each $j\in S\setminus\Lambda$, we have that
    $$
    \bP\left(j\in \hat{\Lambda}_{1,n}\right)\leq 8|A|(n-d)\left\lceil \frac{\log(\mu(n-d)/\alpha+2)}{\log(1+\varepsilon)}\right\rceil e^{-\alpha}.
    $$
    \item For each $j\in \Lambda$, we have that
    \begin{equation*}
    \bP\left(j\notin \hat{\Lambda}_{1,n}, \gamma_{n,j}\leq \delta_j \right) \leq 8|A|\left\lceil \frac{\log(\mu(n-d)/\alpha+2)}{\log(1+\varepsilon)}\right\rceil e^{-\alpha},
     \end{equation*}
    where $\gamma_{n,j}$ and and $\delta_j$ are defined in \eqref{Def:thresholds_for_PCP_esti} and \eqref{def:index_weights} respectively.
    \item Furthermore, if assumptions \ref{Ass:MTD_stationary} and \ref{Ass:Concentration_inequalities} hold, then there exits a constant $C=C(\varepsilon,\mu)>0$ such that for $n$ satisfying
    \begin{equation}
    \label{def:min_sample}
        n\geq d+\frac{C|A|\alpha}{\delta^2_{min}P_S},
\end{equation}
where $\delta_{min}$ and ${\bf P}_S$ are defined in respectively \eqref{def:delta_min_and_} and \eqref{Def:min_min_max_prop_for_PCP_esti}, we have that
\begin{multline}
\label{PCP_estimator_error_probability}
    \bP\left(\hat{\Lambda}_{1,n}\neq \Lambda\right)\leq 8|A|\left((|S|-|\Lambda|)(n-d)+|\Lambda|\right)\left\lceil \frac{\log(\mu(n-d)/\alpha+2)}{\log(1+\varepsilon)}\right\rceil e^{-\alpha}
        \\
        +6|A|(|A|-1)|\Lambda|\exp\left\{-\frac{\Delta^2(n-d)^2{\bf P}^2_S}{8n(|S|+1)^2}\right\}.
\end{multline}
    \end{enumerate}
\end{theorem}

The proof of Theorem \ref{thm:consistency_PCP_estimator} is given in Appendix \ref{Sec:proof_of_thm_cons_PCP_estimator}.

\begin{remark}
\label{rmk:PCP_consistency}
\begin{enumerate}
    \item[(a)] The sum over $j\in S\setminus\Lambda$ of the upper bound provided by Item 1 of Theorem \ref{thm:consistency_PCP_estimator} controls the probability that the \texttt{PCP} estimator $\hat{\Lambda}_{1,n}$ overestimates the set of relevant lags $\Lambda$.
 The sum over $j\in \Lambda$ of the upper bound given in Item 2 of Theorem \ref{thm:consistency_PCP_estimator} is as an upper bound for the probability that the \texttt{PCP} estimator underestimates the subset of relevant lags  $j\in\Lambda$ whose oscillation $\delta_j$ is larger or equal than the ``noise level'' $\gamma_{n,j}$.  
Note that the sum of these upper bounds corresponds to the first term appearing on the right hand side of \eqref{PCP_estimator_error_probability}. 
\item[(b)] The second term on the right hand side of \eqref{PCP_estimator_error_probability} is an upper bound for the probability that there exists some relevant lag $j\in \Lambda$ whose oscillation $\delta_j$ is strictly smaller than the ``noise level'' $\gamma_{n,j}$.
 \item[(c)] (Computation of PCP estimator) 
As we show in Appendix \eqref{computation_PCP_and_FSC_estimator}, the PCP estimator can be implemented with at most $O(|A|^2|S|(n-d))$ computations.
\end{enumerate}
\end{remark}

\begin{remark}
\label{rmk:PCP_exp}
By Assumption \ref{Ass:MTD_stationary}, we have that ${\bf P}_S\geq p_{min}/|A|^{|S|-1}$, where $p_{min}$ is defined in \eqref{def:pstar}. As a consequence, it follows from \eqref{def:min_sample} that if the sample size $n$ is such that
\begin{equation}
\label{def_alternative:min_sample}
n\geq d+\frac{C|A|^{|S|}\alpha}{p_{min}\delta^2_{min}},
\end{equation}
then inequality \eqref{PCP_estimator_error_probability} holds with the second exponential term replaced by
\begin{equation}
\label{PCP_estimator_alternative_error_probability}
\exp\left\{-\frac{\Delta^2p_{min}^2(n-d)^2}{8n(|S|+1)^2|A|^{2(|S|-1)}}\right\}.
\end{equation}

\end{remark}

Combining Theorem \ref{thm:consistency_PCP_estimator} and Remark \ref{rmk:PCP_exp}, one can deduce the following result.

\begin{corollary}
\label{cor:consistency_PCP_estimator}
For each $n$, consider a MTD model with set of relevant lags $\Lambda_{n}$ and transition probabilities $p_n(a|x_{\Lambda_n})$ such that $p_{min,n}\geq p^{\star}_{min}$ and $\Delta_n\geq \Delta^{\star}_{min}$ for some positive constants  $p^{\star}_{min}$ and $\Delta^{\star}_{min}$.
Let $d_n=\beta n$ for some $\beta\in (0,1)$ and suppose that $\Lambda_n\subseteq S_n\subseteq \llb -d_n, -1 \rrb $ with $|S_n|\leq ((1-\gamma)/2)\log_{|A|}(n)$ for some $\gamma\in (0,1)$. Let $X_{1:n}$ be a sample from the MTD specified by $\Lambda_n$ and $p_n(a|x_{\Lambda_n}),$
and denote $\hat{\Lambda}_{1,n}$ the \texttt{PCP} estimator defined in \eqref{Def:PCP_estimator} computed from this sample 
with parameters $\mu_n=\mu\in(0,3)$ such that $\mu>\psi(\mu)$, $\varepsilon_n=\varepsilon>0$, $\alpha_n=(1+\eta)\log(n)$ with $\eta>0$ and $S_n$. Under these assumptions there exists a constant $C=C(\mu,\varepsilon,\beta,p^{\star}_{min},\Delta^{\star}_{min}, \gamma,\eta)>0$ such that if
\begin{equation}
 \label{sub_optimal_condition}   
\delta^2_{min,n}\geq \frac{C\log(n)}{n^{(1+\gamma)/2}},
\end{equation}
then $\bP(\hat{\Lambda}_{1,n}\neq \Lambda_n)\to 0$ as $n\to\infty.$
\end{corollary}

The proof of Corollary is given in Appendix \ref{proof:cor_conist_PCP_Esti}.

\begin{remark}
\label{rmk:cor_PCP_consistency}
\begin{enumerate}
%\item[(a)] Notice that if for instance $p^{\star}_0:=\min_{a\in A}p^{\star}_0(a)>0$ and $\lambda^{\star}_{0,n}\geq \lambda^{\star}_{0,min}$ for all values of n, then we can take $p^{\star}_{min}=p^{\star}_0\lambda^{\star}_{0,min}$
%    and $\Delta^{\star}_{min}=\lambda^{\star}_{0,min}$.
\item[(a)] Under the assumptions of Corollary \ref{cor:consistency_PCP_estimator}, if additionally we have $|\Lambda_n|\leq L$ for all values of $n$ for some positive integer $L$, then one can choose a suitable sequence $\gamma_n\to 1$ as $n\to \infty$ to obtain that $\bP(\hat{\Lambda}_{1,n}\neq \Lambda_n)$ vanishes as $n\to\infty$, as long as
\begin{equation}
\label{optimal_condition}
\delta^2_{min,n}\geq \frac{C\log(n)}{n},    
\end{equation}
where the constant $C$ here is larger than the one  given in \eqref{sub_optimal_condition}.
\item[(b)] Observe that in Corollary \ref{cor:consistency_PCP_estimator}, the set of relevant lags can be either finite or grow very slowly with respect to the sample size $n$. On the other hand, no assumption on the orders $d_n$ of the underlying sequence of MTD models is made. In particular, we could consider MTD models with very large orders, for example $ d_n=\beta n$ with $\beta\in (0,1)$. 
 %   \item[(g)] The lower bounds \eqref{cond_consis_PCP} on $\delta^*_{min}(n)$ required for the consistency of the PCP estimator says that the minimal oscillation can not decrease too fast to $0$ as $n$ diverges. 
 %   As we will see in Section XXX, this lower bound is only a $\log(n)$ factor away from the optimal one (in the minimax sense).
    \end{enumerate}
\end{remark}

As Corollary \ref{cor:consistency_PCP_estimator} indicates, in the setting $d_n=\beta n$, the major drawback of the \texttt{PCP} estimator $\hat{\Lambda}_{1,n}$ is that it requires a prior knowledge of $\Lambda_n$ in the form of a set $S_n$ growing slowly enough and such that $\Lambda_n\subseteq S_n$. The main goal of the next two sections is to propose alternative estimators of $\Lambda_n$ to deal with this issue.
%a necessary condition for consistency is $\delta^*_{min}(n)\geq C'\sqrt{\frac{\log(n)}{n}}$ for some constant $C'>0$. Thus, the
% It is important to mention that when we have a poor a priori knowledge of $\Lambda^{*}$ in the sense that $|S(n)|\geq \gamma d(n)$ for some $\gamma\in (0,1]$, then loosely speaking the known upper bound $d(n)$ for the true order of the MTD model has to grow at most logarithmically as function of $n$ to ensure that the error probability $\bP\left(\hat{\Lambda}_{1,n}\neq \Lambda\right)$ vanishes as $n\to \infty$. In other words, we do not expect to estimate consistently the set of relevant lags $\Lambda$ of high dimensional MTD models by using the \texttt{PCP} estimator $\hat{\Lambda}_{1,n}$, without a proper a priori knowledge of $\Lambda$. 

%Notice that the last assumption above requires that $\min_{j\in\Lambda(n)}\delta^2_j$ can not decrease too fast to $0$ as $n$ diverges, namely, we require that $\min_{j\in\Lambda(n)}\delta^2_j>log(n)/Cn^{(1+\beta_2)/2}$ for all $n$ large enough.

\subsection{Forward Stepwise and Cut estimator}
\label{Sec:FSC_estimator}

In this section  
we introduce a second estimator of the set of relevant lags $\Lambda$, called {\it Forward Stepwise and Cut} (\texttt{FSC}) estimator.  
This estimator is based on a structural result about MTD models presented in Proposition \ref{Prop:structure_result} below. Before presenting this structural result, we need to introduce some notation.

In what follows, for each lag $k\in \llb -d, -1\rrb$, subset $S\subseteq \llb -d, -1\rrb\setminus\{k\}$, configuration $x_{S}\in A^{S}$  and symbols $b,c\in A$, let us denote
\begin{equation}
\label{def:measure_influence}
d_{k,S}(b,c,x_{S})=d_{TV}\left(\bP_{x_{S}}(X_0\in \cdot|X_{k}=b), \bP_{x_{S}}(X_0\in \cdot|X_{k}=c)\right),
\end{equation}
and
\begin{equation}
\label{def:weights}
w_{k,S}(b,c,x_{S})=\bP_{x_{S}}(X_{k}=b)\bP_{x_{S}}(X_{k}=c).
\end{equation}
Recall that $\bP_{x_{S}}(X_0\in \cdot|X_{k}=b)$ and $\bP_{x_{S}}(X_{k}=b)$ denote, respectively, the conditional distribution of $X_0$ given $X_{S}=x_{S}$ and $X_{k}=b$ and the conditional probability of $X_k=b$ given $X_{S}=x_S$, with the convention that these conditional probabilities for $S= \emptyset$ correspond to the unconditional ones. 

Let us also denote for each lag $k\in \llb -d, -1\rrb$ and subset $S\subseteq \llb -d, -1\rrb\setminus\{k\}$,
\begin{equation}
\nu_{k,S}(x_{S})=\sum_{b\in A}\sum_{c\in A}w_{k,S}(b,c,x_{S})d_{k,S}(b,c,x_{S}),    
\end{equation}
and
\begin{equation}
\bar{\nu}_{k,S}=\bE\left(\nu_{k,S}(X_{S})\right).
\end{equation}

The quantity $\nu_{k,S}(x_{S})$ measures the influence of $X_{k}$ on $X_{0}$, conditionally on the variables $X_{S}=x_{S}$. The average conditional influence of $X_{k}$ on $X_{0}$ is measured through the quantity $\bar{\nu}_{k,S}$.

In the sequel, we write $\mCov_{x_{S}}(X_{0},X_{k})$ to denote the conditional covariance between the random variables $X_{0}$ and $X_{k}$ given that $X_{S}=x_{S}$. Here, we also use the convention that the conditional covariance for $S= \emptyset$ corresponds to the unconditional one. 
With this notation, we can prove the following structural result about MTD models.

\begin{proposition}
\label{Prop:structure_result}
For any lag $k\in \llb -d, -1\rrb$ and subset $S\subseteq \llb -d, -1\rrb\setminus\{k\}$,
    \begin{equation}
    \label{lower_bound_bar_nu}
        Diam(A)\|A\|_{\infty}\bar{\nu}_{k,S}\geq  \bE\left(\left|\mCov_{X_{S}}(X_0,X_{k})\right|\right).
    \end{equation}
Moreover, if Assumptions \ref{Ass:MTD_stationary} and \ref{Ass:non_determinist_conditional_averages} hold, then there exists a constant $\kappa>0$ such that the following property holds: for any $S\subseteq \llb -d, -1\rrb$ such that  $\Lambda\not\subseteq S$ there exists $k\in \Lambda\setminus S$ satisfying
\begin{equation}
\label{Key_covariance_inequality}
\bE\left(\left|\mCov_{X_{S}}(X_0,X_{k})\right|\right)\geq \kappa.
\end{equation}
Furthermore, if Assumption \ref{Ass:non_determinist_conditional_averages} is replaced by Assumption \ref{ass:lower_bound_kappa}, then the constant $\kappa$ satisfies
\begin{equation}
\label{lower_bound_on_kappa}
 \kappa\geq  \frac{p^2_{min}\Gamma_1 \min\{|b-c|^2:b\neq c\}\tilde{\delta}_{min}}{2\sqrt{|\Lambda|}},
\end{equation}
where $\tilde{\delta}_{min}$, $p_{min}$ and $\Gamma_1$ are defined respectively in \eqref{def:delta_min_and_}, \eqref{def:pstar} and \eqref{def:lower_bound_kappa}. 
\end{proposition}

The proof of Proposition \ref{Prop:structure_result} is given in \ref{Sec:proof_prop_struc_result}.

\begin{remark}
Denote $f(S)=\max_{k\in S^c}\bar{\nu}_{k,S}$, for each $S\subseteq \llb -d,-1\rrb$. On one hand, we have that $f(S)=0$ for any $S\subseteq \llb -d,-1\rrb$ such that $\Lambda\subseteq S$. This follows immediately from the definition of $\Lambda$. On the other hand, Proposition \ref{Prop:structure_result} assures that $f(S)\geq \kappa/Diam(A)\|A\|_{\infty}>0$ for any $S\subseteq \llb -d,-1\rrb$ such that $\Lambda\not\subseteq S$. Putting together these facts, we deduce that the set of relevant lags can be written as $\Lambda=\arg\min\{|S|:f(S)=0\}.$ This observation motivates the \texttt{FSC} estimator defined below. 
\end{remark}

In what follows, we split the data $X_{1:n}$ into two pieces. The first part is composed of the first $m$ symbols $X_{1:m}$ where $1\leq m<n$, whereas the second part is composed of the $n-m$ last symbols $X_{m+1:n}$. 
In the sequel, we write $\hat{\nu}_{m,k,S}$ to denote the empirical estimate of $\bar{\nu}_{k,S}$ computed from $X_{1:m}$. The formal definition of $\hat{\nu}_{m,k,S}$ involves extra notation and is postponed to Appendix \ref{proofs_of_FSC_selection_procedure}. %We do not make any use of it in the presentation of FSC estimator given below.

The \texttt{FSC} estimator is built in two steps. The first step is called Forward Stepwise (\texttt{\texttt{FS}}) and the second one is called \texttt{CUT}. In the \texttt{FS} step, we start with $S=\emptyset$ and add iteratively to the set $S$ a lag $j \in \arg\max_{k\in S^c}\hat{\nu}_{m,k,S}$, as long as $|S|<\ell$, where $0\leq \ell \leq d$ is a parameter of the estimator.
We denote $\hat{S}_{m}$ the set obtained at the end of \texttt{FS} step, with the convention that $\hat{S}_{m}=\emptyset$ if the parameter $\ell=0$.
As we will see, if $\ell$ is properly chosen the candidate set $\hat{S}_{m}$ will contain the set of relevant lags $\Lambda$  with high probability. It may, of course, include irrelevant lags $j$ (those with $\delta_j=0$).
In the \texttt{CUT} step, for each $j\in \hat{S}_{m}$, we remove $j$ from $\hat{S}_{m}$ 
unless $d_{TV}(\hat{p}_{m,n}(\cdot|x_{\hat{S}_m}),\hat{p}_{m,n}(\cdot|y_{\hat{S}_m}))\geq t_{m,n}(x_{\hat{S}_m},y_{\hat{S}_m}):=s_{m,n}(x_{\hat{S}_m})+s_{m,n}(y_{\hat{S}_m})$ for some $(\hat{S}_{m}\setminus\{j\})$-compatible pasts $x_{\hat{S}_m},y_{\hat{S}_m}\in A^{\hat{S}_{m}}$, where  $s_{m,n}(x_{\hat{S}_m})$ is given by \eqref{def:individual_threshold} replacing $\bar{N}_n(\cdot)$ and $\hat{p}_{n}(\cdot|\cdot)$ by $\bar{N}_{m,n}(\cdot)$ and $\hat{p}_{m,n}(\cdot|\cdot)$ respectively.
The \texttt{FSC} estimator is defined as the set $\hat{\Lambda}_{2,n}$ of all lags not removed in the \texttt{CUT} step. 
The pseudo-code of the algorithm to compute the \texttt{FSC} estimator is given in Algorithm \ref{def:FSC_algo}. 

\begin{algorithm}[ht]
\label{def:FSC_algo}
\texttt{FS} Step\;
 1. $\hat{S}_{m}=\emptyset$\;
 2. While $|\hat{S}_{m}|<\ell$\;
 3. Compute $j_*=\arg\max_{j\in \hat{S}_{m}^c} \hat{\nu}_{m,j,\hat{S}_{m}}$ and include $j_*$ in $\hat{S}_{m}$\;
 \texttt{CUT} step\;
 6. For each $j\in \hat{S}_{m}$, remove $j$ from $\hat{S}_{m}$ 
unless 
\[
d_{TV}(\hat{p}_{m,n}(\cdot|x_{\hat{S}_{m}}),\hat{p}_{m,n}(\cdot|y_{\hat{S}_{m}}))\geq t_{m,n}(x_{\hat{S}_{m}},y_{\hat{S}_{m}}),
\]
for some $(\hat{S}_{m}\setminus\{j\})$-compatible pasts $x_{\hat{S}_{m}},y_{\hat{S}_{m}}\in A^{\hat{S}_{m}}$
  \;
7. Output $\hat{S}_{m}$\;
 \caption{\texttt{FSC}$(X_1,\ldots, X_n)$}
\end{algorithm}

\begin{remark}
\label{rmk:FSC_estimator}
\begin{enumerate}
     \item[(a)] It is worth mentioning the following alternative algorithm (henceforth called Algorithm 2) to estimate the set of relevant lags $\Lambda$. As Algorithm \ref{def:FSC_algo}, Algorithm 2 has two steps as well. In the first step, we start with $S=\emptyset$ and add iteratively a lag $j\in \arg\max_{k\in S^c}\hat{\nu}_{n,k,S}$ as long as $\hat{\nu}_{n,j,S}>\tau$, where $\tau$ is a parameter of the algorithm and $\hat{\nu}_{n,j,S}$ is the empirical estimate of $\bar{\nu}_{j,S}$ computed from the entire data $X_{1:n}$. Let $\hat{S}_{n}$ denote the set obtained at the end of this step. Next, in the second step, for each $j\in \hat{S}_{n}$, we remove j from $\hat{S}_{n}$ unless $\hat{\nu}_{n,j,\hat{S}_{n}\setminus\{j\}}\geq \tau$. The output of Algorithm 2 is the set of all lags in $\hat{S}_{n}$ which were not removed in the second step. 
     Algorithm 2 can be seen as a version adapted for our setting of the 
     \texttt{LearnNbhd} algorithm, proposed in \cite{Bresler:15},  to estimate the interaction graph underlying an Ising model from i.i.d samples of the model. 
     \item[(b)]  As opposed to Algorithm 2, notice that the data $X_{1:n}$ is split into two parts in Algorithm \ref{def:FSC_algo}. The first $m$ symbols $X_{1:m}$ of the sample are used in the \texttt{FS} step, whereas the last $n-m$ symbols $X_{m+1:n}$ are only used in the \texttt{CUT} step. Despite requiring to split the data into two parts, one nice feature of Algorithm \ref{def:FSC_algo} is that even if a large $\ell$ is chosen the \texttt{CUT} step would remove the non-relevant lags, whereas in Algorithm 2, we have to calibrate $\tau$ carefully to recover the relevant lags. 
     %\item[(c)] It is intuitively clear that the first step of Algorithms \ref{def:FSC_algo} and 2 are dual to each other. This means that given a choice of $0\leq \ell\leq d$, one can find $\tau>0$ (and vice-versa) such that the sets obtained at the and of the first step of the respective algorithms are the same. Depending on the application, the user can decide whether it is better to choose $\ell$ or $\tau$.
     \item[(c)] (Computation of FSC estimator) As we show in the Appendix \ref{computation_PCP_and_FSC_estimator}, 
     %each maximization step in line 3 of the FS step requires at most $O(|A|^3(m-d)(d-|S|))$ computations. Hence, 
     we need to perform at most $O(|A|^3\ell(m-d)(d-(\ell-1)/2)+|A|^2(n-m-d)\ell)$  computations to determine the FSC estimator. The first summand in the sum corresponds to the algorithmic complexity of the FS step, whereas that the second summand can be interpreted as the algorithmic complexity of the PCP estimator computed from a sample of size $n-m$ and whose set $S$ has $\ell$ elements (recall item (c) of Remark \ref{rmk:PCP_consistency}).      
\end{enumerate}

\end{remark}

In what follows, for any $\xi>0$ and $0\leq \ell\leq d$, let us define the following event, 
\begin{equation}
\label{def:Good_event}
G_m(\ell,\xi)=\bigcap_{S\in\mathcal{S}_{d,\ell}}\left\{\max_{j\in S^c}|\bar{\nu}_{j,S}-\hat{\nu}_{m,j,S}|\leq \xi\right\},  
\end{equation}
where $\mathcal{S}_{d,\ell}=\{S\subseteq \llb -d, -1\rrb:\ |S|\leq \ell\}.$
In the next result we show that whenever the event $G_m(\ell,\xi)$ holds with
properly chosen parameters $\xi$ and $\ell$, 
the candidate set $\hat{S}_{m}$ constructed in the \text{FS} step with parameter $\ell$ contains $\Lambda$.

\begin{theorem}
\label{Prop:Correctness_1_FSC_Algo}
%Let $\hat{S}_{n,*}$ be the set constructed by Algorithm \ref{def:FSC_algo} at the end of step 5 with parameter $\tau$ and 
Suppose  Assumptions \ref{Ass:MTD_stationary} and \ref{Ass:non_determinist_conditional_averages} hold and let $\kappa$ be the lower bound provided by Proposition \ref{Prop:structure_result}.
%For $0<\zeta<1$ and $\eta>1-\zeta$, 
Let
\begin{equation}
\label{parameters_consistency}
\xi_*=\frac{\kappa}{4\|A\|_{\infty}Diam(A)} \ \text{and} \ \ell_*=\left\lfloor \frac{\log_2(|A|)}{8\xi_*^2}\right\rfloor=\left\lfloor \frac{2(Diam(A)\|A\|_{\infty})^2\log_2(|A|)}{\kappa^2}\right\rfloor.     
\end{equation} 
Let $\hat{S}_{m}$ denote the candidate set constructed in the FS step of Algorithm \ref{def:FSC_algo} with parameter $\ell_*$.
%as defined in \eqref{parameters_consistency}. 
%   On the event $G(\ell,\xi)$, we have that $|\hat{S}_{n,*}|\leq \ell$ and $\Lambda\subseteq \hat{S}_{n,*}$. Moreover, 
On the event $G_m(\ell_*,\xi_*)$, we have that $\Lambda\subseteq \hat{S}_{m}$.
%\begin{enumerate}
%\item Take $0<\xi<\tau$ and let $\ell=\frac{\log_2(|A|)}{2(\tau-\xi)^2}$. On the event $G(\xi,\ell)$, we have that $|\hat{S}_{n,*}|\leq \ell$.
   % \item Suppose  Assumptions \ref{Ass:MTD_stationary} and \ref{Ass:non_determinist_conditional_averages} hold. Take  $\xi=\tau(1-\zeta)$ with $0<\zeta<1$ and $\tau\leq \frac{k}{2(1+\eta)(\max_{a\in A}|a|)^2}$ with $\eta>1-\zeta$ and let $\ell=\frac{\log_2(|A|)}{2\tau^2(\eta+\zeta)^2}$. 
  %  On the event $G(\ell,\xi)$, we have that $|\hat{S}_{n,*}|\leq \ell$ and $\Lambda\subseteq \hat{S}_{n,*}$. Moreover, 
%    On the event $G(\ell,\xi)$, the set $\Lambda_{3,n}$ obtained by applying the CUT step to the set $\hat{S}_{n,*}$ is equal to $\Lambda$.
%\end{enumerate}
\end{theorem}

The proof of Theorem \ref{Prop:Correctness_1_FSC_Algo} is given in Appendix \ref{sec:proff_correctnes_1_FSC_algo}.
Theorem \ref{Prop:Correctness_1_FSC_Algo} ensures that the candidate set $\hat{S}_{m}$ contains the set of relevant lags $\Lambda$ whenever the event $G_m(\ell_*,\xi_*)$ holds. In this case,  we can think of the \texttt{CUT} step as the \texttt{PCP} estimator discussed in the previous section applied to the $n-m$ last observations $X_{m+1:n}$ of the data, where $S=\hat{S}_{m}$. The main difference is that $\hat{S}_{m}$ is a random set, depending on the first $m$ observations $X_{1:m}$ of the data.

In the sequel, let us denote 
\begin{equation}
\label{Def:min_min_max_prop_for_FSC_esti}
{\bf P}_{S_{d,\ell}}=\min_{S\in\mathcal{S}_{d,\ell}}{\bf P}_S,
\end{equation}
where ${\bf P}_S$ is defined in \eqref{Def:min_min_max_prop_for_PCP_esti}.

In the next result we estimate the error probability of the \texttt{FSC} estimator.

\begin{theorem}
\label{thm:consistency_FSC}
Suppose Assumptions \ref{Ass:MTD_stationary},  \ref{Ass:Concentration_inequalities}, and \ref{Ass:non_determinist_conditional_averages} hold. Let $\Delta>0$ be the quantity defined in Assumption \ref{Ass:Concentration_inequalities}.
Denote $\hat{\Lambda}_{2,n}$ the \texttt{FSC} estimator constructed  by Algorithm \ref{def:FSC_algo} with parameter $\ell_*$, as defined in \eqref{parameters_consistency}.
Suppose also that $m>d\geq 2\ell_{*}$. Then there exits a constant $C=C(\varepsilon,\mu)>0$ such that if
\begin{equation}
\label{def:min_sample_FSC}
n\geq m+d+\frac{C|A|\alpha}{\delta^2_{min}{\bf P}_{S_{d,\ell_*}}}, \end{equation}
where $\delta_{min}$ and ${\bf P}_{S_{d,\ell_*}}$ are defined in \eqref{def:delta_min_and_} and \eqref{Def:min_min_max_prop_for_FSC_esti},
then we have that,
\begin{multline}
\label{prob_error_estimate_FSC}
\bP(\hat{\Lambda}_{2,n}\neq \Lambda)\leq 2|A|(\ell_*+1){d \choose \ell_*}\left[ d|A|^{\ell_*+1}
\exp\left\{-\frac{(\xi_*\Delta)^2(m-d)^2}{18m|A|^{2(\ell_*+2)}(\ell_*+2)^2}\right\}\right. \\ \left.
+3(|A|-1)|\Lambda|\exp\left\{-\frac{\Delta^{2}(n-m-d)^2{\bf P}^2_{S_{d,\ell_*}}}{8(n-m)(\ell_*+1)^2}\right\}\right]\\
+8|A|\left((\ell_*-|\Lambda|)(n-m-d)+|\Lambda|\right)\left\lceil\frac{\log(\mu (n-m-d)/\alpha+2)}{\log(1+\varepsilon)}\right\rceil e^{-\alpha},
\end{multline}
where $\xi_*$ is defined in \eqref{parameters_consistency}.
\end{theorem}

The proof of Theorem \ref{thm:consistency_FSC} is given in Appendix \ref{Sec:proof_consistency_FSC}.

\begin{remark}
\label{rmk:consist_FSC}
\begin{enumerate}
\item[(a)]
Let us give some intuition about the three terms appearing on the right-hand side of \eqref{prob_error_estimate_FSC}.
    The first one is an upper bound for $\bP(G^c_m(\ell_*,\xi_*))$. The other two are related to the terms appearing in \eqref{PCP_estimator_error_probability}. Indeed, by recalling that $|\hat{S}_m|=\ell_*$, one immediately sees that the third terms of \eqref{prob_error_estimate_FSC}  corresponds to the first term of \eqref{PCP_estimator_error_probability} with $\hat{S}_m$ and $n-m$ in the place of $S$ and $n$ respectively. Besides, the second term of \eqref{prob_error_estimate_FSC} is similar (modulo a factor which depends on $d$ and $\ell_*$) to the second term of \eqref{PCP_estimator_error_probability}. This extra factor reflects the fact that we do not know a priori a set $S$ containing the set of relevant lags $\Lambda$.
\item[(b)] Similar to Remark \ref{rmk:PCP_exp}, one can also show that ${\bf P}_{S_{d,\ell}}\geq p_{min}/|A|^{\ell-1}$, where $p_{min}$ is defined in \eqref{def:pstar}. Hence, we can deduce from \eqref{def:min_sample_FSC} that when the sample size $n$ satisfies
\begin{equation}
\label{def_alternative:min_sample_FSC}
n\geq d+\frac{C|A|^{\ell_*}\alpha}{p_{min}\delta^2_{min}},
\end{equation}
then inequality \eqref{prob_error_estimate_FSC} holds with the second exponential term replaced by
\begin{equation}
\label{FSC_estimator_alternative_error_probability}
\exp\left\{-\frac{(\Delta p_{min})^2(n-d)^2}{8n(\ell_{*}+1)^2|A|^{2(\ell_*-1)}}\right\}.
\end{equation}
\end{enumerate}    
\end{remark}

The next result is a corollary of Theorem \ref{thm:consistency_FSC}.

\begin{corollary}
\label{cor:consistency_FSC}
For each $n$, consider a MTD model with set of relevant lags $\Lambda_{n}$ and transition probabilities $p_n(a|x_{\Lambda_n})$ satisfying $p_{min,n}\geq p^{\star}_{min}$ $\Delta_n\geq \Delta^{\star}_{min}$ for some positive constants  $p^{\star}_{min}$ and $\Delta^{\star}_{min}$, and such that Assumption \ref{Ass:non_determinist_conditional_averages} holds.
Let $m_n=n/2$ and $d_n=m_n\beta$ with $\beta\in (0,1)$.
Let $X_{1:n}$ be a sample from the MTD specified by $\Lambda_n$ and $p_n(a|x_{\Lambda_n}),$ and denote $\hat{\Lambda}_{2,n}$ the \texttt{FSC} estimator constructed  by Algorithm \ref{def:FSC_algo} with parameters $m_n$, $\mu_n=\mu\in(0,3)$ such that $\mu>\psi(\mu)$, $\varepsilon_n=\varepsilon>0$, $\alpha_n=(1+\eta)\log(n)$ with $\eta>0$ and $\ell_{*,n}$ as defined in \eqref{parameters_consistency}.
Assume that $\ell_{*,n}\leq ((1-\gamma)/2)\log_{|A|}(n)$ for some $\gamma\in (0,1)$.
Then there exists a constant $C=C(\beta,\gamma,\eta,p^{\star}_{min},\Delta^{\star}_{min},\varepsilon,\mu)>0$ such that $\bP(\hat{\Lambda}_{2,n}\neq \Lambda^*)\to 0$ as $n\to \infty$, whenever
\begin{equation}
\label{cond_consis_FSC}
\delta^2_{min,n}\geq \frac{C\log(n)}{n^{(1+\gamma)/2}}.  
\end{equation}
\end{corollary}

The proof of Corollary \ref{cor:consistency_FSC} is given in Appendix \ref{proof:cor_conist_FSC_Esti}.

\begin{remark}
\begin{enumerate}
    \item[(a)] Under Assumption \ref{ass:lower_bound_kappa}, one can check that $\ell_{*,n}\leq ((1-\gamma)/2)\log_{|A|}(n)$ whenever,
    $$
    \frac{\tilde{\delta}^2_{min}}{|\Lambda|}\geq \frac{16(1-\gamma)}{\Gamma_1^2(p^{\star}_{min}\min\{|b-c|:b\neq c\})^4\log_{|A|}(n)}.
    $$
    \item[(b)] Comparing Corollaries \ref{cor:consistency_FSC} and \ref{cor:consistency_PCP_estimator}, we observe that the consistency of both \texttt{FSC} and \texttt{PCP} estimators require the same lower bound on the decay of minimal oscillation $\delta_{min,n}$. Despite requiring additional assumptions (Assumption \ref{Ass:non_determinist_conditional_averages} and a condition on the growth of $\ell_*$), \texttt{FSC} estimator do not need prior knowledge of a small subset S containing the set of relevant lags $\Lambda$ as opposed to the \texttt{PCP} estimator, which is a significant advantage in practice. 
    \item [(c)] Let us mention that under the assumptions of Corollary \ref{cor:consistency_FSC}, we have that the algorithmic complexity of the FSC is $O(|A|^3 n^2\log_{|A|}(n))$. This follows immediately from item (c) of Remark \ref{rmk:FSC_estimator}.         
%\item[(b)] The key assumption to guarantee the consistency of the \texttt{FSC} estimator is the rate at which the lower bound $\kappa=\kappa(n)$ (provided by Proposition \ref{Prop:structure_result}) converges to 0 as $n\to\infty$. Loosely stated, this assumption requires that $\kappa(n)$ decreases slower than $1/\sqrt{\log\log(n)}$. This is always true in the case $\Lambda$ and the oscillations $\delta^{\star}_j$ do not scale with $n$. As we will see in the next section, this assumption on $\kappa$ can be relaxed when the alphabet is binary.   
\end{enumerate}
\end{remark}

\subsection{Improving the efficiency for the binary case}
\label{sec:supereficient_bin_MTD}

In this section, we show that when the alphabet is binary, i.e., $A=\{0,1\}$, we can further improve the FSC algorithm if we consider Assumptions \ref{ass:lower_bound_kappa} and \ref{ass:incoherence_condition_binary}. Observe that when the alphabet is binary,  Assumption \ref{Ass:non_determinist_conditional_averages} holds automatically (see Section \ref{sec:assump}). 
Moreover, we have that 
\begin{align*}
\bar{\nu}_{k,S}&=2\bE\left(\bP_{X_S}(X_k=1)\bP_{X_S}(X_k=1)
\left|\bP_{X_S}(X_0=1|X_k=1)-\bP_{X_S}(X_0=1|X_k=0)\right|\right)\\
&=2\bE\left(\left|\mCov_{X_{S}}(X_0,X_{k})\right|\right),
%&=2\bE\left(\left|\sum_{j\in\Lambda\setminus S}\Delta_{j}\mCov_{x_{S}}(X_{j},X_{k})\right|\right)
\end{align*}
for any lag $k\in \llb -d, -1\rrb$, subset $S\subseteq \llb -d, -1\rrb\setminus\{k\}$ and configuration $x_{S}\in\{0,1\}^{S}$.

For a binary MTD, we have the following result.

\begin{theorem}
\label{thm:FSC_consistency_binary}
Under Assumptions \ref{ass:lower_bound_kappa} and \ref{ass:incoherence_condition_binary}, it holds that
\begin{equation*}
\min_{S\subset\Lambda}\left(\max_{j\in \Lambda\setminus S}\bar{\nu}_{j,S}-\max_{j\in (\Lambda)^c}\bar{\nu}_{j,S}\right)\geq 2(\Gamma_1-\Gamma_2)p^2_{min} \delta_{min},    
\end{equation*}
where $\delta_{min}$ and $p_{min}$ are defined in \eqref{def:delta_min_and_} and \eqref{def:pstar} respectively. 
In particular, if $\Gamma_1>\Gamma_2$ and
$\hat{S}_m$ denotes the candidate set constructed at the end of the FS step of Algorithm \ref{def:FSC_algo} with parameter $\ell\geq |\Lambda|$, then $\Lambda\subseteq\hat{S}_m$ whenever the event $G_m(\ell,\xi)$ holds where
\begin{equation}
\label{choice_xi_consis_with_FS_step_only}
0<\xi<(\Gamma_1-\Gamma_2)p^{2}_{min} \delta_{min}.    
\end{equation}

\end{theorem}

The proof of Theorem \ref{thm:FSC_consistency_binary} is given in Appendix \ref{App:proof_them_FSC_consis_binary}.

\begin{remark} \label{rem:FS}
    Notice that if the size of $\Lambda$ is known, then Theorem \ref{thm:FSC_consistency_binary} implies $\hat{S}_m=\Lambda$ on the event $G_m(|\Lambda|,\xi)$ with $\xi$ satisfying \eqref{choice_xi_consis_with_FS_step_only}. %whenever the MTD model is weakly dependent. 
    In particular, in this case, we neither need to execute the \texttt{CUT} step nor to split the data into two pieces. 
    %We call this super efficient estimator \texttt{SEE} estimator.
    %Let us denote the SEE estimator by $\hat{\Lambda}_{3,n}$.
\end{remark}

In the same spirit of the previous corollaries, we can show the following result.

\begin{corollary}
\label{cor:consistency_SEE}
For each $n$, consider a MTD model with set of relevant lags $\Lambda_{n}$ and transition probabilities $p_n(a|x_{\Lambda_n})$ satisfying Assumptions \ref{ass:lower_bound_kappa} and \ref{ass:incoherence_condition_binary} with $\Gamma_{1,n}=\Gamma_1>\Gamma_2=\Gamma_{2,n}$ and such that $|\Lambda_{n}|\leq L$ for some integer $L$, $p_{min,n}\geq p^{\star}_{min}$ and $\Delta_n\geq \Delta^{\star}_{min}$ for some positive constants  $p^{\star}_{min}$ and $\Delta^{\star}_{min}$.
Let $X_{1:n}$ be a sample from the MTD specified by $\Lambda_n$ and $p_n(a|x_{\Lambda_n}),$ and denote $\hat{\Lambda}_{2,n}$ the \texttt{FSC} estimator with parameters 
with parameters $m_n=n/2$, $\mu_n=\mu\in(0,3)$ such that $\mu>\psi(\mu)$, $\varepsilon_n=\varepsilon>0$, $\alpha_n=(1+\eta)\log(n)$ with $\eta>0$ and $\ell=L$. Suppose that $d_n=\beta n$ with $\beta\in (0,1)$. 
Then there exists a constant $C=C(\beta,L,\Delta^*_{min}, p^*_{min},\Gamma_1, \Gamma_2,\eta,\mu,\varepsilon)>0$ such that $\bP(\hat{\Lambda}_{2,n}\neq \Lambda^*)\to 0$ as $n\to \infty$, as long as
\begin{equation}
\label{cond_consis_SEE}
\delta_{min,n}\geq C\sqrt{\frac{\log(n)}{n}},    
\end{equation}
\end{corollary}

The proof of Corollary \ref{cor:consistency_SEE} is given in Appendix \ref{sec:prof_cor_consist_SEE}.

\subsection{Post-selection transition probabilities estimation }
\label{Sec:ETP}
Once the set of relevant lags have been estimated by applying the FSC estimator to the sample $X_{1:n}$, we reuse the entire sample to compute the estimator $\hat{p}_n(a|x_{\hat{\Lambda}_{2,n}})$ of the transition probability $p(a|x_{\Lambda})$.  
In the next result, we provide an estimate for rate of convergence of $\hat{p}_n(a|x_{\hat{\Lambda}_{2,n}})$ towards $p(a|x_{\Lambda})$, simultaneously for all pasts $x_{-d:-1}\in A^{\llb -d,-1\rrb}.$
\begin{theorem}
\label{THM:ETP}
Under assumptions and notation of Theorem \ref{thm:consistency_FSC},
\begin{multline}
\bP\left(\bigcup_{a\in A}\bigcup_{x_{-d:-1}\in A^{\llb -d, -1\rrb}}\left\{|\hat{p}_n(a|x_{\hat{\Lambda}_{2,n}})-p(a|x_{\Lambda})|\geq \sqrt{\frac{2\alpha(1+\epsilon)\hat{V}_{n}(a,x_{\hat{\Lambda}_{2,n}})}{\bar{N}_{n}(x_{\hat{\Lambda}_{2,n}})}}\right.\right.\\
\left. \left. +\frac{\alpha}{3\bar{N}_{n}(x_{\hat{\Lambda}_{2,n}})} \right\}\right)
\leq  4|A|(n-d)\left\lceil\frac{\log(\mu (n-m-d)/\alpha+2)}{\log(1+\varepsilon)}\right\rceil e^{-\alpha}+\bP(\hat{\Lambda}_{2,n}\neq \Lambda),
\end{multline}
where $\hat{V}_{n}(a,x_{\hat{\Lambda}_{2,n}})$ is given by
$$
\hat{V}_{n}(a,x_{\hat{\Lambda}_{2,n}})=\frac{\mu}{\mu-\psi(\mu)}\hat{p}_{n}(a|x_{\hat{\Lambda}_{2,n}})+\frac{\alpha}{\mu-\psi(\mu)}\frac{1}{\bar{N}_{n}(x_{\hat{\Lambda}_{2,n}})}.
$$
\end{theorem}

The proof of Theorem \ref{THM:ETP} is given in Appendix \ref{sec_proof_section_ETP}.

%{\color{red} (Vamos usar esse remak em algum lugar?)
%\begin{remark}
%Splitting the sample to use part of it for variable selection and the other part to build confidence interval is now a classical procedure when the samples are independent \cite{dezeure2015high}. Nevertheless, for the dependent processes the procedure is not well established \citep{kuchibhotla2021post}. Our result use the martingale structure of the transition probability estimator and an improved martingale concentration inequality (Proposition \ref{prop:bernstein_empirical_trans_prob_correct_var}) to account for the dependence.  
%\end{remark}
%}

\subsection{A remark on the minimax rate for the lag selection}
\label{sec:minimax_rate}

We take $A=\{0,1\}$ and consider the set of $\{p^{(j)}(\cdot|\cdot),j \in \llb -d,-1\rrb\}$ of transition probabilities of the following form:
\begin{equation}
 p^{(j)}(1|x_{-d:-1})=\frac{(1-\lambda)}{2}+\lambda p(1|x_j), \ j\in \llb -d,-1\rrb, \lambda\in (0,1),   
\end{equation}
where $\lambda|p(1|1)-p(1|0)|:=\delta>0$. For each $j\in \llb -d,-1\rrb$, we denote $\bP^{(j)}$ the probability measure under which $(X_t)_{t\in\bZ}$ is a stationary MTD model having transition probability $p^{(j)}(\cdot|\cdot)$. For each $t \geq 1$, we denote $P^{(j)}_{t}$ the marginal distribution with respect to the variables $X_{1:t}$:
$$
P^{(j)}_t(x_{1:t})=\bP^{(j)}(X_{1:t}=x_{1:t})
$$
In what follows, $KL(P^{(j)}_{t}||P^{(k)}_{t})$ denotes the {\it Kullback-Leibler} divergence between the distributions $P^{(j)}_{t}$ and $P^{(k)}_{t}$. We denote $MTD_{d,\delta}$ the set all transition probabilities $p=\{p(a|x_{\Lambda}):a\in A,x_{\Lambda}\in A^{\Lambda}\}$ of a MTD model of order $d$ whose corresponding $\delta_{min}\geq \delta$. 
For a given $p\in MTD_{d,\delta}$, we denote $\bP_{p}$ the probability distribution under which 
$(X_t)_{t\in\bZ}$ is a stationary MTD model of order $d$ with transition probabilities given by $p$.
With this notation, we have the following result.

\begin{proposition}
\label{prop:upper_bound_on_minimal_oscilation}
Let $n>d$. Then the following inequality holds: for $j,k\in \llb -d,-1\rrb$,
\begin{equation}
\label{ineq:KL}
KL(P^{(j)}_{n}||P^{(k)}_{n})\leq \frac{2n\delta^2}{1-\lambda}.    
\end{equation}
In particular, if $\beta \in (0,1)$, $d=n\beta$, and 
\begin{equation}
\label{upper_bound_on_minimal_oscillations}
\delta^2\leq \frac{(1-\lambda)}{2n}\left(\frac{\log(n\beta)}{2}-\log(2)\right),
\end{equation}
then
\begin{equation}
\label{eq:lower_bound_minmax}
\inf_{\hat{\Lambda}_n}\sup_{p\in \text{MTD}_{d,\delta}}\bP_{p}(\hat{\Lambda}_n\neq\Lambda)\geq 1/4,    
\end{equation}
where the infimum is over all lag estimators $\hat{\Lambda}_n$ based on a sample of size $n$.
\end{proposition}

The proof of \eqref{eq:lower_bound_minmax} follows immediately from Fano's inequality and the upper bound \eqref{ineq:KL}.
Combining \eqref{cond_consis_SEE} and \eqref{upper_bound_on_minimal_oscillations}, we deduce that the condition on the minimal oscillation required for the consistency of the \texttt{FSC} estimator in Corollary \ref{cor:consistency_SEE} is sharp. The proof of Proposition \ref{prop:upper_bound_on_minimal_oscilation} is given in Appendix  \ref{prof_prop:upper_bound_on_minimal_oscilation}

\section{Simulations}
\label{sec:sim}
Here, we investigated the performance of the proposed methods using simulations. 

\subsection{Experiment 1}

 We first used a MTD model on alphabet $A=\{0,1\}$ with two relevant lags, denoted here as $-i$ and $-j$ for notational convenience. The choices for the order $d$ and for the values of $i$ and $j$ are shown in the first three columns of Table \ref{tbl0}.  Let $p_0(1) = p_0(0) = 0.5$ and $\lambda_0 = 0.4$. Also, let $\lambda_{-i} = 0.2$, $\lambda_{-j} = 0.4$, $p_{-i}(0|0) = 0.3$, $p_{-i}(0|1) = 0.6$, $p_{-j}(0|0) = 0.5$, and $ p_{-j}(0|1) = 0.9$. For all $x_{-d:-1} \in \{0,1\}^{\llb -d,-1\rrb}$ and $a \in A$, the transition probability of the model was given by
\begin{equation*}
p(a|x_{-d:-1})  = \lambda_0p_0(a) + \lambda_{-i}p_{-i}(a|x_{-i}) + \lambda_{-j}p_{-j}(a|x_{-j}).
\end{equation*}

We simulated the above model using sample sizes $n \in \{10^2, 5.10^2,10^3, 5.10^3, 10^4, 5.10^4, 10^5, 5.10^5 \}$. For each choice of $i,j,d$, and $n$ we simulated 100 realizations.
We compared four different methods to select the relevant lags. \texttt{FSC}$(\ell)$ stands for the Forward Stepwise and Cut algorithm described in Algorithm 1 with parameter $\ell$, $\epsilon = 0.1$, $\mu = 0.5$, and $\alpha = C \log(n)$, where the values of the constant $C$ was chosen by optimizing the probability to select the  relevant lags correctly only for sample size $n = 100$, for the given choice of $d$, $i$ and $j$. We used the first $n/2$ samples for the Forward Stepwise and the last $n/2$ for Cut. Remember that $\epsilon, \mu, \alpha$ are used to define the random threshold for the Cut step. \texttt{BSS}(2) stands for the best subset selection algorithm, where we first estimated the parameters of the MTD model using $n$ samples and the algorithm described in \cite{berchtold2001estimation} with python implementation \texttt{mtd-learn}. This algorithm estimated for $k \in \{1, 2, \ldots, d\}$ the weight parameters $\lambda_{-k}$. We then choose the lags of the two largest $\lambda_{-k}$ as the lags selected by \texttt{BSS}(2). 
We were not able to run the mtd-learn on models with order $d$ larger than 15 in our computers because that algorithm did not converge. Finally, \texttt{CTF}($\ell$) stands for Conditional Tensor Factorization based Higher Order Markov Chain estimation together with the test for relevance of lags described in \cite{sarkar2016bayesian},
the parameter $\ell$ being the maximal number of relevant lags.
We used the code available at \url{https://github.com/david-dunson/bnphomc}. The maximal possible order of the Markov chain was set to $d$ and the number of simulation for the Gibbs sampler was set to 1000. The set of relevant lags chosen by \texttt{CTF} was given by the lags with non-null inclusion probability estimated using the Gibbs sampler. We were not able to run \texttt{CTF}($\ell$) when $j=n/5$ and $d=n/4$ because the algorithm did not converge when $n > 10^{3}$.  We note that \texttt{FS} and \texttt{BSS} assume prior knowledge of the number of relevant sites, giving advantage over \texttt{FSC} and \texttt{CTF}.  The results are indicated in Table \ref{tbl0}.

\begin{table}[ht]
\caption{Estimated probability of correctly selecting only the relevant lags.}
\begin{tabular}{lll|l|lllllllll} \label{tbl0}\\
\hline
\multicolumn{3}{c}{Model parameter} &  \multicolumn{1}{c}{Method} & \multicolumn{8}{c}{Sample size (n)} \\
$i$ & $j$ & $d$ &  & $\mathbf{100}$ & $\mathbf{500}$ & $\mathbf{10^3}$ & $\mathbf{5.10^3}$ & $\mathbf{10^4}$ & $\mathbf{5.10^4}$ & $\mathbf{10^5}$ & $\mathbf{5.10^5}$\\ 
\hline \\
1 & 8 &  8 & {\texttt{FSC}(3)} & 0.05 & 0.08 & 0.13 & 0.53 & 0.81 & 0.86 & 0.93 &1 \\
1 & 8 & 8 & {\texttt{CTF}(3)} & 0 & 0 & 0.04 & 0.67 &  0.99 &  1 & 1 & 1\\
1 & 8 &  8 & {\texttt{FS}(2)} & 0.07 & 0.3 & 0.47 & 0.98 & 1 & 1 & 1 & 1 \\
1 & 8 & 8 & {\texttt{BSS}(2)} & 0.05 & 0.14 & 0.23 & 0.41 & 0.79 & 0.78 & 0.84 & 0.87\\
\hline
1 & 15 & 15 & {\texttt{FSC}(5)} & 0.03 & 0.36 & 0.51 & 0.82 &  0.97 &  1 & 1 & 1\\
1 & 15 & 15 & {\texttt{CTF}(5)} & 0 & 0 & 0.01 & 0.62 &  0.99 &  1 & 1 & 1\\
1 & 15 &  15 & {\texttt{FS}(2)} & 0.02 & 0.2 & 0.66 & 0.92 & 1 & 1 & 1 & 1 \\
1 & 15 & 15 & {\texttt{BSS}(2)} & 0.04 & 0.18 & 0.17 & 0.28 & 0.31 & 0.8 & 0.8 & 0.93 \\
\hline
1 & n/5 & n/4 & {\texttt{FSC}(5)} & 0 & 0 & 0.04 & 0.19 & 0.46 & 1 & 1 & 1 \\
1 & n/5 & n/4 & {\texttt{CTF}(5)} & 0 & 0 & 0 & - & - & - & - & - \\
1 & n/5 & n/4 & {\texttt{FS}(2)} & 0.01 & 0.11 & 0.27 & 0.89 & 0.96 & 1 & 1 & 1 \\
1 & n/5 & n/4 & {\texttt{BSS}(2)} & - & - & - & - & - & - & - & - \\
\hline
\end{tabular}
\end{table}

%\resizebox{\textwidth}{!}{%
%\begin{tabular}{lllllll|l|lllllll}\\
%\hline
%\multicolumn{6}{c}{Model parameter} &  & {Method} & %\multicolumn{6}{c}{Sample size (n)} \\
%$i$ & $j$ & $p_i(0|0)$ & $p_i(0|1)$ & $p_j(0|0)$ & $p_j(0|1)$ & $d$ &  & %\mathbf{384} & \mathbf{768} & \mathbf{1536} & \mathbf{3072} & %\mathbf{6144} & \mathbf{12288} \\ 
%\hline \\
%1 & 5 & 0.7 & 0.3 & 0.3 & 0.7 & 5 & {FSC(2)} & 0.0331 & 0.0312 & 0.0237 & 0.0132 & 0.0086 & 0.0061 \\
%1 & 5 & 0.7 & 0.3 & 0.3 & 0.7 & 5 & {FSC(5)} & 0.0298 & 0.0286 & 0.0271 & 0.0217 & 0.0119 & 0.0062 \\
%1 & 5 & 0.7 & 0.3 & 0.3 & 0.7 & 5 & {cut} & 0.0483 & 0.0350 & 0.0230 & 0.0157 & 0.0102 & 0.0077 \\
%1 & 5 & 0.7 & 0.3 & 0.3 & 0.7 & 5 & {full} & 0.0461 & 0.0301 & 0.0202 & 0.0128 & 0.0106 & 0.0075 \\
%1 & 5 & 0.7 & 0.3 & 0.3 & 0.7 & 10 & {FSC(5)} & 0.0370 & 0.0338 & 0.0299 & 0.0223 & 0.0109 & 0.0080 \\
%1 & 5 & 0.7 & 0.3 & 0.3 & 0.7 & 10 & {cut} & 0.0165 & 0.0131 & 0.0152 & 0.0236 & 0.0345 & 0.0379 \\
%1 & 5 & 0.7 & 0.3 & 0.3 & 0.7 & 10 & {full} & 0.1907 & 0.1728 & 0.1339 & 0.0951 & 0.0741 & 0.0463 \\
%1 & 10 & 0.7 & 0.3 & 0.3 & 0.7 & 15 & {FSC(5)} & 0.0354 & 0.0313 & 0.0325 & 0.0198 & 0.0109 & 0.0075 \\
%1 & 15 & 0.7 & 0.3 & 0.3 & 0.7 & 20 & {FSC(5)} & 0.0359 & 0.0343 & 0.0281 & 0.0197 & 0.0118 & 0.0070 \\
%11 & 100 & 0.7 & 0.3 & 0.3 & 0.7 & 120 & {FSC(5)} & 0.0695 & 0.0195 & 0.0249 & 0.0277 & 0.0147 &  0.0072\\
%1 & 10 & 0.7 & 0.3 & 0.3 & 0.7 & n/12 & {FSC(5)} & 0.0206 & 0.0179 & 0.0205 & 0.0214 & 0.0149 &  0.0106\\
%\hline
%\end{tabular}%
%}

\subsection{Experiment 2}
Here we used the following MTD model on alphabet $A=\{0,1\}$. We considered different choices of order $d$ and relevant lags $-i,-j \in \llb 1, d \rrb$ (see Table \ref{tbl1}).
Let $p_0(1) = p_0(0) = 0.5$ and $\lambda_0 = 0.2$. Also, let $p_{-i}(0|0) = 1-p_{-i}(0|1) = p_{-j}(0|1) = 1-p_{-j}(0|1) = 0.7$ and $\lambda_{-i} = \lambda_{-j} = 0.4$. For all $x_{-d:-1} \in \{0,1\}^{\llb -d,-1\rrb}$ and $a \in A$, the transition probability of the model was given by
\begin{equation*}
p(a|x_{-d:-1})  = \lambda_0p_0(a) + \lambda_{-i}p_{-i}(a|x_{-i}) + \lambda_{-j}p_{-j}(a|x_{-j}).
\end{equation*}

We simulated the above model using sample sizes $n \in \{2^8,2^9,2^{10},2^{11}, 2^{12}, 2^{13} \}$. For each choice of $i,j,d$, and $n$ we simulated 100 realizations. For each realization, we estimated the transition probability $p(0|0^d)$. We used different estimators for the comparisons. \texttt{FSC}$(\ell)$ and  \texttt{FS}$(\ell)$ are the same as described in Experiment 1.
For transition probability estimation with \texttt{FSC}, we used $X_{1:n/2}$ for Forward Stepwise and $X_{n/2+1:n}$ for Cut step, obtaining the estimated relevant lag set $\hat{\Lambda}_n$. Then we used $X_{1:n}$ to calculate $\hat{p}_{n}(0|0_{\hat{\Lambda}_n})$. For transition probability estimation after \texttt{PCP}, we used $X_{1:n}$ to calculate $\hat{\Lambda}_n$ for the  \texttt{PCP} relevant lag estimator with initial set $S = \llb -d,-1\rrb$. The parameters for the threshold were chosen as follows: $\epsilon = 0.1$, $\mu = 0.5$, and $\alpha = C \log(n)$, where we choose the values of the constant $C$ by optimizing the probability to select the  relevant lags correctly only for sample size $n = 100$, for the given choice of $d$, $i$ and $j$. Then we used $X_{1:n}$ to calculate $\hat{p}_{n}(0|0_{\hat{\Lambda}_n})$. We also compared the performance of transition probability estimator $\hat{p}_{n}(0|0_{-d:-1})$, where we did not select the relevant lags (\texttt{Naive} estimator). In our simulations, when $d$ was larger than 5, for both \texttt{PCP} and \texttt{Naive} estimators we did not obtain meaningful results because  $\bar{N}_{n}(0^d) = 0$ with high probability. Therefore, we compared \texttt{PCP} and \texttt{Naive} estimators only for $d = 5$. In this case, \texttt{FSC} showed similar performance to \texttt{PCP} estimator and was in general better than \texttt{Naive} estimator. When $d > 5$, \textit{e.g.} $d = n/8$, \texttt{FSC} still exhibited good performance. The results are indicated in Table \ref{tbl1}.

%\resizebox{0.8\textwidth}{!}{%
\begin{table}[ht]
\caption{Empirical standard deviation of the estimator of $p(0|0^d)$. \texttt{FSC}, \texttt{FS}, \texttt{PCP}, and \texttt{Naive} are described in the main text.}
\begin{tabular}{lll|l|lllllll} \label{tbl1}\\
\hline
\multicolumn{3}{c}{Model parameter} &  \multicolumn{1}{c}{Method} & \multicolumn{6}{c}{Sample size (n)} \\
$i$ & $j$ & $d$ &  & $\mathbf{256}$ & $\mathbf{512}$ & $\mathbf{1024}$ & $\mathbf{2048}$ & $\mathbf{4096}$ & $\mathbf{8192}$ \\ 
\hline \\
1 & 5 &  5 & {\texttt{FS}(2)} & 0.0774 & 0.0682 & 0.0506 & 0.0286 & 0.0174 & 0.0133 \\
1 & 5 & 5 & {\texttt{FSC}(5)} & 0.0745 & 0.0835 & 0.0602 & 0.0426 & 0.0222 & 0.0129 \\
1 & 5 & 5 & {\texttt{PCP}} & 0.0965 & 0.0786 & 0.0577 & 0.0432 & 0.0242 & 0.0131 \\
1 & 5 & 5 & {\texttt{Naive}} & 0.1518 & 0.0933 & 0.0624 & 0.0455 & 0.0340 & 0.0252 \\
1 & 5 & 10 & {\texttt{FSC}(5)} & 0.0836 & 0.0842 & 0.0659 & 0.0425 & 0.0228 & 0.0141 \\
1 & 10 & 15 & {\texttt{FSC}(5)} & 0.0864 & 0.0781 & 0.0641 & 0.0438 & 0.0249 & 0.0151 \\
1 & 15 & 15 & {\texttt{FSC}(5)} & 0.0833 & 0.0834 & 0.0747 & 0.0488 & 0.0222 & 0.0167 \\
11 & 100 & 120 & {\texttt{FSC}(5)} & - & - & 0.0838 & 0.0647 & 0.0312 &  0.0169\\
1 & 10 & n/8 & {\texttt{FSC}(5)} & 0.0563 & 0.0543 & 0.0780 & 0.0698 & 0.0504 &  0.0105\\
\hline
\end{tabular}%
\end{table}
%}

\subsection{Application}
We applied the proposed method to study the relevant lags on a daily weather data registering the rainy and non-rainy days in Canberra Australia for a $n = 1000$ days. We obtained the data from kaggle \\(\url{https://www.kaggle.com/datasets/jsphyg/weather-dataset-rattle-package}). We used Forward Stepwise algorithm with $\ell = 3$ (\texttt{FS}(3)) and maximal order $d = 400$ to include the possibility of the annual cycle.  We obtained as the three relevant lags $\{1, 62, 330\}$. The selected relevant lags were the same for $d = 365$ and $d = 450$, showing teh robustness of the result. The day before (lag 1) is clearly relevant and is often included in weather prediction models. Annual cycles ($\approx 12$ months) are also predictor of the weather, matching the 330 days lag that we found. Finally, the 62 days lag is consistent with the cycle of Madden-Julian oscillator ($\approx$ 60 days), which is the largest inter-seasonal source of precipitation events in Australia (\cite{wheeler2009impacts}). We note that the estimated Markov chain is of order 330, which is around one-third of the sample size $n = 1000$, whereas using VLMC we do not expect to typically estimate Markov chains of order larger than $\log (10) \approx 7$. Indeed, using VLMC with BIC model selection criterion we selected a model with order 1. We set the upper limit of the model size as 400 for the VLMC order selection. As a further comparison, we applied the Conditional Tensor Factorization based Higher Order Markov Chain estimation together with the test for relevance of lags described in \cite{sarkar2016bayesian}. We again used the code available at \url{https://github.com/david-dunson/bnphomc}. The maximal possible order of the Markov chain was set to 400, the maximal number of relevant lags was 3, and the number of simulation for the Gibbs sampler was set to 1000. The inclusion probability calculated using Gibbs sampler for lags $(1, 2, 3, 4, 5, 6, 7)$ were $(100, 1.2, 0.2, 0.6, 0.2, 0.4, 0.2$) percent, respectively. For all other orders the inclusion probability was zero. Therefore, no larger lags were detected by this method. 

\section{Acknowledgments}
This research has been conducted as part of FAPESP project {\em Research, Innovation and
Dissemination Center for Neuromathematics} (grant 2013/07699-0).
GO is partially supported by FAPERJ (grants E-26/201.397/2021 and E-26/211.343/2019). and CNPq (grant 303166/2022-3).

\appendix

\section{Proofs of Section \ref{sec:SLS}}
\label{proofs_of_FSC_selection_procedure}

\subsection{Proofs of Section \ref{Sec:PCP_estimator}}

\subsubsection{Proof of Theorem \ref{thm:consistency_PCP_estimator}}
\label{Sec:proof_of_thm_cons_PCP_estimator}
\begin{proof}[Proof of Theorem \ref{thm:consistency_PCP_estimator}]
Since the set $S\subseteq \llb -d,-1\rrb$ containing the set $\Lambda$ is fixed, we will write $x$ instead of $x_S$ to alleviate the notation. We start proving Item 1.

{\bf Proof of Item 1}. For each $x\in A^{S}$, let us define the event
$$
G_{x}=\bigcap_{a\in A}\left\{|\hat{p}_n(a|x)-p(a|x)|<\sqrt{\frac{2\alpha(1+\varepsilon)\hat{V}_n(a,x)}{\bar{N}_n(x)}}+\frac{\alpha}{3\bar{N}_n(x)}\right\},
$$
where $\hat{V}_n(a,x)$ is given by
$$
\hat{V}_n(a,x)=\frac{\mu}{\mu-\psi(\mu)}\hat{p}_{n}(a|x)+\frac{\alpha}{(\mu-\psi(\mu))\bar{N}_{n}(x)}.
$$

By using first the union bound and then by applying Proposition \ref{prop:bernstein_empirical_trans_prob_correct_var}, we deduce that for each $x\in A^{S}$,
\begin{equation}
\label{ineq_proof_thm_1}
\bP(G^c_{x})\leq 4|A|\left\lceil\frac{\log(\mu (n-d)/\alpha+2)}{\log(1+\varepsilon)}\right\rceil e^{-\alpha}\bP\left(\bar{N}_{n}(x)>0\right).   
\end{equation}

Now, observe that on $G_{x}$, we have that
$$
d_{TV}(\hat{p}_n(\cdot|x),p(\cdot|x))<\sum_{a\in A}\sqrt{\frac{\alpha(1+\varepsilon)\hat{V}_n(a,x)}{2\bar{N}_n(x)}}+\frac{\alpha |A|}{6\bar{N}_n(x)}=s_n(x),
$$
which, together with \eqref{ineq_proof_thm_1}, implies that
\begin{equation}
\label{ineq_2_proof_thm_1}
\bP\left(d_{TV}(\hat{p}_n(\cdot|x),p(\cdot|x))\geq s_n(x)\right)\leq 4|A|\left\lceil\frac{\log(\mu (n-d)/\alpha+2)}{\log(1+\varepsilon)}\right\rceil e^{-\alpha}\bP\left(\bar{N}_{n}(x)>0\right).
\end{equation}

Note that if $j\notin\Lambda$, then by the definition of the set $\Lambda$ it follows that $p(a|x)=p(a|y)$ for all $x,y\in A^{S}$ which are $(S\setminus\{j\})$-compatible.
Hence, by applying first the triangle inequality and then using that 
$t_{n}(x,y)=s_n(x)+s_n(y)$, we deduce that the event
$
\{d_{TV}(\hat{p}_n(\cdot|x),\hat{p}_n(\cdot|y))\geq t_{n}(x,y)\} 
$
is contained in the event 
$$
\{d_{TV}(\hat{p}_n(\cdot|x),p(\cdot|x))\geq s_n(x)\}\cup \{d_{TV}(\hat{p}_n(\cdot|y),p(\cdot|y))\geq s_n(y)\},
$$
so that
\begin{align*}
\bP\left(j\in \hat{\Lambda}_{1,n}\right)&\leq 2\sum_{x\in A^{S}}\bP(d_{TV}(\hat{p}_n(\cdot|x),p_n(\cdot|x))\geq s_{n}(x))\\
&\leq 8|A|\left\lceil\frac{\log(\mu (n-d)/\alpha+2)}{\log(1+\varepsilon)}\right\rceil e^{-\alpha}\sum_{x\in A^{S}}\bP(\bar{N}_{n}(x)>0),
\end{align*}
where in the second inequality we have used \eqref{ineq_2_proof_thm_1}.

Since $n-d=\sum_{x\in A^{S} }\bar{N}_{n}(x)\geq \sum_{x\in A^{S}}1\{\bar{N}_{n}(x)> 0\}$ which implies that $n-d\geq \bE\left[\sum_{x\in A^{S} }1\{\bar{N}_{n}(x)>0 \}\right]=\sum_{x\in A^{S}}\bP(\bar{N}_{n}(x)>0)$, we obtain from the above inequality that,
$$
\bP\left(j\in \hat{\Lambda}_{1,n}\right)\leq 8|A|\left\lceil \frac{\log(\mu(n-d)/\alpha+2)}{\log(1+\varepsilon)}\right\rceil e^{-\alpha}(n-d),
$$
concluding the the proof of Item 1.

%We now prove Item 2. 
{\bf Proof of Item 2}. Let $j\in\Lambda$, recall that $\delta_j=\lambda_j \max_{b,c\in A}d_{TV}(p_j(\cdot|b),p_j(\cdot|c))$ and consider the event $E=\{\delta_j\geq \gamma_{n,j}\}$.
Take $b^{\star},c^{\star}\in A$ such that $d_{TV}(p_j(\cdot|b^{\star}),p_j(\cdot|c^{\star}))=\max_{b,c\in A}d_{TV}(p_j(\cdot|b),p_j(\cdot|c))$,
and observe that with this choice,
$$
\delta_j=\lambda_j d_{TV}(p_j(\cdot|b^{\star}),p_j(\cdot|c^{\star})).
$$
From this equality it follows that for any pair $(x,y)\in \cC_j(b^{\star},c^{\star})$, we have $$
E=\{d_{TV}(p(\cdot|x),p(\cdot|y))\geq 2t_{n,j}\},
$$
where we have used also that $\gamma_{n,j}=2t_{n,j}.$
Now, take a pair $(x^{\star},y^{\star})\in \cC_j(a^{\star},b^{\star})$ attaining the minimum in \eqref{Def:thresholds_for_PCP_esti}:
$$t_{n,j}(b^{\star},c^{\star})=t_{n}(x^{\star},y^{\star}).
$$ 
From the definition of $t_{n,j},$ it follows then that 
$$
t_{n,j}\geq t_{n,j}(b^{\star},c^{\star})=t_{n}(x^{\star},y^{\star}).
$$
Therefore, we conclude that
$$
E\subseteq \{d_{TV}(p(\cdot|x^{\star}),p(\cdot|y^{\star}))
\geq 2t_{n}(x^{\star},y^{\star})\},
$$
so that by the triangle inequality, we obtain that on $E$,
\begin{multline*}
2t_{n}(x^{\star},y^{\star})\leq    d_{TV}(\hat{p}_n(\cdot|x^{\star}),p(\cdot|x^{\star}))\\+d_{TV}(\hat{p}_n(\cdot|y^{\star}),p(\cdot|y^{\star}))+d_{TV}(\hat{p}_n(\cdot|x^{\star}),\hat{p}_n(\cdot|y^{\star})).
\end{multline*}
Hence, on $\{j\notin\hat{\Lambda}_{1,n}\}\cap E$, we have
\begin{equation*}
t_{n}(x^{\star},y^{\star})\leq    d_{TV}(\hat{p}_n(\cdot|x^{\star}),p(\cdot|x^{\star}))\\+d_{TV}(\hat{p}_n(\cdot|y^{\star}),p(\cdot|y^{\star})),
\end{equation*}
implying that
\begin{multline*}
   \bP\left(\{j\notin\hat{\Lambda}_{1,n}\}\cap E\right)\leq \bP(d_{TV}(\hat{p}_n(\cdot|x^{\star}),p(\cdot|x^{\star}))\geq s_n(x^{\star}))\\+\bP(d_{TV}(\hat{p}_n(\cdot|y^{\star}),p(\cdot|y^{\star}))\geq s_n(y^{\star})).
\end{multline*}
From \eqref{ineq_2_proof_thm_1}, it follows then that
\begin{align*}
\bP\left(j\notin\hat{\Lambda}_{1,n},\gamma_{n,j}\leq \delta_j\right)=\bP\left(\{j\notin\hat{\Lambda}_{1,n}\}\cap E\right)\leq  8|A|\left\lceil\frac{\log(\mu (n-d)/\alpha+2)}{\log(1+\varepsilon)}\right\rceil e^{-\alpha},
\end{align*}
concluding the proof of Item 2.

%Thus, it remains to proof Item 3.
{\bf Proof of Item 3}.
Observe that by combining Items 1 and 2 together with the union bound, we deduce that
\begin{multline*}
    \bP\left(\hat{\Lambda}_{1,n}\neq \Lambda\right)\leq 8|A|\left((|S|-|\Lambda|)(n-d)+|\Lambda|\right)\left\lceil \frac{\log(\mu(n-d)/\alpha+2)}{\log(1+\varepsilon)}\right\rceil e^{-\alpha}
        \\+\sum_{j\in \Lambda}\bP\left(\gamma_{n,j}>\delta_j \right).
\end{multline*}

Hence, to conclude the proof of Item 3, it suffices to show that

\begin{equation}
\label{proof_cor_thm_SSP2_key_ineq}
\bP\left(\gamma_{n,j}>\delta_j \right)\leq
6|A|(|A|-1)\exp\left\{\frac{-\Delta^2(n-d)^2{\bf P}^2_S}{8n(|S|+1)^2}\right\},
\end{equation}
for all $j\in \Lambda$, whenever the sample size $n$ satisfies \eqref{def:min_sample}.

By the union bound, we have that 
\begin{equation}
\label{proof_cor_thm_SSP2_key_ineq4}
\bP\left(\gamma_{n,j}>\delta_j \right)\leq \sum_{b\in A}\sum_{c\in A:c\neq b}\bP\left(t_{n,j}(b,c)>\delta_j/2 \right),
\end{equation}
and for each $b,c\in A$ with $b\neq c$,
\begin{equation}
\label{prof_thm_PCP_estimator:ineq_1}
\bP\left(t_{n,j}(b,c)>\delta_j/2 \right)\leq \bP\left(t_{n}(x,y)>\delta_j/2 \right) 
\leq \bP\left(s_n(x)>\delta_j/4\right)+\bP\left(s_n(y)>\delta_j/4\right),
\end{equation}
for any $(x,y)\in \cC_j(b,c)$.
%\begin{multline*}
%\bP\left(t_{n,j}(b,c)>\delta_j/2 \right)\leq 
%\bP\left(2\sum_{a\in A}\sqrt{\frac{2\alpha(1+\varepsilon)\hat{V}_n(x,y,a)}{\bar{N}_{n}(x)\wedge \bar{N}_{n}(y)}}>\delta_j/4\right)\\+\bP\left(\frac{2\alpha}{3\bar{N}_{n}(x)\wedge %\bar{N}_{n}(y)}>\frac{\delta_j}{4|A|}\right),
%\end{multline*}
%for any $(x,y)\in \cC_j(b,c)$.

By using again the union bound, we can deduce that for each in $x\in A^{S}$,
\begin{equation*}
\bP\left(s_n(x)>\delta_j/4\right)\leq \bP\left(\sum_{a\in A}\sqrt{\frac{\alpha(1+\varepsilon)\hat{V}_n(a,x)}{2\bar{N}_{n}(x)}}>\delta_j/8\right)
+\bP\left(\frac{\alpha}{6\bar{N}_{n}(x)}>\frac{\delta_j}{8|A|}\right).
\end{equation*}
and also that
\begin{multline*}
\bP\left(\sum_{a\in A}\sqrt{\frac{\alpha(1+\varepsilon)\hat{V}_n(a,x)}{2\bar{N}_{n}(x)}}>\delta_j/8\right)\leq \bP\left(\sum_{a\in A}\sqrt{\frac{\alpha(1+\varepsilon)\hat{p}_n(a|x)\mu}{2(\mu-\psi(\mu))\bar{N}_{n}(x)}}>\delta_j/16\right)\\
+\bP\left(\frac{|A|\alpha}{\bar{N}_{n}(x)}\sqrt{\frac{(1+\varepsilon)}{2(\mu-\psi(\mu))}}>\delta_j/16\right).
\end{multline*}

By applying Proposition \ref{prop:gaussian_concentration_inequality} with $u_1={\bf P}(x)-(4|A|\alpha)/(3\delta_j(n-d))$ and $u_2={\bf P}(x)-(16|A|\alpha\sqrt{(1+\varepsilon)/2(\mu-\psi(\mu))})/(\delta_j(n-d))$, one can show that
\begin{multline*}
\bP\left(\frac{|A|\alpha}{6\bar{N}_{n}(x)}>\delta_j/8\right)+\bP\left(\frac{|A|\alpha}{\bar{N}_{n}(x)}\sqrt{\frac{(1+\varepsilon)}{2(\mu-\psi(\mu))}}>\delta_j/16\right)
\\ \leq
2\exp\left\{-\frac{\Delta^2(n-d)^2}{2n(|S|+1)^2}\left({\bf P}(x)-\frac{16|A|\alpha}{\delta_j(n-d)}\sqrt{\frac{(1+\varepsilon)}{2(\mu-\psi(\mu))}}\right)^2\right\},
\end{multline*}
as long as
$(n-d)>\frac{16|A|\alpha}{\delta_j{\bf P}(x)}\sqrt{\frac{(1+\varepsilon)}{2(\mu-\psi(\mu))}}$.

By using Jensen inequality, one can verify that
\begin{equation*}
\bP\left(\sum_{a\in A}\sqrt{\frac{\alpha(1+\varepsilon)\hat{p}_n(a|x)\mu}{2(\mu-\psi(\mu))\bar{N}_{n}(x)}}>\delta_j/16\right)\leq \bP\left(\bar{N}_{n}(x)< \frac{128\alpha(1+\varepsilon)\mu|A|}{\delta^2_j(\mu-\psi(\mu))}\right),    
\end{equation*}
so that by Proposition \ref{prop:gaussian_concentration_inequality} with $u_3={\bf P}(x)-(128\alpha(1+\varepsilon)\mu|A|)/(\delta^2_j(\mu-\psi(\mu)))$, it follows that
\begin{multline*}
\bP\left(\sum_{a\in A}\sqrt{\frac{2\alpha(1+\varepsilon)\hat{p}_n(a|x)\mu}{(\mu-\psi(\mu))\bar{N}_{n}(x)}}>\delta_j/16\right)\leq\\
\leq
\exp\left\{-\frac{\Delta^2(n-d)^2}{2n(|S|+1)^2}\left({\bf P}(x)-\frac{128|A|\alpha\mu(1+\varepsilon)}{\delta^2_j(\mu-\psi(\mu))(n-d)}\right)^2\right\},
\end{multline*}
whenever
$(n-d)>(128|A|\alpha\mu(1+\varepsilon))/(\delta^2_j{\bf P}(x)(\mu-\psi(\mu)))$.

Therefore, we have shown  that for any $x\in A^{S}$,
\begin{equation}
\label{proof_cor_thm_SSP2_key_ineq3}
\bP\left(s_n(x)>\delta_j/4\right)\leq \\
3\exp\left\{-\frac{\Delta^2(n-d)^2{\bf P}^2(x)}{8n(|S|+1)^2}\right\},
\end{equation}
as long as
$$
(n-d)\geq 2\left(16\frac{|A|\alpha}{\delta^2_j{\bf P}(x)}\left(\frac{8\mu(1+\varepsilon)}{(\mu-\psi(\mu))}+\sqrt{\frac{(1+\varepsilon)}{2(\mu-\psi(\mu))}}\right)\right).
$$

Now, considering $b^{*,j},c^{*,j}\in A$ with $b^*\neq c^*$ and $(x^{*,j},y^{*,j})\in\cC_j(a^{*,j},b^{*,j})$ such that
$$\min_{b,c\in A:b\neq c}\max_{(x,y)\in\cC_j(a,b)}({\bf P}(x)\wedge{\bf P}(y))=({\bf P}(x^{*,j})\wedge{\bf P}(y^{*,j})),$$
we can deduce from \eqref{prof_thm_PCP_estimator:ineq_1} and \eqref{proof_cor_thm_SSP2_key_ineq3} that 
$$
\bP\left(t_{n,j}(b,c)>\delta_j/2 \right)\leq 6\exp\left\{-\frac{\Delta^2(n-d)^2({\bf P}(x^{*,j})\wedge{\bf P}(y^{*,j}))^2}{8n(|S|+1)^2}\right\},
$$
whenever
$$
(n-d)\geq 2\left(16\frac{|A|\alpha}{\delta^2_j{\bf P}(x^{*,j})\wedge {\bf P}(y^{*,j})}\left(\frac{8\mu(1+\varepsilon)}{(\mu-\psi(\mu))}+\sqrt{\frac{(1+\varepsilon)}{2(\mu-\psi(\mu))}}\right)\right).
$$
Since ${\bf P}(x^{*,j})\wedge {\bf P}(y^{*,j})\geq {\bf P}_S$ for all $j\in\Lambda^*$, we can take 
$$C=C(\mu,\varepsilon)=32\left(\frac{8\mu(1+\varepsilon)}{(\mu-\psi(\mu))}+\sqrt{\frac{(1+\varepsilon)}{2(\mu-\psi(\mu))}}\right),$$
to deduce that \eqref{proof_cor_thm_SSP2_key_ineq} is indeed satisfied whenever
$$
(n-d)\geq \frac{C|A|\alpha}{\delta^2_{min}{\bf P}_S},
$$
concluding the proof.
\end{proof}

\subsubsection{Proof of Corollary \ref{cor:consistency_PCP_estimator}}
\label{proof:cor_conist_PCP_Esti}

\begin{proof}[Proof of Corollary \ref{cor:consistency_PCP_estimator}]
Notice that Assumptions \ref{Ass:MTD_stationary} and \ref{Ass:Concentration_inequalities} are satisfied for all values of $n$, since $p_{min,n}\geq p^{\star}_{min}$ and $\Delta_n\geq \Delta^{\star}_{min}$ for positive constants  $p^{\star}_{min}$ and $\Delta^{\star}_{min}$. Hence, the result follows immediately from Theorem \ref{thm:consistency_PCP_estimator}-Item 3 and Remark \ref{rmk:PCP_consistency}-Item (c).
\end{proof}

\subsection{Proofs of Section \ref{Sec:FSC_estimator}}
\label{Sec:proof_prop_struc_result}
\subsubsection{Proof of Proposition \ref{Prop:structure_result}}
In this section we prove Proposition \ref{Prop:structure_result}. To that end, we need some auxiliary results. The first auxiliary result is the following.
Recall that we write $\mCov_{x_{S}}(X_{0},m_k(X_{k}))$ to denote the conditional covariance between the random variables $X_{0}$ and $m_k(X_{k})$ given that $X_{S}=x_{S}$, where $m_k$ is defined \eqref{def:Conditional_average_pj_given_b}.

\begin{lemma}
\label{Prop:conditional_cov}
For each $S\subseteq \llb -d, -1\rrb$, $k\in S^c$ and $x_{S}\in A^{S}$, the following identity holds:
\begin{equation}
\label{identify_conditional_cov}
\mCov_{x_{S}}(X_0,m_k(X_{k}))=\sum_{j\in\Lambda\setminus S}\lambda_j\mCov_{x_{S}}(m_j(X_{j}),m_k(X_{k})).    
\end{equation}
\end{lemma}
\begin{remark}
In \eqref{identify_conditional_cov}, we use the convention that $\sum_{j\in \emptyset}\lambda_j\mCov_{x_{S}}(m_j(X_{j}),m_k(X_{k}))=0$.
\end{remark}

\begin{proof}[Proof of Lemma \ref{Prop:conditional_cov}]
\label{proof_conditional_cov}
 
Observe that if $\Lambda\subseteq S$, then the both sides of \eqref{identify_conditional_cov} are $0$, so that the result holds immediately in this case.

Now suppose that $\Lambda\not\subseteq S$. In this case, to shorten the notation, let us write
\[
{\bf P}_{x_{S}}(x_{\Lambda\setminus S})=\bP_{x_{S}}(X_{\Lambda\setminus S}=x_{\Lambda\setminus S}), \ \text{for} \  x_{\Lambda\setminus S}\in A^{\Lambda\setminus S}. 
\]

We want to compute $$\mCov_{x_{S}}(X_0,m_k(X_{k}))=\bE_{x_{S}}(X_0 m_k(X_{k}))-\bE_{x_{S}}(X_0)\bE_{x_{S}}(m_k(X_{k})).$$ 
We first compute $\bE_{x_{S}}(X_0)$. 
To that end, write
\begin{equation*}
%\label{Condi_expc_X0}
\bE_{x_{S}}(X_0)=\sum_{a\in A}a\bP_{x_{S}}(X_0=a),
\end{equation*}
and observe that for each $a\in A$,
\begin{align*}
\bP_{x_{S}}(X_0=a)&=\sum_{x_{\Lambda\setminus S}\in A^{\Lambda\setminus S}}{\bf P}_{x_{S}}(x_{\Lambda\setminus S})p(a|x_{S}x_{\Lambda\setminus S})\nonumber\\
&=\lambda_0p_0(a)+\sum_{j\in\Lambda\cap S}\lambda_jp_j(a|x_{j}) +\sum_{j\in \Lambda\setminus S}\lambda_j\sum_{x_{\Lambda\setminus S}\in A^{\Lambda\setminus S}}{\bf P}_{x_{S}}(x_{\Lambda\setminus S})p_j(a|x_{j})\nonumber\\
&=\lambda_0p_0(a)+\sum_{j\in\Lambda\cap S}\lambda_jp_j(a|x_{j}) +\sum_{j\in \Lambda\setminus S}\lambda_j\bE_{x_{S}}(p_j(a|X_{j}))\nonumber,
\end{align*}
where in the second equality we have used the definition of the transition probabilities \eqref{def:transition_probabilities}. 
Hence, we have that
$$
\bE_{x_{S}}(X_0)=\lambda_0m_0+\sum_{j\in \Lambda\cap S}\lambda_jm_j(x_{j})+\sum_{j\in\Lambda\setminus S}\lambda_j\bE_{x_{S}}(m_j(X_{j})),
$$
where $m_0=\sum_{a\in A}ap_0(a)$.

As a consequence of the above equality, it follows that
\begin{multline}
\label{Id_1_proof_cov_condicional}
\bE_{x_{S}}(X_0)\bE_{x_{S}}(m_k(X_{k}))=\left [\lambda_0m_0+\sum_{j\in\Lambda\cap S}\lambda_jm_j(x_{j}) \right]\bE_{x_{S}}(m_k(X_{k}))\\+\sum_{j\in\Lambda\setminus S}\lambda_j\bE_{x_{S}}(m_j(X_{j}))\bE_{x_{S}}(m_k(X_{k})).   
\end{multline}
We now compute $\bE_{x_{S}}(X_0m_k(X_{k})).$ We consider only the case $k\in\Lambda$, the other is treated similarly. In this case, we first write
\begin{equation}
\label{Id_2_proof_cov_condicional}
\bE_{x_{S}}(X_0m_k(X_{k}))=\sum_{a\in A}\sum_{b\in A}am_k(b)\bP_{x_{S}}(X_0=a,X_{k}=b),
\end{equation}
and then we proceed similar as above to deduce that for each $a,b\in A$,
\begin{multline}
\label{Id_3_proof_cov_condicional}
\bP_{x_{S}}(X_0=a,X_{k}=b)=\sum_{x_{\Lambda\setminus S}\in A^{\Lambda\setminus S}}{\bf P}_{x_{S}}(x_{\Lambda\setminus S})p(a|x_{S}x_{\Lambda\setminus S})1\{x_{k}=b\}\\  =\left [\lambda_0p_0(a)+\sum_{j\in\Lambda\cap S}\lambda_jp_j(a|x_{j}) +\lambda_kp_k(a|b)\right]\bP_{x_{S}}(X_{k}=b)\\
+\sum_{j\in\Lambda\setminus (S\cup\{k\})}\lambda_j\sum_{c\in A}p_j(a|c)\bP_{x_{S}}(X_{j}=c,X_{k}=b).
%=\left [\lambda_0p_0(1)+\sum_{j\in\Lambda\cap S}\lambda_jp_j(1|x_{-j}) +\Delta_k+\sum_{j\in\Lambda\setminus S}\lambda_jp_j(1|0)\right]\bE_{x_{S}}(X_{-k})\\
%+\sum_{j\in\Lambda\setminus (S\cup\{k\})}\Delta_j\bE_{x_{S}}(X_{-j}X_{-k})\\
%=\left [\lambda_0p_0(1)+\sum_{j\in\Lambda\cap S}\lambda_jp_j(1|x_{-j}) +\sum_{j\in\Lambda\setminus S}\lambda_jp_j(1|0)\right]\bE_{x_{S}}(X_{-k})\\
%+\sum_{j\in\Lambda\setminus S}\Delta_j\bE_{x_{S}}(X_{-j}X_{-k})
\end{multline}
where in the second equality we have used the definition of the transition probabilities \eqref{def:transition_probabilities}.
Combining \eqref{Id_2_proof_cov_condicional} and \eqref{Id_3_proof_cov_condicional}, we deduce that
\begin{multline}
\label{Id_4_proof_cov_condicional}
\bE_{x_{S}}(X_0m_k(X_{k}))=\left [\lambda_0m_0+\sum_{j\in\Lambda\cap S}\lambda_jm_j(x_{j}) \right]\bE_{x_{S}}(m_k(X_{k}))+\lambda_k\bE_{x_{S}}(m^2_k(X_{k}))\\+\sum_{j\in\Lambda\setminus (S\cup\{k\})}\lambda_j\bE_{x_{S}}(m_k(X_{k})m_j(X_{j}))\\
=\left [\lambda_0m_0+\sum_{j\in\Lambda\cap S}\lambda_jm_j(x_{j}) \right]\bE_{x_{S}}(m_k(X_{k}))\\+\sum_{j\in\Lambda\setminus S}\lambda_j\bE_{x_{S}}(m_k(X_{k})m_j(X_{j})).
\end{multline}
%and in the last inequality we have used the fact that $X^2_{-k}=X_{-k}.$
Putting together the identities \eqref{Id_1_proof_cov_condicional}
and \eqref{Id_4_proof_cov_condicional}, we then conclude that
\begin{multline}
\mCov_{x_{S}}(X_0m_k(X_{k}))=\sum_{j\in\Lambda\setminus S}\lambda_j\left(\bE_{x_{S}}(m_j(X_{j})m_k(X_{k}))\right.\\ \left. -\bE_{x_{S}}(m_j(X_{j}))\bE_{x_{S}}(m_k(X_{k})\right),
\end{multline}
and the result follows.
\end{proof}

The next auxiliary result is the following.

\begin{lemma}
\label{key_lemma}
Suppose Assumptions \ref{Ass:MTD_stationary} and \ref{Ass:non_determinist_conditional_averages} hold.
%, and let ${\bf X}$ be a stationary MTD model with invariant distribution $\pi$.
Then there exists a constant $\kappa'>0$ such that the following property holds: for any $S\subseteq \llb -d, -1\rrb$ such that $\Lambda\not\subseteq S$, we have
$$
\sum_{j\in \Lambda\setminus S}\sum_{k\in \Lambda\setminus S}\lambda_j\lambda_k\bE(\mCov_{X_{S}}(m_j(X_{j}),m_k(X_{k})))\geq \kappa'. 
$$
\end{lemma}

\begin{proof}[Proof of Lemma \ref{key_lemma}]
\label{proof_key_lemma}
It suffices to show that for $S\subseteq \llb -d, -1\rrb$ such that $\Lambda\not\subseteq S$, we have
$$
\sum_{j\in \Lambda\setminus S}\sum_{k\in \Lambda\setminus S}\lambda_j\lambda_k\bE(\mCov_{x_{S}}(m_j(X_{j}),m_k(X_{k})))>0.
$$
Suppose that this is not the case. Then,  
\begin{align*}
0&=\sum_{j\in \Lambda\setminus S}\sum_{k\in \Lambda\setminus S}\lambda_j\lambda_k\bE(\mCov_{x_{S}}(m_j(X_{j}),m_k(X_{k})))\\
&=\sum_{j\in \Lambda\setminus S}\sum_{k\in \Lambda\setminus S}\lambda_j\lambda_k\mCov\left(m_j(X_{j})-\bE_{X_{S}}(m_j(X_{j})),m_k(X_{k})-\bE_{X_{S}}(m_k(X_{k}))\right)\\
&=\text{Var}\left(\sum_{j\in \Lambda\setminus S}\lambda_j(m_j(X_{j})-\bE_{X_{S}}(m_j(X_{j})))\right),  
\end{align*}
so that $\bP$-almost surely,
$$
\sum_{j\in \Lambda\setminus S}\lambda_j(m_j(X_{j})-\bE_{X_{S}}(m_j(X_{j})))=0.
$$
This implies that $\bP$-almost surely,
$$
\sum_{j\in \Lambda\setminus S}\lambda_jm_j(X_{j})=\bE_{X_{S}}\left(\sum_{j\in \Lambda\setminus S}\lambda_jm_j(X_{j})\right),
$$
or equivalently, 
$$
\sum_{j\in \Lambda\setminus S}\lambda_jm_j(X_{j})=f(X_{S}), \ \bP\text{-a.s.},
$$
for some function $f:A^{S}\to \bR.$ 

Now take any configuration $x_{S}\in A^{S}$ and consider the event $A=\{X_{S}=x_{S}\}$. From the above identity, it follows that $\bP\text{-a.s.}$,
$$
1_{A}\sum_{j\in \Lambda\setminus S}\lambda_jm_j(X_{j})=1_{A}f(x_{S}).
$$
Finally, take any configuration $x_{\Lambda\setminus S}\in A^{\Lambda\setminus S}$ such that
$$
\sum_{j\in \Lambda\setminus S}\lambda_jm_j(x_{j})\neq f(x_{S}),
$$
and let $B=\{X_{\Lambda\setminus S}=x_{\Lambda\setminus S}\}$.
Such a configuration must exist by Assumption \ref{Ass:non_determinist_conditional_averages}.
%As a consequence, we have $\bP$-almost surely
%$$
%\sum_{a\in A}ap_{\Lambda}(a|x_{\Lambda})=\lambda_{0,\Lambda}m_{0,\Lambda}+\sum_{k\in %\Lambda}\lambda_km_k(X_{-k})=g(x_{S}),
%$$
%for some function $g:\{0,1\}^{-S}\to \bR.$
As a consequence, we have that
$\bP\text{-a.s.}$,
$$
1_{A\cap B}\sum_{j\in \Lambda\setminus S}\lambda_jm_j(x_{j})=1_{A\cap B}f(x_{S}),
$$
implying that 
$$
\bP(A\cap B)\sum_{j\in \Lambda\setminus S}\lambda_jm_j(x_{j})=\bP(A\cap B)f(x_{S}).
$$
By Assumption \ref{Ass:MTD_stationary}, we have $\bP(A\cap B)={\bf P}(x_{\Lambda})>0$ so that the identify above would imply that
$$
\sum_{j\in \Lambda\setminus S}\lambda_jm_j(x_{j})=f(x_{S}),
$$
which is a contradiction. Therefore, we must have $$\text{Var}\left(\sum_{j\in \Lambda\setminus S}\lambda_j(m_j(X_{j})-\bE_{X_{S}}(m_j(X_{j})))\right)>0,$$
and the result follows.
\end{proof}

We also need the following result.

\begin{lemma}
\label{lema:cov_X0_mkXk_upper_bound}
For each $S\subseteq \llb -d, -1\rrb$, $k\in S^c$ and $x_{S}\in A^{S}$, the following identity holds:
\begin{equation}
\label{cov_X0_mkXk_upper_bound}
|\mCov_{x_{S}}(X_0,m_k(X_{k}))|\leq \|m_k\|_{Lip}|\mCov_{x_{S}}(X_0,X_{k})|,
\end{equation}
where $m_k$ and $\|m_k\|_{Lip}$ are defined \eqref{def:Conditional_average_pj_given_b} and \eqref{def:delta_min_and_} respectively.
\end{lemma}
\begin{proof}[Proof of Lemma \ref{lema:cov_X0_mkXk_upper_bound}]
First observe that $\mCov_{x_{S}}(X_0,m_k(X_{k}))=\mCov_{x_{S}}(X_0,m_k(X_{k})-m_k(c))$, for any $c\in A$. Since,
$$
\mCov_{x_{S}}(X_0,m_k(X_{k})-m_k(c))=\sum_{b\in A}(m_k(b)-m_k(c))\mCov_{x_{S}}(X_0,1\{X_k=b\})
$$
and $|m_k(b)-m_k(c)|\leq \|m_k\|_{Lip}|b-c|$, it follows then that
$$
\mCov_{x_{S}}(X_0,m_k(X_{k}))\leq \|m_k\|_{Lip}\sum_{b\in A}|b-c|\mCov_{x_{S}}(X_0,1\{X_k=b\}). 
$$
By taking $c=\min(A)$, we have that $|b-c|=(b-c)$ for any $b\in A$ and we deduce from the above inequality that
$$
\mCov_{x_{S}}(X_0,m_k(X_{k}))\leq \|m_k\|_{Lip}\mCov_{x_{S}}(X_0,X_k)\leq \|m_k\|_{Lip}|\mCov_{x_{S}}(X_0,X_k)|.  
$$
A similar argument shows that $\mCov_{x_{S}}(X_0,m_k(X_{k}))\geq  -\|m_k\|_{Lip}|\mCov_{x_{S}}(X_0,X_k)|$, concluding the proof.
\end{proof}

Our last auxiliary result is the following.

\begin{lemma}
\label{lema:conditional_cov_explicity_expression}
For each $S\subseteq \llb -d, -1\rrb$ such $\Lambda\not\subset S$, $x_{S}\in A^{S}$ and $j,k\in \Lambda\setminus S$, the following identity holds:
\begin{multline}
\label{identify_conditional_cov_explicity_expression}
\mCov_{x_{S}}(m_j(X_{j}),m_k(X_{k}))=\frac{1}{2}\sum_{b\in A}\sum_{c\in A}\bP_{x_S}(X_k=b)\bP_{x_S}(X_k=c)(m_k(b)-m_k(c))\times\\
\left(\bE_{x_S}(m_j(X_j)|X_k=b)-\bE_{x_S}(m_j(X_j)|X_k=c)\right).
\end{multline}
where $m_j$ is defined \eqref{def:Conditional_average_pj_given_b}.
\end{lemma}

\begin{proof}[Proof of Lemma \ref{lema:conditional_cov_explicity_expression}]
First notice that
$$
\mCov_{x_{S}}(m_j(X_{j}),m_k(X_{k}))=\sum_{a,b\in A}m_j(a)m_k(b)\mCov_{x_{S}}(1\{X_{j}=a\},1\{X_{k}=b\}).
$$
Now, for any $a,b\in A$, one can check that
\begin{multline*}
\mCov_{x_{S}}(1\{X_{j}=a\},1\{X_{k}=b\})=\sum_{c\in A}\bP_{x_S}(X_k=b)\bP_{x_S}(X_k=c)\\ \times(\bP_{x_S}(X_j=a|X_k=b)-\bP_{x_S}(X_j=a|X_k=c))    
\end{multline*}
Hence, we deduce from the above equalities that
\begin{multline*}
\mCov_{x_{S}}(m_j(X_{j}),m_k(X_{k}))=\sum_{b\in A}\sum_{c\in A}m_k(b)\bP_{x_S}(X_k=b)\bP_{x_S}(X_k=c))\\ \times(\bE_{x_S}(m_j(X_j)|X_k=b)-\bE_{x_S}(m_j(X_j)|X_k=c)).       
\end{multline*}
Interchanging the role of the symbols $b$ and $c$ in the equality above, we obtain that
\begin{multline*}
\mCov_{x_{S}}(m_j(X_{j}),m_k(X_{k}))=-\sum_{c\in A}\sum_{b\in A}m_k(c)\bP_{x_S}(X_k=b)\bP_{x_S}(X_k=c))\\ \times(\bE_{x_S}(m_j(X_j)|X_k=b)-\bE_{x_S}(m_j(X_j)|X_k=c)).       
\end{multline*}
The result follows by combing the last two equalities above.
\end{proof}

We now prove Proposition \ref{Prop:structure_result}.

\begin{proof}[Proof of Proposition \ref{Prop:structure_result}]

We first prove inequality \eqref{lower_bound_bar_nu}. Let us denote $D_{k,S}(a,b,c,x_{S})=\bP_{x_{S}}(X_0=a|X_{k}=b)-\bP_{x_{S}}(X_0=a|X_{k}=c)$, for each $a,b,c\in A$, $x_{S}\in A^{S}$ and $k\notin S$.
With this notation, one can check that for any $x_{S}\in A^{S}$ and $k\notin S$, we have that 
\begin{equation}
\label{identity_proof_cor_structural_result_2}
\mCov_{x_{S}}(X_0,X_{k})=
%\sum_{b\in A}m_k(b)\bP_{x_{S}}(X_{-k}=b)\left[\sum_{a\in A}a\left(\bP_{x_{S}}(X_0=a|X_{-k}=b)-\bP_{x_{S}}(X_0=a)\right)\right]\\
%\sum_{b\in A}\sum_{c\in %A}\bP_{x_{S}}(X_{-k}=c)m_k(b)\bP_{x_{S}}(X_{-k}=b)\sum_{a\in A}aD_{k,S}(a,b,c,x_{S})\\
%&=
\frac{1}{2}\sum_{b\in A}\sum_{c\in A}(b-c)w_{k,S}(b,c,x_{S})\sum_{a\in A}aD_{k,S}(a,b,c,x_{S}).
\end{equation}
Now, observe that the triangle inequality
and the equality $$\frac{1}{2}\sum_{a\in A}|D_{k,S}(a,b,c,x_{S})|=d_{k,S}(b,c,x_{S}),$$
imply that
$$
|\mCov_{x_{S}}(X_0,X_{k})|
\leq Diam(A)\|A\|_{\infty}\sum_{b\in A}\sum_{c\in A}w_{k,S}(b,c,x_{S})d_{k,S}(b,c,x_{S}),
$$
so that
$$
\bE\left(|\mCov_{X_{S}}(X_0,X_{k})|\right)\leq Diam(A)\|A\|_{\infty}\bar{\nu}_{k,S},
$$
proving inequality \eqref{lower_bound_bar_nu}.

We now prove \eqref{Key_covariance_inequality}. This is done as follows. In the sequel, we shall write $\lambda 1_{\Lambda\setminus S}$ to denote the vector $\lambda=(\lambda_j)_{j\in\Lambda}$ restricted to the coordinates in $\Lambda\setminus S$: $\lambda 1_{\Lambda\setminus S}=(\lambda_j)_{j\in\Lambda\setminus S}$. With this notation, it follows from Lemma \ref{Prop:conditional_cov} and Lemma \ref{key_lemma}
that for any $S\subseteq \llb -d, -1\rrb$ such that $\Lambda\not\subseteq S$, 
%and any $x_{S}\in \{0,1\}^{-S}$,
$$
\sum_{k\in\Lambda\setminus S}\lambda_k\bE\left(\mCov_{X_{S}}(X_0,m_k(X_{k}))\right)\geq \kappa'.
$$
By the triangle inequality, it then follows that
$$
\sum_{k\in\Lambda\setminus S}\lambda_k\left|\bE\left(\mCov_{X_{S}}(X_0,m_k(X_{k}))\right|\right)\geq \kappa'.
$$
Now using that
$$
\max_{k\in\Lambda\setminus S} \left|\bE\left(\mCov_{X_{S}}(X_0,m_k(X_{k}))\right)\right|\|\lambda1_{\Lambda\setminus S}\|_1\geq \sum_{k\in\Lambda\setminus S}\lambda_k\left|\bE\left(\mCov_{X_{S}}(X_0,m_k(X_{k}))\right)\right|,
$$
we conclude that
$$
\max_{k\in\Lambda\setminus S} \left|\bE\left(\mCov_{X_{S}}(X_0,m_k(X_{k}))\right)\right|\|\lambda1_{\Lambda\setminus S}\|_1\geq \kappa'. 
$$
By observing that  $1-\lambda_0=\sum_{k\in \Lambda}\lambda_k\geq \|\lambda1_{\Lambda\setminus S}\|_1$,
%By Cauchy-Schwartz inequality, it follows then that
%$$
%\sqrt{|\Lambda\setminus S|\sum_{k\in\Lambda\setminus S}\lambda^2_k\mCov^2(X_0,X_{-k})}\geq (\lambda 1_{\Lambda\setminus S})^T\Sigma_{\Lambda\setminus S}(\lambda 1_{\Lambda\setminus S}).
%$$
%Now, Jensen inequality imply that
%$$
%\sqrt{|\Lambda\setminus S|\sum_{k\in\Lambda\setminus %S}\lambda^2_k\bE\left(\mCov^2_{x_{S}}(X_0,X_{-k})\right)}\geq %(\lambda 1_{\Lambda\setminus S})^T\Sigma_{\Lambda\setminus %S}(\lambda 1_{\Lambda\setminus S}).
%$$
%Now, by using that
%\begin{align*}
%(\lambda 1_{\Lambda\setminus S})^T\Sigma_{\Lambda\setminus S}(\lambda 1_{\Lambda\setminus S})&\geq \left(\min_{v\in R^{|\Lambda\setminus S|}:\|v\|_2=1}v^T\Sigma_{\Lambda\setminus S}v\right)\|(\lambda 1_{\Lambda\setminus S})\|_2\\
%& \geq \left(\min_{v\in R^{|\Lambda|}:\|v\|_2=1}v^T\Sigma v\right)\|(\lambda 1_{\Lambda\setminus S})\|_2=\kappa \|(\lambda 1_{\Lambda\setminus S})\|_2,
%\end{align*}
we conclude from the above inequality that
$$
\max_{k\in\Lambda\setminus S} \left|\bE\left(\mCov_{X_{S}}(X_0,m_k(X_{k}))\right)\right|\geq \kappa'/(1-\lambda_0)>0,
$$
and the result follows from Lemma \ref{lema:cov_X0_mkXk_upper_bound}.

Therefore, it remains to prove \eqref{lower_bound_on_kappa}. To that end, we first use Lemma \ref{Prop:conditional_cov}, Lemma \eqref{lema:cov_X0_mkXk_upper_bound}  and Lemma \ref{lema:conditional_cov_explicity_expression} to obtain that
\begin{multline*}
\sum_{k\in\Lambda\setminus S}\lambda_k\|m_k\|_{Lip}|\mCov_{x_{S}}(X_0,X_{k})|\geq\frac{1}{2} \sum_{k\in\Lambda\setminus S}\sum_{j\in\Lambda\setminus S}\sum_{b\in A}\sum_{c\in A}\lambda_k\lambda_j\bP_{x_S}(X_k=b)\bP_{x_S}(X_k=c)\\\times(m_k(b)-m_k(c))
\left(\bE_{x_S}(m_j(X_j)|X_k=b)-\bE_{x_S}(m_j(X_j)|X_k=c)\right).    
\end{multline*}
Next, we observe that Assumption \ref{ass:lower_bound_kappa} implies that
\begin{equation*}
 \left(1-\sum_{j\in\Lambda\setminus S:j\neq k}\frac{\lambda_j\left(\left|\bE_{x_S}(m_j(X_j)|X_k=b)-\bE_{x_S}(m_j(X_j)|X_k=c)\right|\right)}{\lambda_k|m_k(b)-m_k(c)|}
\right)\geq \Gamma_1,   
\end{equation*}
so that
$$
\sum_{k\in\Lambda\setminus S}\lambda_k\|m_k\|_{Lip}\bE\left(|\mCov_{x_{S}}(X_0,X_{k})|\right)\geq \frac{p^{2}_{min}\Gamma_1}{2}\sum_{k\in\Lambda\setminus S}\lambda^2_k\sum_{b\in A}\sum_{c\in A}(m_k(b)-m_k(c))^2,
$$
where we have also used that $\bP_{x_S}(X_k=b)\geq p_{min}$.
Finally, note that $|m_k(b)-m_k(c)|\geq \min\{|b-c|^2:b\neq c\}\|m_k\|^2_{Lip}$ to obtain that
$$
\max_{k\in\Lambda\setminus S}\bE\left(|\mCov_{x_{S}}(X_0,X_{k})|\right)\sum_{k\in\Lambda\setminus S}\lambda_k\|m_k\|_{Lip}\geq \frac{p^{2}_{min}\Gamma_1}{2\min\{|b-c|^2:b\neq c\}}\sum_{k\in\Lambda\setminus S}\lambda^2_k\|m_k\|^2_{Lip}.
$$
Then, by using Cauchy-Schwartz inequality, we deduce that
$$
\sum_{k\in\Lambda\setminus S}\lambda_k\|m_k\|_{Lip}\leq \sqrt{\sum_{k\in\Lambda\setminus S}\lambda^2_k\|m_k\|^2_{Lip}}\sqrt{|\Lambda\setminus S|}\leq \sqrt{\sum_{k\in\Lambda\setminus S}\lambda^2_k\|m_k\|^2_{Lip}}\sqrt{|\Lambda|}.
$$
The result follows by combining the last two inequalities. 
\end{proof}

\subsubsection{Proof of Theorem \ref{Prop:Correctness_1_FSC_Algo}}
\label{sec:proff_correctnes_1_FSC_algo}

Before starting the proof of Theorem \ref{Prop:Correctness_1_FSC_Algo}, we recall some definitions from Information Theory.
In what follows, for $S\in \llb -d, -1\rrb$ and $j\in \llb -d, -1\rrb$, we write $I(X_0;X_{j}| X_{S})$ to denote the conditional mutual information between $X_0$ and $X_{j}$ given $X_{S}$, defined as
\begin{equation}
\label{def:conditional_mutual_information}
I(X_0;X_{j}| X_{S})=\sum_{x_{S}\in A^{S}}{\bf P}(x_{S})I(X_0;X_{j}| X_{S}=x_{S}),
\end{equation}
where $I(X_0;X_{j}| X_{S}=x_{S}):=I_j(x_{S})$ denotes the conditional mutual information between $X_0$ and $X_{j}$ given $X_{S}=x_{S}$, defined as
\begin{equation}
\label{def:conditional_mutual_information_given_x_s}
I_j(x_{S})=\sum_{a,b\in A}\bP_{x_{S}}(X_0=a,X_{j}=b)\log\left(\frac{\bP_{x_{S}}(X_0=a,X_{j}=b)}{\bP_{x_{S}}(X_0=a)\bP_{x_{S}}(X_{j}=b)}\right).    
\end{equation}
We use the convention that when $S=\emptyset$, the conditional probability $\bP_{x_{s}}$ is the unconditional probability $\bP$. Hence, in this case, the conditional mutual information between $X_0$ and $X_{j}$ is the mutual information between these random variables, denoted $I(X_0;X_{j}):=I_j.$

The entropy $H(X_0)$ of $X_0$ is defined as 
\begin{equation}
H(X_0)=-\sum_{a\in A}\bP(X_0=a)\log(\bP(X_0=a)).    
\end{equation}

To prove Theorem  \ref{Prop:Correctness_1_FSC_Algo} we proceed similarly to \cite{Bresler:15}. During the proof we will need the
the following lemma.

\begin{lemma}
\label{key_lemma_Bresler}
Suppose that the event $G_m(\xi,\ell)$ defined in \eqref{def:Good_event}  holds and let $S\subseteq \llb -d, -1\rrb$ with $|S|\leq \ell$. If $\hat{\nu}_{m,k,S}\geq \tau$ with $k\in S^c$, then  $I(X_0;X_{k}|X_{S})\geq 2(\tau-\xi)^2$.
\end{lemma}
\begin{proof}
Definition \eqref{def:conditional_mutual_information} together with Jensen inequality implies that for any $j\in S^c$, 
\begin{align*}
\sqrt{\frac{1}{2}I(X_0;X_{j}|X_{S})}&=\sqrt{\frac{1}{2}\sum_{x_{S}\in A^{S}}{\bf P}(x_{S})I_j(x_{S})}\\
&\geq \sum_{x_{S}\in A^{S}}{\bf P}(x_{S})\sqrt{\frac{1}{2}I_j(x_{S})}.
\end{align*}
By Pinsker inequality, it then follows that
\begin{align*}
\sqrt{\frac{1}{2}I_j(x_S)}
&\geq \frac{1}{2}\sum_{a,b\in A}|\bP_{x_S}(X_0=a,X_{j}=b)-\bP_{x_S}(X_0=a)\bP_{x_{S}}(X_{j}=b)|\\
&=\sum_{b\in A}\bP_{x_S}(X_{j}=b)\frac{1}{2}\sum_{a\in A}|\bP_{x_S}(X_0=a|X_{j}=b)-\bP_{x_S}(X_0=a)|\\
%=\sum_{b\in A}\sum_{c\in A}\bP_{x_{S}}(X_{-j}=b)\bP_{x_{S}}(X_{-j}=c)\frac{1}{2}\sum_{a\in A}|\bP_{x_{S}}(X_0=a|X_{-j}=b)-\bP_{x_{S}}(X_0=a|X_{-j}=c)|\\
&=\sum_{b\in A}\sum_{c\in A}w_{j,S}(b,c,x_S)d_{TV}(\bP_{x_S}(X_0\in\cdot|X_{j}=b),\bP_{x_S}(X_0\in\cdot|X_{j}=c))\\
&=\nu_{j,S}(x_{S}),
\end{align*}
where in the second equality we have used that for any $a,b\in A$,
$$
\bP_{x_S}(X_0=a|X_{j}=b)=\sum_{c\in A}\bP_{x_S}(X_{j}=c)\bP_{x_S}(X_0=a|X_{j}=b),
$$
and also that
$$
\bP_{x_S}(X_0=a)=\sum_{c\in A}\bP_{x_S}(X_{j}=c)\bP_{x_S}(X_0=a|X_{j}=c)).
$$

As a consequence, we deduce that
$$
\sqrt{\frac{1}{2}I(X_0;X_{j}|X_{S})}\geq \sum_{x_S\in A^{S}}{\bf P}(x_S)\nu_{j,S}(x_S)=\bar{\nu}_{j,S}.
$$
Now, on the event $G_m(\xi,\ell)$, we have that $\bar{\nu}_{k,S}\geq \hat{\nu}_{m,k,S}-\xi$
so that
$$
\sqrt{\frac{1}{2}I(X_0;X_{k}|X_{S})}\geq \hat{\nu}_{m,k,S}-\xi\geq\tau-\xi,
$$
where in the rightmost inequality we have used that $\hat{\nu}_{m,k,S}\geq\tau$.
Hence,
$$
I(X_0;X_{j}|X_{S})\geq 2(\tau-\xi)^2,
$$
and the result follows.
\end{proof}

We now prove Theorem \ref{Prop:Correctness_1_FSC_Algo}.

\begin{proof}
Suppose the event $G_m(\xi_*,\ell_*)$ holds and let $\hat{S}_m$ be the set obtained at the end of FS step of Algorithm \ref{def:FSC_algo} with parameter $\ell_*$, where the parameters $\xi_*$ and $\ell_*$  are defined as in \eqref{parameters_consistency}. 
In the sequel, let $S_0=\emptyset$ and $S_k=S_{k-1}\cup\{j_k\}$, where $j_k\in \arg\max_{j\in S^c_{k-1}}\hat{\nu}_{m,j,S_{k-1}}$ for $1\leq k\leq d$, and observe that by construction $\hat{S}_m=S_{\ell_*}$.
We want to show that $\Lambda\subseteq \hat{S}_m$. We argue by contraction. Suppose that $\Lambda$ is not contained in $\hat{S}_{m}$. In this case, it follows that $\Lambda\not\subseteq S_{k}$ for all $1\leq k\leq \ell_*$, and  Proposition \ref{Prop:structure_result}
implies that for all $1\leq k\leq \ell_*,$
$$
\max_{j\in S^c_k}\bar{\nu}_{j,S_k}\geq \frac{\kappa}{\|A\|_{\infty}Diam(A)}=4\xi_*,
$$
where the equality holds by the choice of $\xi_*$.
Since the event $G_m(\xi_*,\ell_*)$ holds and $|S_k|\leq |S_{\ell_*}|=\ell_*$, it follows from the above inequality that
$$
\hat{\nu}_{m,j_{k},S_{k-1}}=\max_{j\in S^c_{k-1}}\hat{\nu}_{m,j,S_{k-1}}\geq 3\xi_*,
$$
for all $1\leq k\leq \ell_*+1.$
By Lemma \ref{key_lemma_Bresler}, we then deduce that $I(X_0,X_{j_k}|X_{S_{k-1}})\geq 8\xi^2_*$ for all $1\leq k\leq \ell_*+1.$

Now, notice that 
\begin{align}
\label{ineq_1_proof_claim_1}
\log_2(|A|)\geq H(X_0)\geq I(X_0;X_{\hat{S}_{m}\cup\{j_{\ell_*+1}\}})=\sum_{k=1}^{\ell_*+1}I(X_0,X_{j_{k}}|X_{S_{k-1}}),
\end{align}
where  we have used Gibbs inequality in the first passage, the fact that the entropy is always larger than the mutual information in the second passage and the Chain Rule in the last passage. The proof of these facts can be found, for instance, in \citep{Cover_Thomas_06}. 

By the choice of $\ell_*=\left\lfloor\log_2(|A|)/8\xi_*^2\right\rfloor$, we have that $\ell_*+1>\log_2(|A|)/8\xi_*^2$ so that it follows from \eqref{ineq_1_proof_claim_1} that  
$$
\log_2(|A|)\geq (\ell_*+1)8\xi^2_*>\log_2(|A|),
$$
a contradiction. Thus, we must have $\Lambda\subseteq \hat{S}_m$ and the result follows.
\end{proof}

%\subsection{Proof of Remark \ref{rmk:upper_bound_on_Lambda}}
%\label{Proof_rmk_upper_bound_on_Lambda}
%Let $k_1\in \Lambda$ such that $$\nu_{k_1,\emptyset}\geq \frac{\kappa}{2\|A\|_{\infty}^2}$$ and for $2\leq j\leq |\Lambda|$, denote $S_{j}=\{k_1,\ldots,k_{j-1}\}$ and
%let $k_j\in\Lambda\setminus S_{j}$ such that
%$$\nu_{k_j,S_j}\geq \frac{\kappa}{2\|A\|_{\infty}^2}.$$
%The existence of such a sequence of $k_j$ is granted by Corollary \ref{Cor:structural_result_2}.
%Hence, as in the proof of Proposition \ref{Prop:Correctness_1_FSC_Algo},
%$$\log_2(|A|)\geq I(X_0,X_{-k_1})+\sum_{j=2}^{|\Lambda|}I(X_0,X_{-k_j}%|X_{-S_j})\geq |\Lambda| \frac{k^2}{2 \|A\|_{\infty}^4},$$
%and the result follows.
%\bibliography{ref}

\subsubsection{Proof of Theorem \ref{thm:consistency_FSC}}
\label{Sec:proof_consistency_FSC}

To prove Theorem \ref{thm:consistency_FSC} we shall need the following result.
\begin{proposition}
\label{prop:G_holds_high_proba}
Suppose Assumptions  \ref{Ass:MTD_stationary} and  \ref{Ass:Concentration_inequalities} hold, and let $\Delta^{\star}>0$ the quantity defined in Assumption \ref{Ass:Concentration_inequalities}. Then, for any $\xi>0$ and $m>d\geq 2\ell\geq 0$,
\begin{equation}
\label{ineq_for_G}
\bP(G^c_m(\ell,\xi))\leq 2d(\ell+1){d \choose \ell} |A|^{\ell+2}\exp\left\{-\frac{{\xi}^2(m-d)^2(\Delta_*)^2}{18|A|^{2(\ell+2)}m(\ell+2)^2}\right\}.
\end{equation}
%whenever the sample size $n$ satisfies
%$$
%n\geq d+\frac{9|A|^4(\ell+1)}{m^2_{\pi,\ell}\xi^2(\Delta^{\star})^2}\log\left(\frac{|A|^{\ell+2}d^{\ell+1}}{d|A|-1}\right).
\end{proposition}

During the proof of Proposition \ref{prop:G_holds_high_proba} we will make use of the following proposition. For any function $f:A^{\llb 1, m\rrb}\to\bR$, define for each $1\leq j\leq m$,
\begin{equation}
\delta_j(f)=\sup\left\{|f(x_{1:j-1}ax_{j+1:m})-f(x_{1:j-1}bx_{j+1:m})|: a,b\in A, x_{1:m}\in A^{\llb 1, m\rrb}\right\},    
\end{equation}
with the convention that $x_{1:0}=x_{m+1:m}=\emptyset$, $\emptyset ax_{2:m}=ax_{2:m}$ and 
$x_{1:m-1}a=x_{1:m-1}a\emptyset$.
%for all $a\in A$ and $x_{1:m}\in A^{\llb 1, m\rrb}$.
Let $\underline{\delta}(f)=(\delta_1(f),\ldots, \delta_m(f))$ and denote $\|\underline{\delta}(f)\|^2_2=\sum_{j=1}^m\delta^2_j(f)$. In what follows, we write $\bE_{1:m}[f]=\sum_{x_{1:m}\in A^{\llb 1, m\rrb}}\bP(X_{1:m}=x_{1:m})f(x_{1:m})$.

\begin{proposition}[Theorem 3.4. of \citep{CGT_20}]
\label{prop:gaussian_concentration_inequality}
Suppose Assumption \ref{Ass:Concentration_inequalities} holds, that is, $\Delta>0$.
\begin{enumerate}
    \item For any $u>0$ and $f:A^{\llb 1, m\rrb}\to\bR$,
    $$
\bP\left(|f(X_{1:m})-\bE_{1:m}[f]|>u\right)\leq 2\exp\left\{-\frac{2u^2\Delta^2}{\|\underline{\delta}(f)\|^2_2}\right\}.
$$
\item For $m>d$, any $g:A^{S}\times A\to\bR$ with $S\subseteq \llb -d, -1\rrb$ and $u>0$,
 $$
\bP\left(\left|\frac{1}{(m-d)}\sum_{t=d+1}^mg(X_{t+S},X_{t})-\bE_S[g]\right|>u\right)\leq 2\exp\left\{-\frac{u^2(m-d)^2\Delta^2}{2m(|S|+1)^2\|g\|^2_{\infty}}\right\},
$$
where $\bE_S[g]=\sum_{x_{S}\in A^{S}}\sum_{a\in A}\bP(X_{S}=x_{S},X_0=a)g(x_{S},a)$ and $X_{t+S}=(X_{t+j})_{j\in S}$.
\end{enumerate}
\end{proposition}

Before starting the proof Proposition \ref{prop:G_holds_high_proba}, we need to introduce some additional notation. For each 
$x\in A^S$ with $S\subseteq \llb -d,-1\rrb$, we write 
%$\mu(x_{S})=\bP(x_{S}=x_{S})$ and 
\begin{equation}
\label{Def:empirical_marginal}
\hat{{\bf P}}_{m}(x)=\frac{\bar{N}_{m}(x)}{m-d}.   
\end{equation}
In what follows, we write $xa_{V\cup\{0\}}$, with $a\in A$ and $V\subseteq \llb -d, -1\rrb$, to denote the configuration $((xa)_{j})_{j\in V\cup\{0\}}$, defined as 
$$
(xa)_{j}=\begin{cases}
x_{j}, \ \text{for} \ j\in V\\
a, \ \text{for} \ j=0
\end{cases}. 
$$
When $V=S\cup\{k\}$ and $x_{k}=b\in A$, we shall write $xba_{S\cup\{k,0\}}$ instead of $xa_{V\cup\{0\}}$.

With this notation, the empirical version of $\bar{\nu}_{k,S}$ is defined as follows:  
\begin{equation}
\hat{\nu}_{m,k,S}=\sum_{x_{S}\in A^{S}}\hat{{\bf P}}_{m}(x_{S}) \hat{\nu}_{m,k,S}(x_{S})  
\end{equation}
where for $x_{S}\in A^{S}$, we define
\begin{equation}
\hat{\nu}_{m,k,S}(x_{S})=\sum_{b\in A}\sum_{c\in A}\hat{w}_{m,k,S}(b,c,x_{S})\hat{d}_{m,k,S}(b,c,x_{S}),  
\end{equation}
and for $b,c\in A$,
\begin{equation}
\hat{w}_{m,k,S}(b,c,x_{S})=\frac{\hat{{\bf P}}_m(xb_{S\cup\{k\}})}{\hat{{\bf P}}_m(x_{S})}\frac{\hat{{\bf P}}_m(xc_{S\cup\{k\}})}{\hat{{\bf P}}_m(x_{S})},
\end{equation}
and
$$
\hat{d}_{m,k,S}(b,c,x_{S})=\frac{1}{2}\sum_{a\in A}\left|
\hat{\bP}_{x_{S}}(X_0=a|X_{k}=b)-\hat{\bP}_{x_{S}}(X_0=a|X_{k}=c)
\right|,
$$
where for each $b\in A$,
$$
\hat{\bP}_{x_{S}}(X_0=a|X_{k}=b)=\frac{\hat{{\bf P}}_m(xba_{S\cup\{k,0\}})}{\hat{{\bf P}}_m(xb_{S\cup\{k\}})}.
$$

Hereafter, we omit the dependence on $S$ and on $m$, whenever there is no risk of confusion. We now prove Proposition \ref{prop:G_holds_high_proba}. 

\begin{proof}[Proof of Proposition \ref{prop:G_holds_high_proba}]

{\bf Claim 1.} Let $S\subseteq \llb -d, -1\rrb$ with $|S|\leq\ell<d/2$ and take $j\in S^c$. Then,
\begin{equation*}
|\bar{\nu}_{j,S}-\hat{\nu}_{j,S}|
\leq 3 \sum_{x\in A^{S}}\sum_{a\in A}\sum_{b\in A}\left|\hat{\bP}(X_{S}=x,X_{j}=b,X_0=a)-\bP(X_{S}=x,X_{j}=b,X_0=a) \right|.
\end{equation*}

{\bf Proof of the Claim 1.} By applying the triangle inequality twice, one can check that
\begin{multline}
\label{ineq_1_proof_claim_1_prop_concen}
|\bar{\nu}_{j,S}-\hat{\nu}_{j,S}|\leq \frac{1}{2}\sum_{x\in A^{S}}\sum_{a,b,c\in A}\left|{\bf P}(x)w_{j,S}(b,c,x)\left(\bP_{x}(X_0=a|X_{j}=b)-\bP_{x}(X_0=a|X_{j}=c)\right)\right. \\ \left. 
-\hat{{\bf P}}(x)\hat{w}_{j,S}(b,c,x)\left(\hat{\bP}_{x}(X_0=a|X_{j}=b)-\hat{\bP}_{x}(X_0=a|X_{j}=c)\right)\right|.    \end{multline}

Now observe that for fixed $x\in A^{S}$ and $a,b,c\in A$,
$$
{\bf P}(x)w_{j,S}(b,c,x)\bP_{x}(X_0=a|X_{j}=b)=\bP_x(X_{j}=c)\bP(X_{S}=x,X_{j}=b,X_0=a)
$$
and similarly, 
$$
\hat{{\bf P}}(x)\hat{w}_{j,S}(b,c,x)\hat{\bP}_{x}(X_0=a|X_{j}=b)=\hat{\bP}_x(X_{j}=c)\hat{\bP}(X_{S}=x,X_{j}=b,X_0=a).
$$
By using these identities in \eqref{ineq_1_proof_claim_1_prop_concen} and then by applying the triangle inequality, one can deduce that
\begin{multline}
\label{ineq_2_proof_claim_1_prop_concen}
|\bar{\nu}_{j,S}-\hat{\nu}_{j,S}|\leq \sum_{x\in A^{S}}\sum_{a,b,c\in A}\left|\bP_x(X_{j}=c)\bP(X_{S}=x,X_{j}=b,X_0=a)\right. \\ \left. 
-\hat{\bP}_x(X_{j}=c)\hat{\bP}(X_{S}=x,X_{j}=b,X_0=a)\right|.    
\end{multline}

By adding and subtracting the term $\bP_x(X_{j}=c)\hat{\bP}(X_{S}=x,X_{j}=b,X_0=a)$ in the right-hand side of the above inequality and using again the triangle inequality, it follows that
\begin{multline}
\label{ineq_3_proof_claim_1_prop_concen}
\sum_{x\in A^{S}}\sum_{a,b,c\in A}\left|\bP_x(X_{j}=c)\bP(X_{S}=x,X_{j}=b,X_0=a)\right. \\ \left. 
-\hat{\bP}_x(X_{j}=c)\hat{\bP}(X_{S}=x,X_{j}=b,X_0=a)\right|\\
\leq \sum_{x\in A^{S}}\sum_{a,b\in A}|\bP(X_{S}=x,X_{j}=b,X_0=a)-\hat{\bP}(X_{S}=x,X_{j}=b,X_0=a)|\\
+ \sum_{x\in A^{S}}\sum_{a,c\in A}\hat{\bP}(X_{S}=x,X_0=a)|\bP_x(X_{j}=c)-\hat{\bP}_x(X_{j}=c)|.
\end{multline}

By adding and subtracting the term $\bP(X_{S}=x)\bP(X_{S}=x,X_{j}=c)$, we can then check that
\begin{multline}
\label{ineq_4_proof_claim_1_prop_concen}
|\bP_x(X_{j}=c)-\hat{\bP}_x(X_{j}=c)|\leq
\frac{\bP(X_{j}=c)}{\hat{\bP}(X_{S}=x)}|\hat{\bP}(X_{S}=x)-\bP(X_{S}=x)|\\+\frac{1}{\hat{\bP}(X_{S}=x)}|\hat{\bP}(X_{S}=x,X_{j}=c)-\bP(X_{S}=x,X_{j}=c)|.
\end{multline}
From \eqref{ineq_3_proof_claim_1_prop_concen} and \eqref{ineq_4_proof_claim_1_prop_concen}, one deduces that
\begin{multline}
\label{ineq_5_proof_claim_1_prop_concen}
|\bar{\nu}_{j,S}-\hat{\nu}_{j,S}|\leq \sum_{x\in A^{S}}\sum_{a,b\in A}|\bP(X_{S}=x,X_{j}=b,X_0=a)-\hat{\bP}(X_{S}=x,X_{j}=b,X_0=a)|\\
+\sum_{x\in A^{S}}\sum_{c\in A} |\bP(X_{S}=x,X_{j}=c)-\hat{\bP}(X_{S}=x,X_{j}=c)|\\+\sum_{x\in A^{S}}|\bP(X_{S}=x)-\hat{\bP}(X_{S}=x)|.
\end{multline}
Since 
\begin{multline*}
|\bP(X_{S}=x,X_{j}=c)-\hat{\bP}(X_{S}=x,X_{j}=c)|\\
\leq \sum_{a\in A}|\bP(X_{S}=x,X_{j}=c,X_0=a)-\hat{\bP}(X_{S}=x,X_{j}=c,X_0=a)| \end{multline*}
and
\begin{multline*}
|\bP(X_{S}=x)-\hat{\bP}(X_{S}=x)|\\
\leq \sum_{a,c\in A}|\bP(X_{S}=x,X_{j}=c,X_0=a)-\hat{\bP}(X_{S}=x,X_{j}=c,X_0=a)|, 
\end{multline*}
the proof of Claim 1 follows from \eqref{ineq_5_proof_claim_1_prop_concen}.

{\bf Claim 2.} For any $u>0$,
\begin{multline*}
\bP\left(3 \sum_{x\in A^{S}}\sum_{a\in A}\sum_{b\in A}\left|\hat{\bP}(X_{S}=x,X_{j}=b,X_0=a)-\bP(X_{S}=x,X_{j}=b,X_0=a) \right|>u\right)\\
\leq 2|A|^{|S|+2}\exp\left\{-\frac{u^2(m-d)^2\Delta^2}{18|A|^{2(|S|+2)}m(|S|+2)^2}\right\}.     
\end{multline*}

{\bf Proof of Claim 2.} It follows from the union bound and Proposition \ref{prop:gaussian_concentration_inequality}.

\medskip
We now will conclude the proof.  Let $\mathcal{S}_k=\{S\subseteq \llb -d, -1\rrb: |S|=k\}$ and observe that by the union bound
$$
\bP(G^c_m(\xi,\ell))\leq \sum_{k=0}^{\ell}\sum_{S\in \mathcal{S}_k }\sum_{j\in S^c}\bP\left(|\bar{\nu}_{j,S}-\hat{\nu}_{j,S}|>\xi\right).
$$
Combining Claims 1 and 2, it follows that 
$$
\bP\left(|\bar{\nu}_{j,S}-\hat{\nu}_{j,S}|>\xi\right)\leq 2|A|^{|S|+2}\exp\left\{-\frac{{\xi}^2(m-d)(\Delta_*)^2}{18|A|^{2(|S|+2)}m(|S|+2)^2}\right\},
$$
which implies that
$$
\bP(G^c_m(\xi,\ell))\leq 2\sum_{k=0}^{\ell} \binom{d}{k}(d-k) |A|^{k+2}\exp\left\{-\frac{{\xi}^2(m-d)^2\Delta^2}{18|A|^{2(k+2)}m(k+2)^2}\right\}.
$$
Since $\ell\leq d/2$, we can use that $\binom{d}{k}\leq \binom{d}{\ell}$  for all $0\leq k\leq \ell$
%and the relation $d(\ell+1)\binom{d-1}{\ell}=(\ell+1)^2\binom{d}{\ell}$, 
to obtain that
\[
\bP(G^c_m(\xi,\ell))\leq 2d (\ell+1)\binom{d}{\ell} |A|^{\ell+2}\exp\left\{-\frac{{\xi}^2(m-d)(\Delta_*)^2}{18|A|^{2(\ell+2)}m(\ell+2)^2}\right\},
\]
and the result follows.
\end{proof}

We now prove Theorem \ref{thm:consistency_FSC}.

\begin{proof}[Proof of Theorem \ref{thm:consistency_FSC}]

First, observe that by Theorem \ref{Prop:Correctness_1_FSC_Algo},
\begin{equation}
\label{proof_consi_FSC_ineq_1}
\bP(\hat{\Lambda}_{2,n}\neq\Lambda)\leq \bP(G^c_m(\xi_*,\ell_*))+\bP(\Lambda\subseteq \hat{S}_m, \hat{\Lambda}_{2,n}\neq\Lambda)\end{equation}

Next, notice that the second term on the right hand side of \eqref{proof_consi_FSC_ineq_1} can be written as
\[
\bP(\Lambda\subseteq \hat{S}_m, \hat{\Lambda}_{2,n}\neq\Lambda)=\sum_{S\subseteq \llb -d,-1\rrb: \Lambda\subseteq S,|S|\leq\ell_*}\bP(\hat{S}_m=S,\hat{\Lambda}_{2,n}\neq\Lambda).
\]

Now for any $S\in \llb -d,-1\rrb$ such that $\Lambda\subseteq S,|S|\leq\ell_*$, it follows from the union bound that
\[
\bP(\hat{S}_m=S,\hat{\Lambda}_{2,n}\neq\Lambda)\leq \sum_{j\in\Lambda}\bP(\hat{S}_m=S,j\notin\hat{\Lambda}_{2,n})+\sum_{j\in S\setminus\Lambda}\bP(\hat{S}_m=S,j\in\hat{\Lambda}_{2,n}).
\]
By proceeding similarly as in the proof of Item 1 of Theorem \ref{thm:consistency_PCP_estimator}, one can deduce that for any $j\in S\setminus\Lambda$, 
$$
\bP(\hat{S}_m=S,j\in\hat{\Lambda}_{2,n})\leq 8|A|\left\lceil\frac{\log(\mu (n-m-d)/\alpha+2)}{\log(1+\varepsilon)}\right\rceil e^{-\alpha}\sum_{x\in A^S}\bP(\hat{S}_m=S,\bar{N}_{m,n}(x)>0),
$$
so that
\begin{multline*}
\sum_{j\in S\setminus\Lambda}\bP(\hat{S}_m=S,j\in\hat{\Lambda}_{2,n})\leq  8(\ell_*-|\Lambda|)|A|\left\lceil\frac{\log(\mu (n-m-d)/\alpha+2)}{\log(1+\varepsilon)}\right\rceil e^{-\alpha}\\ \times\sum_{x\in A^S}\bP(\hat{S}_m=S,\bar{N}_{m,n}(x)>0).   
\end{multline*}
Since 
$$
\sum_{S\subseteq \llb -d,-1\rrb: \Lambda\subseteq S,|S|\leq\ell_*}\sum_{x\in A^S}\bP(\hat{S}_m=S,\bar{N}_{m,n}(x)>0)\leq (n-m-d)\bP(\Lambda\subseteq\hat{S}_m)
$$
we then deduce that
\begin{multline*}
\sum_{S\subseteq \llb -d,-1\rrb: \Lambda\subseteq S,|S|\leq\ell_*}\sum_{j\in S\setminus\Lambda}\bP(\hat{S}_m=S,j\in\hat{\Lambda}_{2,n})\leq  8(\ell_*-|\Lambda|)|A|\\\times\left\lceil\frac{\log(\mu (n-m-d)/\alpha+2)}{\log(1+\varepsilon)}\right\rceil  e^{-\alpha} (n-m-d).   
\end{multline*}

Following the steps of the proof of Item 2 of of Theorem \ref{thm:consistency_PCP_estimator}, we can also show that
\begin{multline*}
\sum_{S\subseteq \llb -d,-1\rrb: \Lambda\subseteq S,|S|\leq\ell_*}\sum_{j\in\Lambda}\bP(\hat{S}_m=S,j\notin\hat{\Lambda}_{2,n},\delta_j\geq \gamma^S_{m,n,j})\leq  8|\Lambda||A|\\\times\left\lceil\frac{\log(\mu (n-m-d)/\alpha+2)}{\log(1+\varepsilon)}\right\rceil  e^{-\alpha} \sum_{S\subseteq \llb -d,-1\rrb: \Lambda\subseteq S,|S|\leq\ell_*}\bP(\hat{S}_m=S)\\
\leq 8|\Lambda||A|\left\lceil\frac{\log(\mu (n-m-d)/\alpha+2)}{\log(1+\varepsilon)}\right\rceil  e^{-\alpha},
\end{multline*}
where $\gamma^S_{m,n,j}$ is defined as in \eqref{Def:thresholds_for_PCP_esti} with 
$t_{m,n,j}=\max_{b,c\in A:b\neq c}\min_{(x_S,y_S)\in \mathcal{C}_j(b,c)}t_{m,n}(x_S,y_S)$ in the place of $t_{n,j}$.

Hence, it remains to estimate 
$$
\sum_{S\subseteq \llb -d,-1\rrb: \Lambda\subseteq S,|S|\leq\ell_*}\sum_{j\in\Lambda}\bP(\hat{S}_m=S,j\notin\hat{\Lambda}_{2,n},\delta_j< \gamma^S_{m,n,j}).
$$
By proceeding similarly to the proof of Item 3 of Theorem \ref{thm:consistency_PCP_estimator}, one can show that for each $S\subseteq \llb -d,-1\rrb$ such that $|S|\leq\ell_*$, 
\begin{equation*}
\bP\left(\gamma^S_{m,n,j}>\delta_j \right)\leq
6|A|(|A|-1)\exp\left\{\frac{-(\Delta p_{min})^2(n-m-d)^2}{2(n-m)|A|^{2(\ell_*-1)}(\ell_*+1)^2}\left(1-\frac{n_{min}}{n-m-d}\right)
^2\right\},
\end{equation*}
for all $j\in \Lambda$ as long as 
$n$ satisfies $n>m+d+n_{min}.$
By using this upper bound and by recalling that
$\sum_{S\subseteq \llb -d,-1\rrb: \Lambda\subseteq S,|S|\leq\ell_*}\leq \sum_{k=0}^{\ell_*}\binom{d}{k}\leq (\ell_*+1) \binom{d}{\ell_*}$ (since $2\ell_*\leq d$), we deduce that
\begin{multline*}
\sum_{S\subseteq \llb -d,-1\rrb: \Lambda\subseteq S,|S|\leq\ell_*}\sum_{j\in\Lambda}\bP(\hat{S}_m=S,j\notin\hat{\Lambda}_{2,n},\delta_j< \gamma^S_{m,n,j})\leq \\
6|A|(|A|-1)(\ell_*+1) \binom{d}{\ell_*}\exp\left\{\frac{-(\Delta p_{min})^2(n-m-d)^2}{2(n-m)|A|^{2(\ell_*-1)}(\ell_*+1)^2}\left(1-\frac{n_{min}}{n-m-d}\right)
^2\right\},
\end{multline*}
for all $j\in \Lambda$ whenever $n>m+d+n_{min},$ implying the result.
\end{proof}

\subsubsection{Proof of Corollary \ref{cor:consistency_FSC}}
\label{proof:cor_conist_FSC_Esti}

\begin{proof}[Proof of Corollary \ref{cor:consistency_FSC}]
Assumptions \ref{Ass:MTD_stationary} and \ref{Ass:Concentration_inequalities} are satisfied for all values of $n$, since $p^{\star}_n\geq p^{\star}_{min}$ and $\Delta^{\star}_n\geq \Delta^{\star}_{min}$ for positive constants  $p^{\star}_{min}$ and $\Delta^{\star}_{min}$ for all n. 
Since the sequence of MTD models also satisfy Assumption \ref{Ass:non_determinist_conditional_averages},  the result follows from Theorem \ref{thm:consistency_FSC} and Remark \ref{rmk:consist_FSC}-Item (b).
\end{proof}

\subsection{Proofs of Section \ref{sec:supereficient_bin_MTD}}

\subsubsection{Proof of Theorem \ref{thm:FSC_consistency_binary}}
\label{App:proof_them_FSC_consis_binary}

\begin{proof}[Proof of Theorem \ref{thm:FSC_consistency_binary}]

Notice that we can write for each $j\in \llb -d, -1\rrb$,
$$
m_j(X_{j})=(p_j(1|1)-p_j(1|0)X_{-j}+p_j(1|0),
$$
so that equality \eqref{identify_conditional_cov} can be rewritten for any $j\in \llb -d, -1\rrb$ satisfying $(p_j(1|1)-p_j(1|0))\neq 0$, $S\subseteq \llb -d, -1\rrb\setminus\{j\}$ and $x_{S}\in\{0,1\}^{S}$,  as
\begin{equation*}
\mCov_{x_{S}}(X_0,X_{j})=\sum_{\ell\in\Lambda\setminus S}\Delta_\ell\mCov_{x_{S}}(X_{\ell},X_{j}),
\end{equation*}
where $\Delta_{\ell}=\lambda_{\ell}(p_{\ell}(1|1)-p_{\ell}(1|0))$ for $\ell\in\Lambda$. Recalling that $\bar{\nu}_{j,S}=2\bE\left(\left|\mCov_{X_{S}}(X_0,X_{j})\right|\right)$ in the binary case, we can deduce that 
for any $S\subseteq \llb -d, -1\rrb$, $x_{S}\in \{0,1\}^{S}$ and $j\in \llb -d, -1\rrb\setminus S$,
\begin{equation*}
\bar{\nu}_{j,S}=2\bE\left(\left|\sum_{\ell\in\Lambda\setminus S}\Delta_{\ell}\mCov_{X_{S}}(X_{\ell},X_{j})\right|\right).
\end{equation*}

As a consequence, it follows that for $S\subset \Lambda$ and $j\in \Lambda\setminus S$,
\begin{multline}
\label{ineq_1_proof_consis_FSC_binary_alphabet}
\bar{\nu}_{j,S}\geq 2\bE\left(\text{Var}_{X_{S}}(X_j)|\Delta_j|-\sum_{\ell\in\Lambda\setminus S:\ell\neq j}|\Delta_{\ell}||\mCov_{X_{S}}(X_{\ell},X_{j})|\right)\nonumber
\\
=2\bE\left(|\Delta_j|\text{Var}_{X_{S}}(X_j)\times \right.\\ \left. \left(1-\sum_{\ell\in\Lambda\setminus S:\ell\neq j}\frac{|\Delta_{\ell}|}{|\Delta_{j}|}|\bP_{x_{S}}(X_{\ell}=1|X_{j}=1)-\bP_{x_{S}}(X_{\ell}=1|X_{j}=0)|\right)\right),
\end{multline}
where in the second inequality we have used that 
\begin{equation*}
|\mCov_{x_{S}}(X_{\ell},X_{j})|=\text{Var}_{x_{S}}(X_j)|\bP_{x_{S}}(X_{\ell}=1|X_{j}=1)-\bP_{x_{S}}(X_{\ell}=1|X_{j}=0)|.
\end{equation*}
By Assumption \ref{ass:lower_bound_kappa}, we then deduce that
\begin{equation}
\label{ineq_2_proof_consis_FSC_binary_alphabet}
\bar{\nu}_{j,S}\geq 2\Gamma_1 |\Delta_j|\bE(\text{Var}_{X_{S}}(X_j)).    
\end{equation}

Now, take $S\subset \Lambda$ and let $j_S\in\arg\min_{\ell \in \Lambda\setminus S} |\Delta_j|\bE(\text{Var}_{X_{S}}(X_{\ell}))$. For any $j\in (\Lambda)^c$, use the triangle inequality, equality \eqref{ineq_2_proof_consis_FSC_binary_alphabet} and Assumption \ref{ass:incoherence_condition_binary} to deduce that
\begin{align}
\label{ineq_3_proof_consis_FSC_binary_alphabet}
\bar{\nu}_{j,S}\leq 2\bE\left(\sum_{\ell\in\Lambda\setminus S}|\Delta_{\ell}||\mCov_{X_{S}}(X_{\ell},X_{j})|\right)\nonumber
\\
\leq 2|\Delta_{j_S}|\bE\left(\text{Var}_{X_{S}}(X_{j_S})\right)\Gamma_2
\end{align}
Using that $\delta_j=|\Delta_j|$ and combing inequalities \eqref{ineq_2_proof_consis_FSC_binary_alphabet}
and \eqref{ineq_3_proof_consis_FSC_binary_alphabet}, it  follows then that
\begin{equation}
\max_{j\in \Lambda\setminus S}\bar{\nu}_{j,S}-\max_{j\in (\Lambda)^c}\bar{\nu}_{j,S}\geq 2(\Gamma_1-\Gamma_2)|\Delta_{j_{S}}|\bE\left(\text{Var}_{X_{S}}(X_{j_S})\right), 
\end{equation}
where we have used also that $\max_{j\in \Lambda\setminus S}\bar{\nu}_{j,S}\geq \bar{\nu}_{j_S,S}$.
Using that $\bE\left(\text{Var}_{X_{S}}(X_{j_S})\right)\geq (p^{\star})^2$ and $|\Delta_{j_{S}}|\geq \min_{j\in\Lambda}|\Delta_j|$ we obtain that
$$
\min_{S\subset\Lambda}\left(\max_{j\in \Lambda\setminus S}\bar{\nu}_{j,S}-\max_{j\in (\Lambda)^c}\bar{\nu}_{j,S}\right)\geq 2(\Gamma_1-\Gamma_2)p^{2}_{min} \min_{j\in\Lambda}|\Delta_j| .
$$
concluding the first half of the proof.

To show the second assertion of the theorem, take $S\subset\Lambda$, let $j^*_S\in \arg\max_{j\in \Lambda\setminus S}\bar{\nu}_{j,S}$ and note that on $G_n(\ell,\xi),$
$$
\max_{j\in \Lambda\setminus S}\hat{\nu}_{n,j,S}\geq \hat{\nu}_{n,j^*_S,S} \geq \bar{\nu}_{j^*_S,S}-\xi=\max_{j\in \Lambda\setminus S}\bar{\nu}_{j,S}-\xi. 
$$
Similarly, one can show that on $G_n(\ell,\xi),$
$$
\max_{j\in (\Lambda)^c}\hat{\nu}_{n,j,S}\leq \max_{j\in (\Lambda)^c}\bar{\nu}_{j,S}+\xi. 
$$
As a consequence, it follows that
$$
\max_{j\in \Lambda\setminus S}\hat{\nu}_{n,j,S}-\max_{j\in (\Lambda)^c}\hat{\nu}_{n,j,S}\geq \left(\max_{j\in \Lambda\setminus S}\bar{\nu}_{j,S}-\max_{j\in (\Lambda)^c}\bar{\nu}_{j,S}\right) -2\xi,
$$
whenever $G_n(\ell,\xi)$. By taking $\xi$ as in \eqref{choice_xi_consis_with_FS_step_only}, we have that
$$
\max_{j\in \Lambda\setminus S}\hat{\nu}_{n,j,S}-\max_{j\in (\Lambda)^c}\hat{\nu}_{n,j,S}>0,
$$
implying that $\arg\max_{j\in S^c}\hat{\nu}_{n,j,S}\in \Lambda$ for all $S\subset \Lambda$, and the result follows.
\end{proof}

\subsubsection{Proof of Corollary \ref{cor:consistency_SEE}}
\label{sec:prof_cor_consist_SEE}

\begin{proof}[Proof of Corollary \ref{cor:consistency_SEE}]
By Theorem \ref{thm:FSC_consistency_binary}, we have that
$$
\bP(\hat{\Lambda}_{2,n}\neq\Lambda)\leq \bP(G^c_m(\xi,L))+\bP(\Lambda\subseteq \hat{S}_m, \hat{\Lambda}_{2,n}\neq\Lambda),
$$
for any $\xi<(\Gamma_1-\Gamma_2)p^{\star}_{min}\min_{j\in\Lambda}\delta_j.$

By Proposition \ref{prop:G_holds_high_proba}, we have that

$$
\bP(G^c_m(\xi,L))\leq 2d(L+1){d \choose L} |A|^{L+2}\exp\left\{-\frac{{\xi}^2(m-d)^2\Delta^2}{18|A|^{2(L+2)}m(L+2)^2}\right\}.
$$
By taking $\xi=(\Gamma_1-\Gamma_2)p^{\star}_{min}\min_{j\in\Lambda}\delta_j$, one can check that if 
\begin{equation*}
 \min_{j\in\Lambda}\delta_j\geq C_1\frac{\log(n)}{n},   
\end{equation*}
for some constant $C_1=C_1(\beta,L,\Delta^*_{min},p^*_{min}, \Gamma_1,\Gamma_2)$, then
$\bP(G^c_m(\xi,L))\to 0$ as $n\to\infty.$

By proceeding exactly as in the proof of Theorem \ref{thm:consistency_FSC} and using \ref{rmk:consist_FSC}-item (b), we can show that 
\begin{multline*}
\bP(\Lambda\subseteq \hat{S}_m, \hat{\Lambda}_{2,n}\neq\Lambda)\leq \\
 6|A|(L+1){d \choose L}\left[ d|A|^{L+1}(|A|-1)|\Lambda|\exp\left\{-\frac{(\Delta p_{min})^2(n-d)^2}{8n(L+1)^2|A|^{2(L-1)}}\right\}\right]
\\+8|A|\left((L-|\Lambda|)(n-m-d)+|\Lambda|\right)\left\lceil\frac{\log(\mu (n-m-d)/\alpha+2)}{\log(1+\varepsilon)}\right\rceil e^{-\alpha},
\end{multline*}
as long as 
$$
n\geq d+\frac{C|A|^{L}\alpha}{p_{min}\delta^2_{min}},
$$
where $C=C(\mu,\varepsilon).$

Therefore, using that $\Delta \geq \Delta^{\star}_{min}$, $p_{min}\geq p^{\star}_{min}$, $d=n\beta$, $\alpha=(1+\eta)\log(n)$, we also see that if  $\delta^2_{min}\geq C_2\frac{\log(n)}{n}$ for some $C_2=C_2(\mu,\varepsilon,\eta,\Delta^{\star}_{min},p^{\star}_{min},\beta,L)$ then $\bP(\Lambda\subseteq \hat{S}_m, \hat{\Lambda}_{2,n}\neq\Lambda)\to 0$ as $n\to\infty$.

By taking $C=C_1\vee C_2$, we deduce that $\bP(\hat{\Lambda}_{2,n}\neq\Lambda)\to 0$ as $n\to\infty$ as long as  $\delta^2_{min}\geq C\frac{\log(n)}{n}$, and the result follows.

\end{proof}

\subsection{Proofs of Section \ref{Sec:ETP}}
\label{sec_proof_section_ETP}

\begin{proof}[Proof of Theorem \ref{THM:ETP}]

By the union bound, we have that
\begin{multline*}
\bP\left(\bigcup_{a\in A}\bigcup_{x_{-d:-1}\in A^{\llb -d, -1\rrb}}\left\{|\hat{p}_n(a|x_{\hat{\Lambda}_{2,n}})-p(a|x_{\Lambda})|\geq \sqrt{\frac{2\alpha(1+\epsilon)\hat{V}_{n}(a,x_{\hat{\Lambda}_{2,n}})}{\bar{N}_{n}(x_{\hat{\Lambda}_{2,n}})}}\right.\right.\\
\left. \left. +\frac{\alpha}{3\bar{N}_{n}(x_{\hat{\Lambda}_{2,n}})} \right\}\right)
\leq
\bP(\Lambda\neq\hat{\Lambda}_{2,n})\\+\sum_{a\in A}\sum_{x_{\Lambda}\in A^{\Lambda}}\bP\left(
|\hat{p}_n(a|x_{\Lambda})-p(a|x_{\Lambda})|\geq \sqrt{\frac{2\alpha(1+\epsilon)\hat{V}_{n}(a,x_{\Lambda})}{\bar{N}_{n}(x_{\Lambda})}}+\frac{\alpha}{3\bar{N}_{n}(x_{\Lambda})}
\right).
\end{multline*}
Now, Proposition \ref{prop:bernstein_empirical_trans_prob_correct_var} implies that for any $a\in A$ and $x_{\Lambda}\in A^{\Lambda}$,  
\begin{multline*}
\bP\left(
|\hat{p}_n(a|x_{\Lambda})-p(a|x_{\Lambda})|\geq \sqrt{\frac{2\alpha(1+\epsilon)\hat{V}_{n}(a,x_{\Lambda})}{\bar{N}_{n}(x_{\Lambda})}}+\frac{\alpha}{3\bar{N}_{n}(x_{\Lambda})}
\right)
\\ \leq 4\left\lceil\frac{\log(\mu (n-d)/\alpha+2)}{\log(1+\varepsilon)}\right\rceil e^{-\alpha}\bP\left(\bar{N}_{n}(x_{\Lambda})>0\right),    
\end{multline*}
so that 
\begin{multline*}
\sum_{a\in A}\sum_{x_{\Lambda}\in A^{\Lambda}}\bP\left(
|\hat{p}_n(a|x_{\Lambda})-p(a|x_{\Lambda})|\geq \sqrt{\frac{2\alpha(1+\epsilon)\hat{V}_{n}(a,x_{\Lambda})}{\bar{N}_{n}(x_{\Lambda})}}+\frac{\alpha}{3\bar{N}_{n}(x_{\Lambda})}
\right)\\
\leq 4|A|\left\lceil\frac{\log(\mu (n-d)/\alpha+2)}{\log(1+\varepsilon)}\right\rceil e^{-\alpha}\times\sum_{x_{\Lambda}\in A^{\Lambda}}\bE\left[1\{\bar{N}_{n}(x_{\Lambda})>0\}\right].
\end{multline*}
Since
$$
\sum_{x_{\Lambda}\in A^{\Lambda}}\bE[1\{\bar{N}_{n}(x_{\Lambda})>0\}]\leq (n-d),
$$
the result follows.
\end{proof}

\subsection{Proof of Section \ref{sec:minimax_rate}}
\label{prof_prop:upper_bound_on_minimal_oscilation}

\subsubsection{Proof of Proposition \ref{prop:upper_bound_on_minimal_oscilation}}
\begin{proof}[Proof of Proposition \ref{prop:upper_bound_on_minimal_oscilation}]

First observe that since all MTDs are stationary Markov chains of order at most d, we can use the Markov property to show that

$$
KL(P^{(j)}_n||P^{(k)}_n)=KL(P^{(j)}_d||P^{(k)}_d)+(n-d)\bE^{(j)}(KL(p^{(j)}(\cdot|X_{-d:-1})||p^{(k)}(\cdot|X_{-d:-1}))).
$$
where $KL(p^{(j)}(\cdot|x_{-d:-1})||p^{(k)}(\cdot|x_{-d:-1}))$ denotes the Kullback-Leibler divergence between $p^{(j)}(\cdot|x_{-d:-1})$ and $p^{(k)}(\cdot|x_{-d:-1})$. 

Now note that for each fixed $x_{-:d_1}\in A^{\llb -d,-1\rrb}$, we can 
use the definition of the transition probabilities $p^{(j)}(\cdot|\cdot)$ together with Lemma 6 of Csizar and Talata (2005) to deduce that
$$
KL(p^{(j)}(\cdot|x_{-d:-1})||p^{(k)}(\cdot|x_{-d:-1}))\leq \lambda^2|p(1|1)-p(1|0)|^2(p_{min})^{-1}1_{\{x_j\neq x_k\}}.
$$
Since $p_{min}\geq (1-\lambda)/2$ and $\delta=\lambda |p(1|1)-p(1|0)|$, it follows from the above inequality that 
$$
\bE^{(j)}(KL(p^{(j)}(\cdot|X_{-d:-1})||p^{(k)}(\cdot|X_{-d:-1})))\leq \frac{2\delta^2}{1-\lambda}.
$$
By using similar arguments, one can also show that
\begin{multline*}
KL(P^{(j)}_d||P^{(k)}_d)\leq \frac{1}{p_{min}}\left(\max_{x_{-d:-1},y_{-d:-1}}|p^{(j)}(1|x_{-d:-1})-p^{(k)}(1|y_{-d:-1})|^2\right.\\ \left.
+\sum_{i=1}^{d-1}\max_{x_{-d:-1},y_{-d:-1}:x_{-i:-1}=y_{-i:-1}}|p^{(j)}(1|x_{-d:-1})-p^{(k)}(1|y_{-d:-1})|^2\right)\leq \frac{2d\delta^2}{1-\lambda}.
\end{multline*}
Therefore, it follows that
$$
KL(P^{(j)}_n||P^{(k)}_n)\leq \frac{2d\delta^2}{1-\lambda}+(n-d)\frac{2\delta^2}{1-\lambda}=\frac{2n\delta^2}{1-\lambda},
$$
and the result follows.
\end{proof}

\subsection{Computation of PCP and FSC estimators}
\label{computation_PCP_and_FSC_estimator}

We will first show that one can compute the PCP estimator with at most 
$O(|A|^2|S|(n-d))$ computations, as claimed in item (c) of Remark \ref{rmk:PCP_consistency}.

{\bf Proof of item (c) of Remark \ref{rmk:PCP_consistency}}. One way to compute the PCP estimator is the following. First, we compute $N_n(x_S,a)$ simultaneously for all pasts $x_S$ and symbols $a\in A$, and build the set $E_S=\{x_S: \bar{N}_n(x_S)>0\}$. This can be done with $O(n-d)$ computations. Indeed, we set initially $N_{n}(x_S,a)=0$ for all past $x_S$ and symbol $a\in A$. Then at each time $d+1\leq t\leq n$, we increment by $1$ the count of $N_n(x_S,a)$ for which $X_{t+S}=x_S$ and $X_t=a$, leaving all the other counts unchanged. Moreover, at the first time that $N_{n}(x_S,a)>0$, we include $x_S$ in the set $E$. Note that the cardinality of the set $E_S$ is at most $(n-d)$. Next, we need to compute $s_n(x_S)$ and $\hat{p}_n(\cdot|x_S)$ for each $x_S\in E$, which can be done with at most $O(|A|)$ additional computations. Once all these quantities are determined, we then need to test whether a given lag $j\in S$ has to be removed or not, by evaluating inequality \eqref{Def:PCP_estimator}
for all pairs of $(S\setminus\{j\})$-compatible pasts in $E_S$. This can be done with at most $O(|A|^2(n-d))$ more computations because 1) the number of different pasts in $E_S$ is at most $(n-d)$; 2) there are at most $|A|$ pasts in $E$ which are compatible with a fixed past $x_S$ in $E_S$; and 3) one can evaluate whether inequality \eqref{Def:PCP_estimator} holds or not to a given pair of compatible past with $O(|A|)$ additional computations. Finally, since the number of lags to be tested is $|S|$, it follows that we can implement the PCP estimator with at most $O(|A|^2|S|(n-d))$ computations, concluding the proof.

We now show that we can compute the FSC estimator by using at most $O(|A|^3\ell(m-d)(d-(\ell-1)/2)+|A|^2(n-m-d)\ell)$ computations, as stated in item (c) of Remark \ref{rmk:FSC_estimator}.

{\bf Proof of item (c) of Remark \ref{rmk:FSC_estimator}}.

By the item (c) of Remark \ref{rmk:PCP_consistency}, the CUT step can be computed with at most $O(\ell|A|^2(n-m-d))$ computations since the FS step outputs a subset of size $\ell$ and the size of the second half of the sample is $n-m$. Hence, the proof will be concluded if we show that the FS step can be computed with at most $O(|A|^3(m-d)(\ell d-(\ell-1)\ell/2)$ computations. To see that, let us fix $S\subseteq \llb -d, -1\rrb$ and $j\notin S$. Proceeding as in the proof item (c) of Remark \ref{rmk:PCP_consistency}, one can check that we compute $N_m(xba_{S\cup\{0,j\}})$ simultaneously for all configurations $xba_{S\cup\{j,0\}}$ and build the set $E_{S}=\{x_S:N_{m}(x_S)>0\}$ with $O(m-d)$ computations. Notice that the size of the set $E_{S}$ is most $(m-d)$.
Since for each $x_{S}\in E_{S}$, we need to perform at most  $O(|A|^3)$ additional operations to compute $\hat{\bP}_{m}(x_{S})$ and $\hat{\nu}_{j,m}(x_S)$, it follows that with at most $O(|A|^3(m-d))$ computations we can determine $\hat{\nu}_{m,j,S}$. Therefore,  the step 3 of the FS step (where we need to compute $\hat{\nu}_{m,j,S}$ for $j\in S^c$) can be implemented with $O(|A|^3(m-d)(d-|S|))$ calculations. Since we need to repeat step 3 of the FS step for $\ell$ different sets, we conclude that with at most 
$$
O(|A|^3(m-d)\sum_{|S|=0}^{\ell-1}(d-|S|))=O(|A|^3(m-d)(\ell d-(\ell-1)\ell/2)
$$
computations, we can implement the FS step. This concludes the proof.

%\smallskip
%{\bf Claim 2}. The CUT step can be computed with at most $O(\ell|A|^2(n-m-d)^2)$ computations.

%{\bf Proof of Claim 2.} Recall that the FS step outputs a subset of size $\ell$, denoted $\hat{S}_m$. To test whether a given lag $j\in \hat{S}_m$ has to be removed or not, we need to check if  $
%d_{TV}(\hat{p}_{m,n}(\cdot|x_{\hat{S}_{m}}),\hat{p}_{m,n}(\cdot|y_{\hat{S}_{m}}))\geq t_{m,n}(x_{\hat{S}_{m}},y_{\hat{S}_{m}}),
%$
%for all $(\hat{S}_{m}\setminus\{j\})$-compatible pasts $x_{\hat{S}_{m}},y_{\hat{S}_{m}}\in A^{\hat{S}_{m}}$ which appear in $X_{m+1:n}$, the second half of the sample. Clearly, this can be done with at most $O(|A|^2(n-m-d)^2)$ since the number of different configurations which may appear in  $X_{m+1:n}$ is at most $(n-m-d)$. Since the number of lags to be tested is $\ell$, it follows that we can implement the CUT step with at most $O(\ell(n-m-d)^2|A|^2)$, concluding the proof of Claim 2.

%From Claim 1 and 2 it follows that FSC estimator can be computed by using at most $O(|A|^2\ell(m-d)( d-(\ell-1)/2)+|A|^2(n-m-d)^2\ell)$ computations.

\section{Martingale concentration inequalities}
\label{Sec:Mart_Conc_Ineq}

In the sequel, $\bN$ denotes the set of non-negative integers $\{0,1,\ldots\}$
Let $(\Omega,\cF,\bP)$ be a probability space. We assume that this probability space is rich enough so that the following stochastic processes may be defined on it. 
In what follows, let $(X_{t})_{t\in\bZ}$ be a Markov chain of order $d\in\bZ_+$, taking values on a finite alphabet $A$, with family of transition probabilities $\{p(\cdot|x_{-d:-1}): x_{-d:-1}\in \Supp\}$. 
Denote 
%$\cF_0=\{\Omega,\emptyset\}$ and 
$\cF_t=\sigma(X_{-d:t})$ for $t\in\bN$.
For each $a\in A$, consider the stochastic process $M^{a}=((M^{a}_t))_{t\in\bN}$ defined as,
$$
M^{a}_t=1\{X_t=a\}-p(a|X_{(t-d):(t-1)}), \ t\in\bN.
$$ 
%where $X_{t-d}^{t-1}:=(X_{t-d}\ldots X_{t-1})$.

Let $H=(H_{t})_{t\in \bN}$ be a stochastic process taking values on a finite alphabet $B\subset\bR$, satisfying $H_0=0$ and $H_t\in\cF_{t-1}$ for all $t\in\bZ_+$, and consider $H\bullet M^a=(H\bullet M^a_t)_{t\in\bN}$ defined as,
\begin{equation}
\label{Def:M_t}
H\bullet M^a_{t}=
\sum_{s=0}^tH_s M^a_s, \ t\in \bN.
\end{equation}
Notice that $H\bullet M^a$ is adapted to the fitration $\bF:=(\cF_t)_{t\in\bN}$, that is $H\bullet M^a_t\in \cF_t$ for all $t\in\bN$. Also $H\bullet M^a_0=0$.
Recall the notation $\|B\|_{\infty}=\max_{b\in B}|b|$. 
\begin{lemma}
\label{Lemma:M_t_is_martingale}
Let $H\bullet M^a=(H\bullet M^a_t)_{t\in\bN}$ be the stochastic process defined in \eqref{Def:M_t}. Then $H\bullet M^a$ is a square integrable Martingale w.r.t. $\bF$ starting from $H\bullet M^a_0=0$. Moreover, the predictable quadratic variation of $H\bullet M^a$, denoted by $\langle H\bullet M^a\rangle= (\langle H\bullet M^a\rangle_t)_{t\in\bN}$, is given by
\begin{equation}
\label{Def:PQV_of_M_t}
\langle H\bullet M^a\rangle_t=
\sum_{s=0}^t H^2_sp\left(a|X_{(s-d):(s-1)}\right)\left(1-p\left(a|X_{(s-d):(s-1)}\right)\right), \ t\in\bN.
\end{equation}
Furthermore, for any $\lambda>0$ and $b>0$ such that $\|B\|_{\infty}\leq b$, the stochastic process
$$
\exp\left(\lambda H\bullet M^a-\frac{e^{\lambda b}-\lambda b-1}{b^2}\langle H\bullet M^a\rangle\right)=\left(\exp\left(\lambda H\bullet M^a_t-\frac{e^{\lambda b}-\lambda b-1}{b^2}\langle H\bullet M^a\rangle_t\right)\right)_{t\in\bN}
$$
is a supermartingale w.r.t. $\bF$ starting from $1$.
\end{lemma}
\begin{proof}
%For $1\leq t\leq d$, we have that $M_t=0$ so that $\bE[M_t|\cF_{t-1}]=0=M_{t-1}$. 
For each $t\in\bZ_+$, we have that $H_t\in \cF_{t-1}$ and also that $\bE\left[1\{X_t=a\}|\cF_{t-1}\right]=p(a|X_{(t-d):(t-1)})$. These two facts imply that for any $t\in\bZ_+$,   
\begin{equation*}
\bE\left[H_tM^a_t|\cF_{t-1} \right]=H_t\bE\left[M^a_t|\cF_{t-1} \right]=0,
\end{equation*}
which, in turn, implies that $E[H\bullet M^a_t|\cF_{t-1}]=H\bullet M^a_{t-1}$. Hence, $H\bullet M^a$ is a martingale w.r.t. to $\bF$. Since $|H\bullet M^a_t|\leq \|B\|_{\infty}t$ for $t\geq 1$, it follows that $H\bullet M^a$ is also square integrable. 

The predictable quadratic variation of $H\bullet M^a$ is defined as
$$
\langle H\bullet M^a\rangle_t=\sum_{s=1}^t\bE\left(\left(M_s-M_{s-1}\right)^2|\cF_{s-1}\right),
$$
for $t\in\bZ_+$ with $\langle H\bullet M^a\rangle_0=0.$ For any $t\in\bZ_+$, one can check that
$$
\left(H\bullet M^a_t-H\bullet M^a_{t-1}\right)^2=H_t^2(1\{X_t=a\}-2p(a|X_{(t-d):(t-1)})1\{X_t=a\}-p^2(a|X_{(t-d):(t-1)})).
$$
Using again that $H_t\in\cF_{t-1}$ and also that $\bE\left[1\{X_t=a\}|\cF_{t-1}\right]=p(a|X_{(t-d):(t-1)})$, one then deduces that for any $t\in\bZ_+$,
$$
\bE\left(\left(H\bullet M^a_t-H\bullet M^a_{t-1}\right)^2|\cF_{t-1}\right)=H_t^2p(a|X_{(t-d):(t-1)})(1-p(a|X_{(t-d):(t-1)})),
$$
which establishes \eqref{Def:PQV_of_M_t}. 
%Since $H_1=\ldots=H_d=0$, it follows that \eqref{Def:PQV_of_M_t} also holds for $1\leq t\leq d$.
The proof that $\exp\left(\lambda H\bullet M^a -\frac{e^{\lambda b}-\lambda b-1}{b^2}\langle H\bullet M^a\rangle\right)$ is a supermartingale w.r.t $\bF$ can be found in \citep{Raginsky_Sason_14}.
\end{proof}

%\begin{remark}
% \gocom{This lemma has been used both in Patricia's paper (in a continuous-time framework) and in Roberto's paper with $\varphi=1_{\{x\}}$.}  
% \end{remark} 

We will use Lemma \ref{Lemma:M_t_is_martingale} to prove the following concentration inequality. 
\begin{proposition}
\label{Concentration_ineq_martingale_EFI}
Let $H\bullet M^a=(H\bullet M^a_t)_{t\in\bN}$ be the stochastic process defined in \eqref{Def:M_t}.
Suppose that $\|B\|_{\infty}\leq b$ for some $b>0$. 
For any fixed $\alpha>0$ and $v>0$, we have for $t\in\bN$,
$$
\bP\left(H\bullet M^a_t\geq \sqrt{2v\alpha}+\frac{\alpha b}{3},\langle H\bullet M^a \rangle_t\leq v\right)\leq \exp\left(-\alpha\right)\bP(\langle H\bullet M^a \rangle_t>0).
$$
\end{proposition}

\begin{remark}
This is basically  Lemma 5 of \citep{Imbuzeiro:15} (see the Economical Freedman's inequality provided in Inequality (41)) applied to the square integrable martingale $H\bullet M^a$.  The only difference is the factor 2 in front of the linear term $\frac{\alpha b}{3}$ which is not present here. Notice that for $t\in\bZ_+$, the concentration inequality above can be rewritten in the following form:
$$
\bP\left(H\bullet M^a_t\geq \sqrt{2v\alpha}+\frac{\alpha b}{3},\langle H\bullet M^a \rangle_t\leq v|\langle H\bullet M^a \rangle_t>0\right)\leq \exp\left(-\alpha\right).
$$
The conditioning on event $\{\langle H\bullet M^a\rangle_t>0\}$ reflects the fact that if $\langle H\bullet M^a\rangle_t=0$ almost surely, then $H\bullet M^a_t=H\bullet M^a_0=0$ almost surely as well.

\end{remark}

\begin{proof}
For $t=0$ the result holds trivially.
Now, suppose $t\in\bZ_+$. 
By considering the set $B/b=\{c/b:c\in B\}$ instead of $B$, it suffices to prove the case $b=1$.
To shorten the notation, we denote $M_t=H\bullet M^a_t$ in the sequel.
By the Markov property, we have that for any $\lambda>0$,
$$
\bP(\lambda M_t-\psi(\lambda)\langle M \rangle_t\geq \alpha)\leq \exp(-\alpha) \bE\left[ \exp\left(\lambda M_t-\psi(\lambda)\langle M\rangle_t\right)1\{\langle M\rangle_t>0\}\right],
$$
where $\psi(\lambda)=e^{\lambda }-\lambda-1$ and
we have used that if $\langle M\rangle_t=0$ almost surely, then $M_t=M_0=0$ almost surely.
By using the fact that $\left(\exp\left(\lambda M_t-\psi(\lambda)\langle M\rangle_t\right)\right)_{t\in\bN}$ is a supermartingale (Lemma \ref{Lemma:M_t_is_martingale} with $b=1$) together with the decomposition $$\{\langle M\rangle_t>0\}=\bigcup_{k=1}^t\{\langle M\rangle_k>0 \ \text{and} \ \langle M\rangle_j=0 \ \text{for all} \ j<k  \},$$
as in \citep{Imbuzeiro:15}, we can deduce that
$$
\bE\left[ \exp\left(\lambda M_t-\psi(\lambda)\langle M\rangle_t\right)1\{\langle M\rangle_t>0\}\right]\leq \bP(\langle M\rangle_t>0),
$$
which implies not only that for any $\lambda>0$,
\begin{equation}
\label{Key_ineq_proof_mart_concentration_ineq}
\bP\left(M_t\geq \frac{\psi(\lambda)}{\lambda }\langle M\rangle_t+\frac{\alpha}{\lambda} \right)\leq \exp(-\alpha)\bP(\langle M\rangle_t>0),
\end{equation}
but also that
$$
\bP\left(M_t\geq \frac{\psi(\lambda)}{\lambda }v+\frac{\alpha}{\lambda}, \langle M\rangle_t\leq v \right)\leq \exp(-\alpha)\bP(\langle M\rangle_t>0).
$$

Now, we use that for $\lambda\in (0,3)$ it holds that $\psi(\lambda)\leq \lambda^2(1-\lambda/3)^{-1}/2$. Hence,  from the above inequalities we deduce that for any $\lambda\in (0,3)$,
\begin{equation}
\label{Key_ineq_2_proof_mart_concentration_ineq}
 \bP\left(M_t\geq \frac{\lambda}{2 (1-\lambda/3)}\langle M\rangle_t+\frac{\alpha}{\lambda} \right)\leq \exp(-\alpha)\bP(\langle M\rangle_t>0),
\end{equation}
and also that
$$
\bP\left(M_t\geq \frac{\lambda}{2 (1-\lambda/3)}v+\frac{\alpha}{\lambda}, \langle M\rangle_t\leq v \right)\leq \exp(-\alpha)\bP(\langle M\rangle_t>0).
$$
Minimizing $\lambda\in (0,3)\mapsto \frac{\lambda}{(1-\lambda/3)}v+\frac{\alpha}{\lambda}$, the result follows.
\end{proof}
By using a peeling argument as in \citep{Hanses_Bouret_Rivoirard_15}, we deduce from the above result the following.
\begin{proposition}
\label{EFI_sencond_half}
Let $H\bullet M^a=(H\bullet M^a_t)_{t\in\bN}$ be the stochastic process defined in \eqref{Def:M_t}.
 Suppose that $\|B\|_{\infty}\leq b$ for some $b>0$. 
For $\epsilon>0$, $v>w>0$ and $\alpha>0$,
we have for $t\in\bN$,
\begin{multline*}
\bP\left(H\bullet M^a_t\geq \sqrt{2\alpha(1+\epsilon)\langle H\bullet M^a \rangle_t}+\frac{\alpha b}{3},w\leq \langle H\bullet M^a \rangle_t\leq v\right)\leq \\
\left\lceil\frac{\log(v/w+1)}{\log(1+\varepsilon)}\right\rceil
\exp\left(-\alpha\right)\bP(\langle H\bullet M^a \rangle_t>0).
\end{multline*}
\end{proposition}
\begin{proof}
It suffices to prove the case $b=1$. The general case follows from this one by first replacing $B$ by $B/b=\{c/b:c\in B\}$ and then rearranging the terms properly. Let us denote $v_0=w$ and $v_k=(1+\varepsilon)v_{k-1}$ for $1\leq k\leq K:=\left\lceil\frac{\log(v/w+1)}{\log(1+\varepsilon)}\right\rceil.$ Notice that $v_{K}\geq v$, by the definition of $K$. To shorten the notation, we denote $H\bullet M^a_t=M_t$ in what follows.

Starting from \eqref{Key_ineq_2_proof_mart_concentration_ineq}, one can deduce that for any $0\leq k< K$ and $\lambda\in (0,3),$ we have
$$
\bP\left(M_t\geq \frac{\lambda}{2(1-\lambda/3)}\langle M\rangle_t+\frac{\alpha}{\lambda}, v_k\leq \langle M\rangle_t\leq v_{k+1}  \right)\leq \exp(-\alpha)\bP(\langle M\rangle_t>0),
$$
which, in turn, implies that 
$$
\bP\left(M_t\geq \frac{\lambda}{2(1-\lambda/3)} v_{k+1}+\frac{\alpha}{\lambda}, v_k\leq \langle M\rangle_t\leq v_{k+1}  \right)\leq \exp(-\alpha)\bP(\langle M\rangle_t>0).
$$
Minimizing w.r.t. to $\lambda\in (0,3)$ as in Proposition \ref{Concentration_ineq_martingale_EFI}, it then follows that
$$
\bP\left(M_t\geq \sqrt{2v_{k+1}\alpha}+\frac{\alpha }{3},v_{k}\leq \langle M \rangle_t\leq v_{k+1}\right)\leq \exp\left(-\alpha\right)\bP(\langle M \rangle_t>0).
$$
Now, on the event $\{v_{k}\leq \langle M \rangle_t\}$, we have that $\langle M \rangle_t (1+\varepsilon)\geq (1+\varepsilon)v_k=v_{k+1}$, so that the inequality above implies 
$$
\bP\left(M_t\geq \sqrt{2(1+\varepsilon)\langle M \rangle_t \alpha}+\frac{\alpha }{3},v_{k}\leq \langle M \rangle_t\leq v_{k+1}\right)\leq \exp\left(-\alpha\right)\bP(\langle M \rangle_t>0).
$$
Summing over $k$ the result follows (recall that $v\leq v_{K}$ by the choice of $K$).  
\end{proof}

Hereafter, let $m\in\N$ and consider a function $\varphi:A^{\llb 1, m\rrb}\times A^{\llb -d, -1\rrb}$ such that its supremum norm $\|\varphi\|_{\infty}=\max_{(x_{1:m},x_{-d:-1})\in A^{\llb 1, m\rrb}\times A^{\llb -d, -1\rrb}}|\varphi(x_{1:m},x_{-d:-1})|\leq b$. Here we use the convention that $\varphi$ is a function defined only on $A^{\llb -d, -1\rrb}$ when $m=0$. Given such a function $\varphi$, let us denote $H^{\varphi}=(H^{\varphi}_{t})_{t\geq 0}$ the stochastic process defined as $H^{\varphi}_0=\ldots=H^{\varphi}_{m+d}=0$ and $H^{\varphi}_t=\varphi(X_{1:m},X_{(t-d):(t-1)})$ for $t\geq d+m+1.$

Clearly, $H^{\varphi}_t\in \cF_{t-1}$ for all $t\in \bZ_+$.
From \eqref{Def:PQV_of_M_t}, one can check that the predictable quadratic variation $\langle H^{\varphi}\bullet M^{a} \rangle$ of the martingale $H^{\varphi}\bullet M^{a}$ is given by $\langle H^{\varphi}\bullet M^{a} \rangle_0=\ldots=\langle H^{\varphi}\bullet M^{a} \rangle_{m+d}=0$ and for $t\geq m+d+1,$  
\begin{equation}
\label{Def:PQV_of_M_varphi_t}
\langle H^{\varphi}\bullet M^{a} \rangle_t=
\sum_{s=m+d+1}^t \varphi^2(X_{1:m},X_{(s-d):(s-1)})p\left(a|X_{(s-d):(s-1)}\right)\left(1-p\left(a|X_{(s-d):(s-1)}\right)\right).
\end{equation}

As a direct consequence of Proposition \ref{EFI_sencond_half}, we derive the following result.

\begin{corollary}
\label{prop:Berstein_type_ineq_empirical_proba}
Let $X_{1:n}$ be a sample from a MTD model of order $d$ with set of relevant lags $\Lambda$. Let $\hat{\Lambda}_m$ be an estimator of $\Lambda$ computed from $X_{1:m}$, where $n>m$.  For each $x\in A^{\llb -d, -1\rrb}$, $a\in A$ and $S\subseteq \llb -d, -1\rrb$, let $\hat{p}_{m,n}(a|x_{S})$ be the empirical transition probability  defined in \eqref{def:emp_trans_proba} computed from $X_{m+1:n}$. Then for any $S\subseteq \llb -d, -1\rrb$ such that $\Lambda\subseteq S$, $\varepsilon>0$, $\alpha>0$ and $n\geq m+d+1$, we have
\begin{multline}
\label{Ineq:Berstein_type_ineq_empirical_proba}
  \bP\left(\hat{\Lambda}_m=S, |\hat{p}_{m,n}(a|x_{S})-p(a|x_{\Lambda})|\geq  \sqrt{\frac{2\alpha(1+\varepsilon) p(a|x_{\Lambda})(1-p(a|x_{\Lambda}))}{\bar{N}_{m,n}(x_{S})}}\right.  \\ \left.
+\frac{\alpha}{3\bar{N}_{m,n}(x_{S})}\right)  \leq 2\left\lceil \frac{\log(n-m-d+1)}{\log(1+\varepsilon)}\right\rceil e^{-\alpha}\bP\left(\bar{N}_{m,n}(x_{S})>0, \hat{\Lambda}_m=S\right).   
\end{multline}
In particular, 
\begin{multline}
\label{Ineq:Berstein_type_ineq_empirical_proba_Hat_Lambda_contains_Lambda}
  \bP\left(\Lambda\subseteq\hat{\Lambda}_m, |\hat{p}_{m,n}(a|x_{\hat{\Lambda}_m})-p(a|x_{\Lambda})|\geq  \sqrt{\frac{2\alpha(1+\varepsilon) p(a|x_{\Lambda})(1-p(a|x_{\Lambda}))}{\bar{N}_{m,n}(x_{\hat{\Lambda}_m})}}\right.  \\ \left.
+\frac{\alpha}{3\bar{N}_{m,n}(x_{\hat{\Lambda}_m})}\right)  \leq 2\left\lceil \frac{\log(n-m-d+1)}{\log(1+\varepsilon)}\right\rceil e^{-\alpha}\bP\left(\bar{N}_{m,n}(x_{\hat{\Lambda}_m})>0, \Lambda \subseteq\hat{\Lambda}_m\right).   
\end{multline}
\end{corollary}
\begin{proof}
Summing in both sides of \eqref{Ineq:Berstein_type_ineq_empirical_proba} over $S\subseteq \llb -d, -1\rrb$ such that $\Lambda\subseteq S$, we obtain inequality \eqref{Ineq:Berstein_type_ineq_empirical_proba_Hat_Lambda_contains_Lambda}.
Thus, it remains to prove \eqref{Ineq:Berstein_type_ineq_empirical_proba}. To that end, 
take $\varphi(X_{1:m},X_{(t-d):(t-1)})=1\{\hat{\Lambda}_m=S\}1\{X_{t+j}=x_{j},j\in S\}$
and notice that in this case 
$$\langle H^{\varphi}\bullet M^a \rangle_{n}=1\{\hat{\Lambda}_m=S\}\bar{N}_{m,n}(x_{S})p(a|x_{\Lambda})(1-p(a|x_{\Lambda})).$$
So, if either $p(a|x_{\Lambda})=0$ or $p(a|x_{\Lambda})=1$, then we have necessarily $\langle H^{\varphi}\bullet M^a \rangle_{n}=0$ for all $n\geq m+d+1$, which implies that almost surely for all $n\geq m+ d+1$,
$$H^{\varphi}\bullet M^a_{n}=1\{\hat{\Lambda}_m=S\}(\bar{N}_{m,n}(x_{S}, a)-\bar{N}_{m,n}(x_{S})p(a|x_{\Lambda}))=0.$$
By noticing that $H^{\varphi}\bullet M^a_{n}=1\{\hat{\Lambda}_m=S\}\bar{N}_{m,n}(x_{S})(\hat{p}_{m,n}(a|x_{S})-p(a|x_{\Lambda}))$, it follows that, on the event $\{\hat{\Lambda}_m=S,\bar{N}_{m,n}(x_{S})>0\}$, we must have $\hat{p}_{m,n}(a|x_{S})=p(a|x_{\Lambda})$ almost surely and so the left-hand side of \eqref{Ineq:Berstein_type_ineq_empirical_proba} is 0 and the result holds trivially.

Let us now suppose $0<p(a|x_{\Lambda})<1$. In this case, we apply Proposition \ref{EFI_sencond_half} with $w=p(a|x_{\Lambda})(1-p(a|x_{\Lambda}))$, $v=(n-m-d)w$ and $b=1$ to deduce that,
\begin{multline}
\bP\left(\hat{\Lambda}_m=S,N^{*}_{n,m}(x_{S})(\hat{p}_{m,n}(a|x_{S})-p(a|x_{\Lambda}))\geq \right.  \\ \left. \sqrt{2\alpha(1+\varepsilon) p(a|x_{\Lambda})(1-p(a|x_{\Lambda}))\bar{N}_{m,n}(x_{S})}+
\frac{\alpha}{3}, \bar{N}_{m,n}(x_{S})>0\right) \\ \leq \left\lceil \frac{\log(n-m-d+1)}{\log(1+\varepsilon)}\right\rceil e^{-\alpha}\bP(1\{\hat{\Lambda}_m=S\}\bar{N}_{m,n}(x_{S})p(a|x_{\Lambda})(1-p(a|x_{\Lambda}))>0).   
    \end{multline}
To conclude the proof, observe that in this case 
\[
\{1\{\hat{\Lambda}_m=S\}\bar{N}_{m,n}(x_{S})p(a|x_{\Lambda})(1-p(a|x_{\Lambda}))>0\}=\{\hat{\Lambda}_m=S,\bar{N}_{m,n}(x_{S})>0\},
\] 
and use again Proposition  \ref{EFI_sencond_half} with $H^{-\varphi}\bullet M^a$ in the place of $H^{\varphi}\bullet M^a$ (by noting also that $\langle H^{\varphi}\bullet M^a \rangle = \langle H^{-\varphi}\bullet M^a \rangle$).
\end{proof}

\begin{remark}
Let us briefly comment on the results of Corollary \ref{prop:Berstein_type_ineq_empirical_proba}. Suppose $x\in A^{\llb -d, -1\rrb}$ and $a\in A$ are such that $0<p(a|x_{\Lambda})<1$ and also that $\hat{\Lambda}_m$ is a consistent estimator of $\Lambda$.
By the CLT for aperiodic and irreducible Markov Chains it follows that
$\sqrt{\bar{N}_{m,n}(x_{\Lambda})}(\hat{p}_{m,n}(a|x_{\Lambda})-p(a|x_{\Lambda}))1\{\hat{\Lambda}_{m}=\Lambda,\bar{N}_{m,n}(x_{\Lambda})>0\}$ converges in distribution (as $\min\{m,n\}\to\infty$) to a centered Gaussian random variable with variance $p(a|x_{\Lambda})(1-p(a|x_{\Lambda}))$.
%for any $S\subseteq \llb 1, d\rrb$ such that $\Lambda\subseteq S$.
This implies that for sufficiently large $n$, 
\begin{multline*}
\bP\left(\hat{p}_{m,n}(a|x_{\Lambda})-p(a|x_{\Lambda}) \right. \\ \left. \geq \sqrt{\frac{2\alpha p(a|x_{\Lambda})(1-p(a|x_{\Lambda}))}{\bar{N}_{m,n}(x_{\Lambda})}}|\hat{\Lambda}_{m}=\Lambda,\bar{N}_{m,n}(x_{\Lambda})>0\right)\leq  e^{-\alpha}.
\end{multline*}
Let us compare this heuristic argument with Corollary \ref{prop:Berstein_type_ineq_empirical_proba} applied to $S=\Lambda.$ In this case, in Inequality \eqref{Ineq:Berstein_type_ineq_empirical_proba}, the variance term $\sqrt{\frac{2\alpha(1+\varepsilon) p(a|x_{\Lambda})(1-p(a|x_{\Lambda}))}{\bar{N}_{m,n}(x_{\Lambda})}}$ can be made arbitrarily close to optimal value $\sqrt{\frac{2\alpha p(a|x_{\Lambda})(1-p(a|x_{\Lambda}))}{\bar{N}_{m,n}(x_{\Lambda})}}$, at the cost $\frac{1}{\log(1+\varepsilon)}$.  
Both the linear term $\frac{\alpha}{\bar{N}_{m,n}(x_{\Lambda})}$ and the $\log(n-m-d+1)$ factor are the price to pay to achieve the result which holds every $n\geq m+d+1$ and reflect the fact that $\hat{p}_{m,n}(a|x_{\Lambda})-p(a|x_{\Lambda})$ is Gaussian only asymptotically.

In particular, Corollary \ref{prop:Berstein_type_ineq_empirical_proba} improves the Economical Freedman's Inequality (as stated in \citep{Imbuzeiro:15} - Lemma 5, Inequality (42)) when restricted to the martingale $H^{\varphi}\bullet M^a.$
\end{remark}

%[\gocom{Keeping $w>0$ in the statement is convenient to prove the next result}]

In the sequel, let us denote $P^a=(P^a_t)_{t\in\bN }$, for each $a \in A$, the stochastic process defined as $P^a_t=p(a|X_{(t-d:t-1)})$ for each $t\in\bN.$ With this notation, notice that
\begin{equation}
\label{Def:upper_bound_PQV_of_Mt}
 H^2\bullet P^a_t=
 \sum_{s=0}^tH_s^2p(a|X_{(s-1):(s-d)}), \  t\in\bN,
\end{equation}
is such that $\langle H\bullet M^a \rangle_t\leq H^2\bullet P^a_t$ for all $t\in\bN$. In particular, Proposition \ref{Concentration_ineq_martingale_EFI} holds if we replace $\langle H\bullet M^a \rangle_t$ by $H^2\bullet P^a_t.$ A closer inspection of the proof of Proposition \ref{EFI_sencond_half} reveals that 
this proposition also holds with $H^2\bullet P^a_t$ in the place of $\langle H\bullet M^a \rangle_t$. In the next theorem, we show that we can replace $H^2\bullet P^a_t$ by a linear transformation of its empirical version which is crucial for our analysis.

\begin{theorem}
\label{Thm:MCI_empirical_quadratic_variation}
Let $H\bullet M^a=(H\bullet M^a_t)_{t\in\bN}$ be the stochastic process defined in \eqref{Def:M_t}.
 Suppose that $\|B\|_{\infty}\leq b$ for some $b>0$. 
For any fixed $\mu\in (0,3)$ satisfying $\mu>\psi(\mu)=\exp(\mu)-\mu-1$ and $\alpha>0$, define for $t\in\bN$,
$$
H^2\bullet \hat{P}^{a}_t=\frac{\mu}{\mu-\psi(\mu)}\sum_{s=0}^{t}H_s^2\hat{P}^a_s+\frac{b^2 \alpha}{\mu-\psi(\mu)},
$$
where $\hat{P}^a_s=1\{X_s=a\}$ for all $s\in\bN.$
Then, for any fixed $\epsilon>0$ and $v>w>0$, we have for any $t\in\bN$,
\begin{multline*}
\bP\left(H\bullet M^a_t\geq \sqrt{2(1+\epsilon)\alpha H^2\bullet \hat{P}^a_t}+\frac{\alpha b}{3},w\leq H^2\bullet \hat{P}^a_t\leq v\right)    \\ \leq 2 \left\lceil\frac{\log(v/w+1)}{\log(1+\varepsilon)}\right\rceil
\exp\left(-\alpha\right) \bP(\langle H\bullet M^a\rangle_t>0).
\end{multline*}
\end{theorem}
\begin{proof}
We prove only the case $b=1$.
%During the proof, we shall use also the following notation:
%$$
%\tilde{V}_t=\sum_{s=0}^tH^2_s1\{X_s=a\}, \ t\in\bN.
%$$
Note that $-H^{2}\bullet M^{a}_t=H^2\bullet P^a_t-H^2\bullet \hat{P}^a_t.$ Also recall that $\langle H\bullet M^{a} \rangle_t\leq H^2\bullet P^a_t$. 

 We now proceed to the proof. We first use Inequality \eqref{Key_ineq_proof_mart_concentration_ineq} for the martingale $-H^{2}\bullet M^{a}$ together with that fact that $\langle -H^{2}\bullet M^{a} \rangle_t\leq H^{4}\bullet P^{a}_t\leq H^{2}\bullet P^{a}_t$ (the last inequality holds because $b=1$) to deduce that for any $\mu>0$,
$$
\bP\left(H^2\bullet P^a_t\geq H^2\bullet \hat{P}^a_t+\frac{\psi(\mu)}{\mu}H^2\bullet P^a_t+\frac{\alpha}{\mu} \right)\leq \exp(-\alpha)\bP(\langle H\bullet M^a \rangle_t>0),
$$
which implies that for any $\mu \in (0,3)$ satisfying $\mu>\psi(\mu),$ it holds
$$
\bP\left(H^2\bullet P^a_t\geq H^2\bullet \hat{P}^{a}_t\right)\leq \exp(-\alpha)\bP(\langle H\bullet M^a \rangle_t>0).
$$
Hence, combining this inequality with \eqref{Key_ineq_2_proof_mart_concentration_ineq}, we conclude that for any $\lambda\in (0,3)$ and any $\mu \in (0,3)$ satisfying $\mu>\psi(\mu),$
\begin{multline*}
\bP\left(H\bullet M^{a}_t\geq \frac{\lambda}{2 (1-\lambda/3)}H^2\bullet \hat{P}^{a}_t+\frac{\alpha}{\lambda} \right)\leq \\
\bP\left(H\bullet M^{a}_t\geq \frac{\lambda}{2 (1-\lambda/3)}H^2\bullet \hat{P}^{a}_t+\frac{\alpha}{\mu}, H^2\bullet P^a_t\leq H^2\bullet \hat{P}^{a}_t \right)\\+ \bP\left(H^2\bullet P^a_t\geq H^2\bullet \hat{P}^{a}_t \right)\leq 2\exp(-\alpha)\bP(\langle H\bullet M^a \rangle_t>0),
\end{multline*}
where we also used in the last inequality the fact that $\langle H\bullet M^a \rangle_t\leq H^2\bullet P^a_t$.

To conclude the proof, we need to follow the same steps as Proposition \ref{Concentration_ineq_martingale_EFI} and then use the peeling argument as in the proof of Proposition \ref{EFI_sencond_half} with $H^2\bullet \hat{P}^{a}_t$ in the place of $\langle H\bullet M^a \rangle_t$.
\end{proof}

%Let us define for each $\varphi:A^{-\llb 1, d\rrb}\to \bR$ and $a\in A$
%\begin{equation}
%\label{Def:upper_bound_PQV_of_Mt}
% V^{\varphi,a}_t=
% \begin{cases}
% 0, \ \text{for} \ 0\leq t\leq d\\
% \sum_{s=d+1}^t\varphi^2(X^{s-1}_{s-d})p(a|X^{s-1}_{s-d}),  \ \text{for} \  t\geq d+1
% \end{cases}.
%\end{equation}

As consequence of Theorem \ref{Thm:MCI_empirical_quadratic_variation}, we obtain the following result.

\begin{proposition}
\label{prop:bernstein_empirical_trans_prob_correct_var}
Let $X_{1:n}$ be a sample from a MTD model of order $d$ with set of relevant lags $\Lambda$. Let $\hat{\Lambda}_m$ be an estimator of $\Lambda$ computed from $X_{1:m}$ where $n>m$.  For any $x\in A^{\llb -d, -1\rrb}$, $a\in A$ and $S\subseteq \llb -d, -1\rrb$, let $\hat{p}_{m,n}(a|x_{S})$ be the empirical transition probability  defined in \eqref{def:emp_trans_proba} computed from $X_{m+1:n}$, and consider for $\alpha>0$ and $\mu\in (0,3)$ satisfying $\mu>\psi(\mu)=\exp(\mu)-\mu-1$, 
$$
\hat{V}_{m,n}(a,x,S)=\frac{\mu}{\mu-\psi(\mu)}\hat{p}_{m,n}(a|x_{S})+\frac{\alpha}{\mu-\psi(\mu)}\frac{1}{\bar{N}_{m,n}(x_{S})}.
$$
Then for any $S\subseteq \llb -d, -1\rrb$ such that $\Lambda\subseteq S$ and $n\geq m+d+1$, we have
\begin{multline}
\label{Ineq:concentration_ineq_empirical_threshold_adapative}
\bP\left(\hat{\Lambda}_m=S,|\hat{p}_{m,n}(a|x_{S})-p(a|x_{\Lambda})|\geq \sqrt{\frac{2\alpha(1+\epsilon)\hat{V}_{m,n}(a,x,S)}{\bar{N}_{m,n}(x_{S})}}+\frac{\alpha}{3\bar{N}_{m,n}(x_{S})}\right)\\ \leq 4\left\lceil\frac{\log(\mu (n-m-d)/\alpha+2)}{\log(1+\varepsilon)}\right\rceil e^{-\alpha}\bP\left(\hat{\Lambda}_m=S,\bar{N}_{m,n}(x_{S})>0\right). \end{multline}
In particular, 
\begin{multline}
\label{Ineq:concentration_ineq_empirical_threshold__adapative_Hat_Lambda_contains_Lambda}
\bP\left(\Lambda\subseteq\hat{\Lambda}_m,|\hat{p}_{m,n}(a|x_{\hat{\Lambda}_m})-p(a|x_{\Lambda})|\geq \sqrt{\frac{2\alpha(1+\epsilon)\hat{V}_{m,n}(a,x,\hat{\Lambda}_m)}{\bar{N}_{m,n}(x_{\hat{\Lambda}_m})}}+\frac{\alpha}{3\bar{N}_{m,n}(x_{\hat{\Lambda}_m})}\right)\\ \leq 4\left\lceil\frac{\log(\mu (n-m-d)/\alpha+2)}{\log(1+\varepsilon)}\right\rceil e^{-\alpha}\bP\left(\Lambda\subseteq\hat{\Lambda}_n,\bar{N}_{m,n}(x_{\hat{\Lambda}_m})>0\right). \end{multline}
\end{proposition}
%\gocom{We need to write more generally for $x\in A^{S}$ and $S\subseteq \llb 1, d\rrb$ such that $\Lambda\subset S$}
\begin{proof}
Summing in both sides of \eqref{Ineq:concentration_ineq_empirical_threshold_adapative} over $S\subseteq \llb -d, -1\rrb$ such that $\Lambda\subseteq S$, we obtain inequality \eqref{Ineq:concentration_ineq_empirical_threshold__adapative_Hat_Lambda_contains_Lambda}. Hence, it remains to show \eqref{Ineq:concentration_ineq_empirical_threshold_adapative}.
Arguing as in Corollary \ref{prop:Berstein_type_ineq_empirical_proba}
we need to consider only the case $0<p(a|x)<1$.

By applying Theorem \ref{Thm:MCI_empirical_quadratic_variation} with $H=H^{\pm \varphi}$
where $\varphi$ is as in the proof of Corollary \ref{prop:Berstein_type_ineq_empirical_proba}, $v=\frac{\mu}{\mu-\psi(\mu)}(n-m-d)+\frac{\alpha}{\mu-\psi(\mu)}$ and $w=\frac{\alpha}{\mu-\psi(\mu)}$, we obtain that
\begin{multline*}
\bP\left(\hat{\Lambda}_m=S,\bar{N}_{m,n}(x_{S})|\hat{p}_{m,n}(a|x_{S})-p(a|x_{\Lambda})|\geq \sqrt{2(1+\epsilon)\alpha\tilde{V}_{m,n}(a,x,S)}+\frac{\alpha}{3} \right)\\ \leq 4\left\lceil\frac{\log(\mu (n-m-d)/\alpha+2)}{\log(1+\varepsilon)}\right\rceil e^{-\alpha}\bP(1\{\hat{\Lambda}_m=S\}\bar{N}_{m,n}(x_{S})p(a|x_{\Lambda})(1-p(a|x_{\Lambda}))>0),
\end{multline*}
where $\tilde{V}_{m,n}(a,x,S)=\bar{N}_{m,n}(x_{S})\hat{V}_{m,n}(a,x,S)$.

By using that when $0<p(a|x)<1$,
$$\{1\{\hat{\Lambda}_m=S\}\bar{N}_{m,n}(x_{S})p(a|x_{\Lambda})(1-p(a|x_{\Lambda}))>0\}=\{\hat{\Lambda}_m=S,\bar{N}_{m,n}(x_{S})>0\},$$ %and then noticing that $M_n^{x,a}=\bar{N}_{n}(x)(\hat{p}_n(a|x)-p(a|x))$ 
and the fact that $\tilde{V}_{m,n}(a,x,S)=\bar{N}_{m,n}(x_{S})\hat{V}_{m,n}(a,x,S)$,   
%$\hat{V}^{x,a,\mu,\alpha}_n=\bar{N}_{n}(x)\tilde{V}^{x,a,\mu,\alpha}_n$,
we deduce \eqref{Ineq:concentration_ineq_empirical_threshold_adapative} from the above inequality.
\end{proof}

\vskip 0.2in
\bibliography{ref.bib}

\end{document}